\documentclass[12pt]{amsart}
\usepackage{amssymb}
\usepackage{longtable}
\usepackage{hyperref}
\usepackage{tikz-cd}
\usepackage[total={6.25truein,9truein},centering]{geometry}

\usepackage[toc]{appendix}

\usepackage{xcolor}
\DeclareRobustCommand{\gobblefive}[5]{}
\newcommand*{\SkipTocEntry}{\addtocontents{toc}{\gobblefive}}

\allowdisplaybreaks

\newtheorem{theorem}[equation]{Theorem}
\newtheorem{lemma}[equation]{Lemma}
\newtheorem{proposition}[equation]{Proposition}
\newtheorem{corollary}[equation]{Corollary}

\newtheorem{claim1}{Claim}

\newtheorem{claim3}{Claim}

\theoremstyle{definition}
\theoremstyle{remark}
\newtheorem{remark}[equation]{Remark}

\numberwithin{equation}{subsection}

\newcommand{\FF}{\mathbb{F}}
\newcommand{\ZZ}{\mathbb{Z}}
\newcommand{\QQ}{\mathbb{Q}}

\newcommand{\NN}{\mathbb{N}}
\newcommand{\TT}{\mathbb{T}}
\newcommand{\GG}{\mathbb{G}}

\newcommand{\LL}{\mathbb{L}}

\newcommand{\KK}{\mathbb{K}}

\newcommand{\bbb}{\mathbf{b}}
\newcommand{\bA}{\mathbf{A}}
\newcommand{\bBB}{\mathbf{B}}
\newcommand{\bE}{\mathbf{E}}
\newcommand{\be}{\mathbf{e}}

\newcommand{\bg}{\mathbf{g}}
\newcommand{\bh}{\mathbf{h}}
\newcommand{\bk}{\mathbf{k}}

\newcommand{\bn}{\mathbf{n}}

\newcommand{\bs}{\mathbf{s}}
\newcommand{\bP}{\mathbf{P}}

\newcommand{\bT}{\mathbf{T}}
\newcommand{\bu}{\mathbf{u}}
\newcommand{\bv}{\mathbf{v}}
\newcommand{\bw}{\mathbf{w}}
\newcommand{\bx}{\mathbf{x}}

\newcommand{\bY}{\mathbf{Y}}
\newcommand{\bz}{\mathbf{z}}
\newcommand{\bZ}{\mathbf{Z}}
\newcommand{\bX}{\mathbf{X}}
\newcommand{\bK}{\mathbf{K}}

\newcommand{\bcG}{\mathbf{\mathcal{G}}}

\newcommand{\cC}{\mathcal{C}}

\newcommand{\cM}{\mathcal{M}}
\newcommand{\cN}{\mathcal{N}}

\newcommand{\cP}{\mathcal{P}}
\newcommand{\cR}{\mathcal{R}}
\newcommand{\cT}{\mathcal{T}}
\newcommand{\cG}{\mathcal{G}}

\newcommand{\rB}{\mathrm{B}}
\newcommand{\rE}{\mathrm{E}}
\newcommand{\rF}{\mathrm{F}}
\newcommand{\rJ}{\mathrm{J}}
\newcommand{\rH}{\mathrm{H}}
\newcommand{\rd}{\mathrm{d}}
\newcommand{\rP}{\mathrm{P}}
\newcommand{\rp}{\mathrm{p}}

\newcommand{\rR}{\mathrm{R}}
\newcommand{\rpr}{\mathrm{pr}}

\newcommand{\fp}{\mathfrak{p}}

\newcommand{\bsy}{\boldsymbol{y}}

\newcommand{\bse}{\boldsymbol{e}}

\newcommand{\bsm}{\boldsymbol{m}}

\newcommand{\bss}{\boldsymbol{s}}

\newcommand{\bsu}{\boldsymbol{u}}
\newcommand{\bsw}{\boldsymbol{w}}

\newcommand{\bsx}{\boldsymbol{x}}

\newcommand{\bsz}{\boldsymbol{z}}
\newcommand{\bsD}{\boldsymbol{D}}

\DeclareMathOperator{\Aut}{Aut}
\DeclareMathOperator{\Der}{Der}
\DeclareMathOperator{\diag}{diag}
\DeclareMathOperator{\DR}{DR}
\DeclareMathOperator{\Exp}{Exp}

\DeclareMathOperator{\Ker}{Ker}
\DeclareMathOperator{\GL}{GL}
\DeclareMathOperator{\Lie}{Lie}
\DeclareMathOperator{\Mat}{Mat}
\DeclareMathOperator{\Cent}{Cent}
\DeclareMathOperator{\End}{End}

\DeclareMathOperator{\Span}{Span}
\DeclareMathOperator{\Gal}{Gal}
\DeclareMathOperator{\Hom}{Hom}

\DeclareMathOperator{\rank}{rank}

\DeclareMathOperator{\Rep}{\mathbf{Rep}}
\DeclareMathOperator{\Ext}{\mathbf{Ext}}

\DeclareMathOperator{\vect}{\mathbf{vec}}

\DeclareMathOperator{\im}{Im}
\DeclareMathOperator{\trdeg}{tr.deg}

\DeclareMathOperator{\Iden}{Id}

\newcommand{\ok}{\overline{k}}
\newcommand{\oK}{\mkern2.5mu\overline{\mkern-2.5mu K}}

\newcommand{\inn}{\mathrm{in}}
\newcommand{\per}{\mathrm{per}}
\newcommand{\si}{\mathrm{si}}

\newcommand{\sep}{\mathrm{sep}}

\newcommand{\power}[2]{{#1 [\![ #2 ]\!]}}
\newcommand{\laurent}[2]{{#1 (\!( #2 )\!)}}

\newcommand{\inorm}[1]{{\lvert #1 \rvert}_{\infty}}

\newcommand{\pd}{\partial}

\newcommand{\tr}{{\mathsf{T}}}

\newenvironment{psmallmatrix}
{\left(\begin{smallmatrix}}
	{\end{smallmatrix}\right)}

\begin{document}
	
	\title[HYPERDERIVATIVES OF PERIODS AND LOGARITHMS]{ALGEBRAIC RELATIONS AMONG HYPERDERIVATIVES OF PERIODS AND LOGARITHMS OF DRINFELD MODULES}
	
	\author{Changningphaabi Namoijam}
	\address{Department of Mathematics, Colby College, Waterville, Maine 04901, USA}
	\email{cnamoijam@gmail.com}

	\thanks{This project was supported by MoST Grant 110-2811-M-007-517 and  partially by NSF Grant DMS-1501362}
	
	
	\date{June 9, 2024}
	
	\begin{abstract}
		We determine all algebraic relations among all hyperderivatives of the periods, quasi-periods, logarithms, and quasi-logarithms of Drinfeld modules defined over a separable closure of the rational function field. In particular, for periods and logarithms that are linearly independent over the endomorphism ring of the Drinfeld module, we prove the algebraic independence of their hyperderivatives and the hyperderivatives of the corresponding quasi-periods and quasi-logarithms. 
	\end{abstract}
	
	
	\maketitle
	
	
	\section{Introduction} \label{S:Intro}
	
The objects of study in the present paper are inspired by elliptic curves in the classical setting. Let $E$ be an elliptic curve defined over $\overline{\QQ}$. The period conjecture states that the transcendence degree over $\overline{\QQ}$ of the two periods $\{\omega_1, \omega_2\}$ and the two corresponding quasi-periods $\{\eta_1, \eta_2\}$ of $E$ is $2$ when $E$ has complex multiplication (CM), and $4$ otherwise. The CM case was confirmed to be true by Chudnovsky, while the non-CM case is still open. With regards to logarithms of $E$, one can expect logarithms of algebraic numbers that are linearly independent over $\End(E)$ to be algebraically independent over $\overline{\QQ}$. Although linear independence over $\overline{\QQ}$ of  these logarithms is known due to Masser (for the CM case), Bertrand-Masser (for the non-CM case), and as a consequence of W\"{u}stholz's analytic subgroup theorem, algebraic independence of these logarithms
is still fully open. See \cite{BakWus07}, \cite{Wald} for details. 

In the function field setting, Drinfeld \cite{Drinfeld} introduced ``elliptic modules," now called Drinfeld modules, as an analogue of elliptic curves. Later, Anderson \cite{And86} defined higher dimensional generalizations of Drinfeld modules, called $t$-modules. One can ask analogous questions regarding algebraic independence of periods, quasi-periods, logarithms, and quasi-logarithms of Drinfeld modules and Anderson $t$-modules defined over algebraic function fields. Yu proved the sub-$t$-module theorem \cite{Yu97}, a remarkable result regarding linear independence among logarithms of Anderson $t$-modules, which is an analogue of W\"{u}stholz's analytic subgroup theorem, and proved the complete transcendence results concerning periods and logarithms of Drinfeld modules \cite{Yu86}, \cite{Yu90}. Thiery \cite{Thiery} proved algebraic independence results among periods and quasi-periods of rank $2$ Drinfeld modules with complex multiplication. Chang and Papanikolas \cite{CP11}, \cite{CP12} proved algebraic independence of periods, quasi-periods, logarithms, and quasi-logarithms of Drinfeld modules of arbitrary rank. The goal of the present paper is to generalize completely under separability hypothesis this work of Chang and Papanikolas \cite{CP11}, \cite[Thm.~3.5.4, Thm.~5.1.5, and Cor.~5.1.6]{CP12} to include all hyperderivatives, which are defined below.

Let $\mathbb{F}_q$ be a finite field, where $q$ is a positive power of a prime number $p$, and let $\theta$ be an indeterminate. For the rational function field $\FF_q(\theta)$, the $j$-th hyperderivative $\pd_\theta^j: \FF_q(\theta) \rightarrow \FF_q(\theta)$ is defined by ${\pd_\theta^j(\theta^m) := \binom{m}{j} \theta^{m-j}}$, where $j \geq 0$. Taking the completion $\laurent{\FF_q}{1/\theta}$ of $\FF_q(\theta)$ with respect to its $\infty$-adic absolute value $\inorm{\,\cdot\,}$, $\smash{\pd_\theta^j(\cdot)}$ extends uniquely to $\laurent{\FF_q}{1/\theta}^\sep$. Note that hyperderivatives play the analogous role of formal derivatives in the classical case.  Unlike in the classical setting of elliptic curves, one can take hyperderivatives of periods and logarithms of Anderson $t$-modules defined over $\FF_q(\theta)^{\sep}$. Moreover, many interpretations of objects of interest in terms of logarithms of Anderson $t$-modules involve hyperderivatives. The entries of periods of the $d$-th tensor power $\mathfrak{C}^{\otimes d}$ of the Carlitz module $\mathfrak{C}$ (rank $1$ Drinfeld module) are obtained using hyperderivatives \cite[Lem.~8.3]{Maurischat18} of Anderson-Thakur functions \cite[$\S$2.5]{AndThak90}. Also, Carlitz zeta values \cite{Thakur} appear in the last coordinate of a logarithm of $\mathfrak{C}^{\otimes d}$ \cite[Thm.~3.8.3]{AndThak90}. Generalizing this, Chang, Green, and Mishiba \cite{ChangGreenMishiba19} showed that multizeta values \cite{Thakur} also appear as coordinates of logarithms of a particular Anderson $t$-module and further showed that its periods and logarithms are obtained using hyperderivatives. There are also logarithmic interpretations of special values of Goss $L$-functions attached to Drinfeld modules in terms of logarithms of an Anderson $t$-module, where hyperderivatives play a crucial role \cite{GezmisN21}. These interpretations further motivate interest in determining algebraic independence of hyperderivatives of periods and logarithms of Anderson $t$-modules.

Algebraic independence among hyperderivatives of the fundamental period of the Carlitz module were proved by Denis \cite{Denis93}, \cite{Denis95}, \cite{Denis00} and Maurischat \cite{Maurischat18}, \cite{Maurischat21}. Further work in this direction was also done in unpublished work by Brownawell and van der Poorten. Utilizing Yu's sub-$t$-module theorem, Brownawell and Denis \cite{Brownawell99}, \cite{BrownawellDenis00}, and Brownawell \cite{Brownawell01} investigated linear independence of hyperderivatives of logarithms and quasi-logarithms of Drinfeld modules. In the present paper, we determine all algebraic independence results among all hyperderivatives of periods, quasi-periods, logarithms, and quasi-logarithms of Drinfeld modules of arbitrary rank under the hypothesis of separability. 

	\subsection{Hyperderivatives of Periods and Logarithms}\label{S:IntroHyperPer}

	For a finite field $\mathbb{F}_q$, where $q$ is a positive power of a prime number $p$, we set $A:=\mathbb{F}_q[\theta]$, $k:=\mathbb{F}_q(\theta)$ and $k_{\infty}:=\mathbb{F}_q((1/\theta))$, the completion of $k$ at its infinite place. We further set $\KK$ to be the completion of an algebraic closure of $k_{\infty}$, and let $\ok$ and $k^\sep$ be the algebraic closure and the separable closure  respectively of $k$ inside $\KK$. For a variable $t$ independent from $\theta$, we further define $\bA:= \mathbb{F}_q[t]$ and $\bk:=\mathbb{F}_q(t)$.\par
	
	For $n \in \ZZ$, we define the \emph{Frobenius twist} $\tau^n:\laurent{\KK}{t} \to \laurent{\KK}{t}$ by setting for $f = \sum_ia_it^i$,
	\begin{equation}\label{E:Frob}
		\tau^n(f):= f^{(n)} = \sum_ia_i^{q^n}t^i. 
	\end{equation}
	For a field $K \subseteq \KK$, we define the twisted power series ring $\power{K}{\tau}$ subject to the condition $\tau c = c^q\tau$ for all $c\in K$. Then, we define the twisted polynomial ring $K[\tau]$ as the subring of $\power{K}{\tau}$, where $K[\tau]$ is viewed as a subalgebra of the $\FF_q$-linear endomorphisms of the additive group of $K$. \par

		For a field $k \subseteq K \subseteq \KK$, a {\emph{Drinfeld $\bA$-module of rank $r$ defined over $K$}} is an $\FF_q$-algebra homomorphism $\rho : \bA \rightarrow K[\tau]$ uniquely determined by 
	\[\rho_t = \theta + \kappa_1\tau + \dots + \kappa_r \tau^r\] 
	such that $\kappa_r \neq 0$. The \emph{exponential function} associated to~$\rho$ is given by $\Exp_\rho(z) = \smash{z + \sum_{h \geq 1} \alpha_h z^{q^h}} \in \power{K}{z}$ and it satisfies the functional equation $\Exp_\rho(\theta z) = \rho_t(\Exp_\rho(z))$. The {\emph{period lattice}} of $\rho$ is the kernel $\Lambda_\rho$ of $\Exp_\rho$, which is a free discrete $\bA$-submodule of rank $r$ inside $\KK$. 
	
	The de Rham cohomology theory for Drinfeld $\bA$-modules was developed by Anderson, Deligne, Gekeler, and Yu \cite{Gekeler89}, \cite{Yu90}. A \emph{$\rho$-biderivation} is an $\FF_q$-linear map $\delta: \bA \to \KK[\tau] \tau$ satisfying, for all $a, b \in \bA$,
	\[
	\delta_{ab}= a(\theta) \delta_b + \delta_a \rho_b.
	\]
Let $u \in \KK[\tau]$. Then, the $\rho$-biderivation $\delta^{(u)}$ defined by $\delta^{(u)}_a = u \rho_a - a(\theta)u$ for all $a \in \bA$ is called an \emph{inner biderivation}. If $u \in \KK[\tau]\tau$, then $\delta^{(u)}$ is said to be \emph{strictly inner}. The set of $\rho$-biderivations $\Der(\rho)$ forms a $\KK$-vector space.  The set of inner biderivations $\Der_{\inn}(\rho)$ and the set of strictly inner biderivations $\Der_{\si}(\rho)$  are also $\KK$-vector subspaces of $\Der(\rho)$. We define the \emph{de Rham module for $\rho$} to be $\rH^1_{\DR}(\rho) := \Der(\rho) / \Der_{\si}(\rho)$, which is an $r$-dimensional $\KK$-vector space. The de Rham module $\rH^1_{\DR}(\rho)$ parametrizes the extensions of $\rho$ by $\GG_a$.\par
	
	For each $\delta \in \Der(\rho)$ there is a unique $\FF_q$-linear and entire power series $\rF_{\delta}(z) = \sum_{i\geq 1} c_i z^{(i)} \in \power{\KK}{z}$ such that, for all $a \in \bA$,
	\begin{equation} \label{E:Fdeltafneq}
		\rF_{\delta}( a(\theta)z) = a(\theta) \rF_{\delta}(z) + \delta_a (\Exp_\rho(z)).
	\end{equation}
	
	\noindent We call $\rF_{\delta}$ the \emph{quasi-periodic function} associated to $\delta$. For $\lambda \in \Lambda_\rho$, the value $\rF_\delta(\lambda)$ is called a \emph{quasi-period of $\rho$}. For $u\in \KK$, the value $\rF_\delta(u)$, which is a coordinate of logarithms of quasi-periodic extensions, is called a {\emph{quasi-logarithm of $\rho$}} (see \cite{BP02}, \cite{NPapanikolas21}).

		A $\KK$-basis of $\rH^1_{\DR}(\rho)$ is represented by $\{\delta_1, \dots, \delta_r\}$, where $\delta_1$ is the inner biderivation such that $(\delta_1)_t = \rho_t-\theta$, and $\delta_j(t) = \tau^{j-1}$ for $2 \leq j \leq r$. We see that $\rF_{\delta_{1}}(z) = \Exp_\phi(z) - z$, and so $\rF_{\delta_{1}}(\lambda) = -\lambda$ for all $\lambda \in \Lambda_\rho$. If we take $\{\lambda_1, \dots, \lambda_r\}$ to be an $\bA$-basis of $\Lambda_\rho$ and we set $\rF_{\tau^{j-1}}(z) := \rF_{\delta_j}(z)$ for $2 \leq j \leq r$, then we define the \emph{period matrix of $\rho$} to be
	\begin{equation*}\label{periodmatrix}
		\bP_\rho := \begin{pmatrix} \lambda_1 & \rF_\tau(\lambda_1) & \dots & \rF_{\tau^{r-1}}(\lambda_1)\\
			\lambda_2 & \rF_\tau(\lambda_2) & \dots & \rF_{\tau^{r-1}}(\lambda_2)\\
			\vdots & \vdots & & \vdots \\
			\lambda_r & \rF_\tau(\lambda_r) & \dots & \rF_{\tau^{r-1}}(\lambda_r)\\
		\end{pmatrix},
	\end{equation*}
which accounts for all periods and quasi-periods of $\rho$. The de Rham cohomology theory for Drinfeld $\bA$-modules runs in parallel to the theory of elliptic functions such that the periods and quasi-periods summarized above play the role of periods and quasi-periods of the Weierstrass $\wp$-functions. 

	If the Drinfeld $\bA$-module $\rho$ is defined over $k^{\sep}$, Denis \cite[p.~6]{Denis95} showed that, for a $\rho$-biderivation $\delta$ defined over $k^\sep$, if $u\in \KK$ such that $\Exp_\rho(u) \in k^{\sep}$, then $u \in k_\infty^\sep$ and $\rF_{\delta}(u)\in k_\infty^\sep$ (see also \cite[Lem. 4.1.22]{NPapanikolas21}). Therefore, for $n \geq 0$ we can consider $\pd_\theta^n(u)$ and $\pd_\theta^n(\rF_\delta(u))$. Let $\pd_{\theta}^n(\bP_{\rho})$ be the matrix formed by entry-wise action of  $\pd_{\theta}^n(\cdot)$ on $\bP_{\rho}$.

	We define $\End(\rho):= \{x \in \KK \mid x\Lambda_\rho \subseteq \Lambda_\rho\}$ and let $K_\rho$ be its fraction field. Our first main result is as follows (restated as Theorem~\ref{T:Main1}):
	
	\begin{theorem}\label{T:Main01}
		Let $\rho$ be a Drinfeld $\bA$-module of rank $r$ defined over $k^{\textup{sep}}$ and suppose that $K_\rho$ is separable over $k$. If $\bss = [K_\rho:k]$, then for $n \geq 1$ we have
		
		\[
		\trdeg_{\ok} \ok \left( \bP_\rho, \pd_\theta^1(\bP_\rho), \dots, \pd_\theta^n(\bP_\rho)\right) 
		= (n+1) \cdot r^2/\bss.\]
	\end{theorem}

	Building on Theorem~\ref{T:Main01}, we prove algebraic independence among hyperderivatives of logarithms and quasi-logarithms of Drinfeld $\bA$-modules. Our second main result is as follows (restated as Theorem~\ref{T:Main2}): 
	\begin{theorem}\label{T:Main02}
		Let $\rho$ be a Drinfeld $\bA$-module of rank $r$ defined over $k^\sep$ and suppose that $K_\rho$ is separable over $k$. Let $u_1, \dots, u_w \in \KK$ with $\Exp_\rho(u_i) = \alpha_i \in k^\sep$ for each $1\leq i \leq w$ and suppose that $\dim_{K_\rho} \Span_{K_\rho}(\lambda_1, \dots, \lambda_r, u_1, \dots, u_w) = r/\bss + w$, where $\bss = [K_\rho:k]$. Then, for $n \geq 1$ 
		\[
		\trdeg_{\ok} \ok \bigg(\bigcup\limits_{s=0}^{n} \bigcup\limits_{i=1}^{r-1}\bigcup\limits_{m=1}^w \bigcup\limits_{j=1}^r \{\pd_\theta^s(\lambda_j), \pd_\theta^s(\rF_{\tau^i}(\lambda_j)),\pd_\theta^s(u_m), \pd_\theta^s(\rF_{\tau^i}(u_m))\}\bigg) 
		= (n+1)(r^2/\bss
		+ rw).\]
	\end{theorem}
	
	For an arbitrary basis of $\rH^1_{\DR}(\rho)$ defined over $k^\sep$, we deduce the following corollary. 
	\begin{corollary}\label{C:Main02}
		Let $\rho$ be a Drinfeld $\bA$-module of rank $r$ defined over $k^\sep$ and suppose that $K_\rho$ is separable over $k$. Let $u_1, \dots, u_w \in \KK$ with $\Exp_\rho(u_i) = \alpha_i \in k^\sep$ for each $1\leq i \leq w$. Let $\{\delta_1, \dots, \delta_r\}$ be a basis of $\rH^1_{\DR}(\rho)$ defined over $k^\sep$. If $u_1, \dots, u_w$ are linearly independent over $K_\rho$, then for $n \geq 1$ the $(n+1)rw$ quantities \[\left\{\bigcup_{s=0}^{n}\bigcup_{j=1}^{r}  \left(\pd_\theta^s(\rF_{\delta_{j}}(u_1)), \pd_\theta^s(\rF_{\delta_{j}}(u_2)), \dots, \pd_\theta^s(\rF_{\delta_{j}}(u_w))\right)\right\}\]  are algebraically independent over $\ok$.
	\end{corollary}

 Combining Theorems~\ref{T:Main01}, \ref{T:Main02}, and Corollary~\ref{C:Main02}, the $\ok$-linear relations among the periods and logarithms of $\rho$  and their hyperderivatives induced by endomorphisms of $\rho$ account for all the $\ok$-algebraic relations among all hyperderivatives of the periods and logarithms as well as all hyperderivatives of the corresponding quasi-periods and quasi-logarithms of $\rho$. 
	
		\subsection{Remarks on structure of the paper}
In \cite{NPapanikolas21}, Papanikolas and the author showed that $t$-motives whose period matrices comprise the values of interest in Theorems~\ref{T:Main01} and \ref{T:Main02} are constructed from the $t$-motive associated to prolongations \cite{Maurischat18} of $\rho$, but did not prove any transcendence results about the values in question. Papanikolas's theorem \cite[Thm.~1.1.7]{P08} states that the transcendence degree of the period matrix of a $t$-motive is equal to the dimension of its Galois group. The primary hurdle, then, is determining the dimension of the associated Galois group of the $t$-motive.

The first goal of this paper is to explicitly determine the Galois group of the $t$-motive corresponding to the $n$-th prolongation $t$-module $\rP_n\rho$ of $\rho$. To do this, we calculate the Zariski closure of the image of the Galois representation on the $\fp$-adic Tate module of $\rP_n\rho$, for a non-zero prime $\fp$ of $\bA$. Next, we immediately extend this result. We construct new $t$-motives whose period matrices are comprised of both periods and quasi-periods of $\rP_n\rho$, and hyperderivatives of logarithms and quasi-logarithms of $\rho$, then determine their Galois groups. We construct a sequence of surjections between specific sub-$t$-motives using consecutive prolongations $\rP_\ell \rho$ for $0\leq \ell \leq n$. These surjections are crucial in establishing that algebraic independence over $\ok$ of all hyperderivatives of the logarithms and quasi-logarithms depend only on $K_\rho$-linear independence of the logarithms.

 The paper is outlined as follows. 
 
 \begin{itemize}
 
 \item In $\S$\ref{S:Prelim} we give necessary background concerning $t$-motives and their Galois groups. Next, we give a brief review of hyperderivatives and then discuss prolongations of dual $t$-motives introduced by Maurischat 
	\cite{Maurischat18}.\par

\item In $\S$\ref{S:Hyperperquasiper}, we describe $t$-motives and rigid analytic trivializations corresponding to Drinfeld $\bA$-modules and their prolongations, then we state Theorem~\ref{P:rigidhyper}. Based on Theorem~\ref{P:rigidhyper} (see \cite[$\S$5.3]{NPapanikolas21} for a detailed account), to prove Theorem~\ref{T:Main01}, for $n \geq 1$ we calculate the Galois group $\Gamma_{\rP_n M_\rho}$ of the $n$-th prolongation $\rP_n M_\rho$ of the $t$-motive $M_\rho$ associated to $\rho$.

\item  
We first make use of a direct connection $\Gamma_{\rP_n M_\rho}$ has with Galois representations. For a non-zero prime $\fp$ of $\bA$, let $\bA_\fp$ be the completion of $\bA$ and let $\bk_{\fp}$ be its fraction field. For a Drinfeld $\bA$-module $\rho$ defined over $K$, where $k \subseteq K \subseteq \ok$ with $[K:k]<\infty$, there is a representation $\varphi_\fp: \Gal(K^\sep/K) \rightarrow \GL_r(\bA_\fp)$ coming from the Galois action on the $\fp$-power torsion points $\rho[\fp^m]:= \{x \in \KK \mid \rho_{\fp^m}(x)=0\}$. In $\S$\ref{S:Periodquasi}, using Anderson generating functions and $\varphi_\fp$, we consider the Galois representation on the $\fp$-adic Tate module of the $n$-th prolongation $t$-module $\rP_n \rho$ associated to $\rho$. The image of this Galois representation is determined using hyperderivatives of the image for the Drinfeld $\bA$-module $\rho$ and is naturally contained in the $\bk_\fp$-valued points of $\Gamma_{\rP_n M_\rho}$ (Theorem~\ref{T:GrHom}).
	
\item For $n \geq 1$, $\rP_{n-1} M_\rho$ is a sub-$t$-motive of $\rP_n M_\rho$ and therefore, $\rP_n M_\rho$ is not simple which makes determining the Zariski closure of the aforementioned image a complicated task. 
To find the Zariski closure, we bring in differential algebraic geometry. We consider hyperdifferential polynomials (precise definition in $\S$\ref{S: Kolchin+}) to determine the above Zariski closure by first determining the defining differential ideal of the aforementioned image and then restricting to Zariski topology (Theorem~\ref{T:GeqGamma}). This allows us to prove Theorem~\ref{T:Main01} and compute the Galois group $\Gamma_{\rP_n M_\rho}$ explicitly (Corollary~\ref{T:LeqGamma}).

\item In $\S$\ref{S:Hyperlogquasilog}, for $u_1, \dots, u_w \in \KK$ satisfying $\Exp_\rho(u_i)\in k^\sep$ for each $1\leq i \leq w$, we build on results of $\S$\ref{S:Periodquasi} to construct new $t$-motives $Y_{1, n}, \dots, Y_{w,n}$ such that the entries of the period matrix of $\oplus_{m=1}^w Y_{m,n}$ comprise $\bigcup_{s=0}^n \bigcup_{i=1}^{r-1}\bigcup_{m=1}^w\{\pd_\theta^s(u_m), \pd_\theta^s(\rF_{\tau^i}(u_m))\}$. Let $\cT$ denote the category of $t$-motives.
In Lemma~\ref{L:Equivvv}, we obtain a surjective map from certain sub-$t$-motives of $Y_{m,n}$ to corresponding sub-$t$-motives of $Y_{m, \ell}$ for $\ell \leq n$ and $1 \leq m \leq w$. This map allows us to implement Theorem~\ref{T:Trivial} that is based on an $\End_\cT(M_{\rho})$-linear independence result of Chang and Papanikolas \cite[Thm.~4.2.2]{CP12} which enables us to prove Theorem~\ref{T:Main02}. 
	
	\item Finally, in Appendix~\ref{DAG}, we cover necessary background concerning differential algebraic geometry in positive characteristic. We explore various properties, especially a result on the determination of the Zariski closure of a set in a differential field (Lemma~\ref{L:ZariskiKolchin}).  

 \end{itemize}
	
	\SkipTocEntry\section*{Acknowledgements}
	
The author is immensely grateful to Matthew A. Papanikolas for many valuable discussions and for pointing out the appropriate references that have aided in the completion of this paper. The author further thanks him for his encouragement throughout the process of writing this paper. The author also thanks Chieh-Yu Chang, O\u{g}uz Gezm\.{i}\c{s}, and Federico Pellarin for helpful comments and suggestions; Yen-Tsung Chen for helpful discussions; and Andreas Maurischat for his questions that helped in closing a gap in the proof of one of the results. Finally, the author is thankful to the referees for many questions, comments, and suggestions, which have helped correct arguments and greatly improved clarity of arguments and exposition. 
	
\section{Preliminaries} \label{S:Prelim}

	\subsection{Notation} \label{S:Not}
We continue with the notation introduced in $\S$\ref{S:IntroHyperPer}. We also define the following. 
 
	Let $\TT$ be the Tate algebra of the closed unit disk of $\KK$, 
	\[
	\TT := \left\{ \sum_{h=0}^\infty a_ht^h \in \power{\KK}{t} \mid \lim_{h \to \infty} \inorm{\,a_h\,} = 0 \right\},
	\]
	and let $\LL$ be its fraction field. 
	
	 For $n \in \ZZ$, recall the Frobenius twist $\tau^n$ from \eqref{E:Frob}. In some cases, we will write $\sigma$ for $\tau^{-1}$. For $M = (m_{ij}) \in \Mat_{e\times d}(\laurent{\KK}{t})$, we define $\smash{M^{(n)}}$ by setting $\smash{M^{(n)} := (m_{ij}^{(n)})}$. Let $\ok(t)[\sigma, \sigma^{-1}]$ be the Laurent polynomial ring over $\ok(t)$ in $\sigma$ subject to the relation 
	\[
	\sigma f = f^{(-1)}\sigma, \quad f \in \ok(t).
	\]
	
	For a field $K \subseteq \KK$, recall from $\S$\ref{S:IntroHyperPer} the twisted power series ring $\power{K}{\tau}$ and the subring $K[\tau]$ given by $\tau f = f^{(1)} \tau$ for all $f \in K$. We also define $\power{K}{\sigma}$ and $K[\sigma]$ when $K$ is a perfect field. For $b = \sum c_i\tau^i \in \KK[\tau]$, we define $b^* := \sum c_i^{(-i)}\sigma^i \in \KK[\sigma]$.  If $B = (b_{ij}) \in \Mat_{e\times d}(\KK[\tau]) = \Mat_{e\times d}(\KK)[\tau]$, then we set $B^* := (b_{ji}^*)$. Thus, if $B \in \Mat_{e\times d}(\KK[\tau])$ and $C \in \Mat_{d\times h}(\KK[\tau])$, then $(BC)^* = C^*B^*$. Moreover, if $B = \beta_0 + \beta_1 \tau + \dots + \beta_\ell \tau^\ell$, then we set $\rd B := \beta_0$.
	
	\subsection{Dual \texorpdfstring{$t$}{t}-motives and \texorpdfstring{$t$}{t}-motives}\label{S:motives}
	In this subsection, we briefly introduce the main tools used in Papanikolas's result. The reader is directed to \cite{P08} for further details. A \textit{pre-$t$-motive} $M$ is a left $\ok(t)[\sigma, \sigma^{-1}]$-module that is finite dimensional over $\ok(t)$. We denote by $\cP$ the category of pre-$t$-motives whose morphisms are the left $\ok(t)[\sigma, \sigma^{-1}]$-module homomorphisms. Let $\bsm \in \Mat_{r\times 1}(M)$ be such that its entries form a $\ok(t)$-basis of $M$. Then, there is a matrix $\Phi \in \GL_r(\ok(t))$ such that \[\sigma \bsm = \Phi \bsm,\] 
	where the action of $\sigma$ on $\bsm$ is entry-wise. 
	\noindent
	We say that $M$ is \emph{rigid analytically trivial} if there exists a matrix $\Psi \in \GL_r(\LL)$ such that
	\[ \Psi^{(-1)} = \Phi \Psi.\]
	\noindent
	The matrix $\Psi$ is called a \emph{rigid analytic trivialization for $\Phi$}. Set $M^\dagger:= \smash{\LL \otimes_{\ok(t)}M}$, where we give $M^\dagger$ a left $\ok(t)[\sigma, \sigma^{-1}]$-module by letting $\sigma$ act diagonally:
	\[
	\sigma(f\otimes m) := f^{(-1)} \otimes \sigma m, \quad f\in \ok(t), \ m\in M.
	\]
	\noindent
	If we let 
	\[
	M^B := (M^\dagger)^\sigma := \{\mu \in M^\dagger  :  \sigma \mu = \mu \},
	\]
	\noindent
	then $M^B$ is a finite dimensional vector space over $\bk$, and $M\mapsto M^B$ is a covariant functor from $\cP$ to the category of $\bk$-vector spaces. The natural map $\LL \otimes_{\ok(t)} M^B \rightarrow M^\dagger$ is an isomorphism if and only if $M$ is rigid analytically trivial \cite[\S3.3]{P08}. If $\Psi$ is a rigid analytic trivialization of $\Phi$, then the entries of $\Psi^{-1}\bsm$ form a $\bk$-basis for $M^B$ \cite[Thm. 3.3.9(b)]{P08}. By \cite[Thm. 3.3.15]{P08}, the category $\cR$ of \emph{rigid analytically trivial pre-t-motives} forms a neutral Tannakian category over $\bk$ with fiber functor $M \mapsto M^B$.

	We now consider $\bA$-finite dual $t$-motives, which were first introduced in \cite{ABP04} (see also \cite{HartlJuschka20}, \cite{NPapanikolas21}). A \emph{dual $t$-motive} $\cM$ is a left $\ok[t, \sigma]$-module that is free and finitely generated as a left $\ok[\sigma]$-module and such that $(t-\theta)^s\cM \subseteq \sigma \cM$ for $s\in \NN$ sufficiently large. If, in addition, $\cM$ is free and finitely generated as a left $\ok[t]$-module, then $\cM$ is said to be \emph{$\bA$-finite}. Thus, if the entries of $\bsm \in \Mat_{r\times1}(\cM)$ form a $\ok[t]$-basis for $\cM$, then there is a matrix $\Phi \in \Mat_r(\ok[t])$ such that $\sigma \bsm = \Phi \bsm$ with $\det \Phi = c(t-\theta)^s$ for some $c\in \ok^\times, s\geq 1$. We say that $\cM$ is \emph{rigid analytically trivial} if there exists a matrix $\Psi \in \GL_r(\TT)$ so that $\Psi^{(-1)} = \Phi \Psi.$ In \cite{ABP04}, the term ``dual $t$-motives" is used for $\bA$-finite dual $t$-motives. We will consider both dual $t$-motives and $\bA$-finite dual $t$-motives. \par
	
	Given an $\bA$-finite dual $t$-motive $\cM$, 
	\[ M :=\ok(t) \otimes_{\ok[t]} \cM\]
	is a pre-$t$-motive where $\sigma(f \otimes m) := f^{(-1)} \otimes \sigma m$. Then, $\cM \mapsto M$ is a functor from the category of $\bA$-finite dual $t$-motives to the category of pre-$t$-motives. We define the category $\cT$ of \emph{$t$-motives} to be the strictly full Tannakian subcategory of $\cR$ generated by the essential image of rigid analytically trivial $\bA$-finite dual $t$-motives under the assignment $\cM \mapsto M$.

	For a $t$-motive $M$, we let $\cT_M$ be the strictly full Tannakian subcategory of $\cT$ generated by $M$. As $\cT_M$ is a neutral Tannakian category over $\bk$, there is an affine group scheme $\Gamma_M$ over $\bk$, a subgroup of the $\bk$-group scheme  $\GL_{r}/\bk$ of $r\times r$ invertible matrices, so that $\cT_M$ is equivalent to the category of finite dimensional representations of $\Gamma_M$ over $\bk$, i.e., $\cT_M \approx \Rep(\Gamma_M, \bk)$ \cite[\S3.5]{P08}. We call $\Gamma_M$ the \emph{Galois group of $M$}.

	\subsection{The difference Galois group}\label{S:DiffGalGr}
	
We now  present a brief summary of the construction of the Galois group of a $t$-motive as the Galois group of a system of difference equations. The reader is directed to \cite{P08} for further details. For a subfield $F \subset \laurent{\KK}{t}$ invariant under the action of $\sigma$, let $F^\sigma$ denote the elements of $F$ fixed by $\sigma$. Note that the automorphism $\sigma:\LL \rightarrow \LL$ restricts to automorphisms of $\bk$ and $\ok(t)$, and $\bk = \bk^\sigma = \ok(t)^\sigma = \LL^\sigma$. 

For a $t$-motive $M$, let $\Phi \in \GL_r(\ok(t))$ denote the action of $\sigma$ on a $\ok(t)$-basis of $M$ and let $\Psi \in \GL_r(\LL)$ be the rigid analytic trivialization for $\Phi$ satisfying $\Psi^{(-1)} =\Phi \Psi$.

	We define a $\ok(t)$-algebra homomorphism $ \nu : \ok(t)[X, 1/\det X] \rightarrow \LL$ by setting $\nu(X_{ij}) := \Psi_{ij}$, where $X = (X_{ij})$ is an $r \times r$ matrix of independent variables. We let $\mathfrak{p} := \ker \nu$ and $\Sigma := \im \nu = \ok(t)[\Psi, 1/\det \Psi] \subseteq \LL$,	and set $Z_\Psi  = \text{Spec} \ \Sigma$. Then, $Z_\Psi$ is the smallest closed subscheme of $\GL_{r}/\ok(t)$ such that $\Psi \in Z_\Psi(\LL)$. \par
	
	Set $\Psi_1, \Psi_2 \in \GL_r(\LL \otimes_{\ok(t)} \LL)$ to be such that $(\Psi_1)_{ij} = \Psi_{ij} \otimes 1$ and $(\Psi_2)_{ij} = 1 \otimes \Psi_{ij}$, and let $\widetilde{\Psi} := \Psi_1^{-1}\Psi_2 \in \GL_r(\LL \otimes_{\ok(t)} \LL)$. We define a $\bk$-algebra homomorphism $ \mu : \bk[X, 1/\det X] \rightarrow \LL \otimes_{\ok(t)} \LL$ by setting $\mu(X_{ij}) := \widetilde{\Psi}_{ij}$. We let $\mathfrak{q} := \ker \mu$ and $\Delta := \im \mu$, and set $\Gamma_\Psi = \text{Spec} \ \Delta$. Then, $\Gamma_\Psi$ is the smallest closed subscheme of $\GL_{r}/\bk$ such that $\widetilde{\Psi} \in \Gamma_\Psi(\LL \otimes_{\ok(t)} \LL)$. The following properties hold.
	
	\begin{theorem}[{Papanikolas~\cite[\S4]{P08}}] \label{Thm.:diffP}
		Let $M$ be a $t$-motive, and let $\Phi \in \GL_r(\ok(t))$ represent multiplication by $\sigma$ on a $\ok(t)$-basis of $M$. Let $\Psi \in \GL_r(\LL)$ satisfy $\Psi^{(-1)} = \Phi \Psi$. 
		\begin{enumerate}
		
			\item[(a)] The closed $\ok(t)$-subscheme $Z_\Psi$ is stable under right-multiplication by $\ok(t) \times_{\bk} \Gamma_\Psi$ and is a $\ok(t) \times_{\bk} \Gamma_\Psi$-torsor over $\ok(t)$. In particular, $\Gamma_\Psi(\overline{\LL}) = \Psi^{-1} Z_\Psi(\overline{\LL})$.
		
			\item[(b)] The $\bk$-scheme $\Gamma_\Psi$ is absolutely irreducible and smooth over $\overline{\bk}$. 
	 
			\item[(c)] $\Gamma_\Psi \cong \Gamma_M$ over $\bk$. 
		\end{enumerate}
	\end{theorem}
	
	\noindent For the $t$-motive $M$, if $\Phi \in \GL_r(\ok(t)) \cap \Mat_r(\ok[t])$ and $\det \Phi = c(t-\theta)^s$ for some $c \in \ok^\times$, $s\geq 1$, then we can pick $\Psi$ to be in $\GL_r(\TT)$ \cite[Prop.~3.3.9(c)]{P08}. Moreover, the entries of $\Psi$ are regular at $t=\theta$ \cite[Prop.~3.1.3]{ABP04}.  Let $\Psi|_{t=\theta}$ denote the specialization of the entries of $\Psi$ at $t=\theta$ and let $\ok(\Psi|_{t=\theta})$ be the field formed by adjoining the entries of $\Psi|_{t=\theta}$ to $\ok$. The main theorem of \cite{P08} is as follows.

	\begin{theorem}[{Papanikolas~\cite[Thm. 1.1.7]{P08}}] \label{T:Tannakian}
		Let $M$ be a $t$-motive, and let $\Gamma_M$ be its Galois group. Suppose that $\Phi \in \GL_r(\ok(t)) \cap \Mat_r(\ok[t])$ represents multiplication by $\sigma$ on a $\ok(t)$-basis of $M$ and that $\det \Phi = c(t-\theta)^s, c \in \ok^\times$, $s\geq 1$. Let $\Psi \in \GL_r(\TT)$ be a rigid analytic trivialization of $\Phi$. Then,
		$\trdeg_{\ok} \ok(\Psi|_{t=\theta}) = \dim \Gamma_M$.
	\end{theorem}

	\subsection{Hyperderivatives and Hyperdifferential operators}\label{HyperD}
	
	In this subsection, we review hyperderivatives and hyperdifferential operators. The reader may refer to \cite{Brownawell99}, \cite{Jeong11}, and \cite[$\S$2.4]{NPapanikolas21} for details. For $m, j \geq 0$, let $\smash{\binom{m}{j}} \in \NN$ denote the usual binomial coefficient modulo $p$. Then, for $F$ a field of characteristic $p>0$ where $\theta$ is transcendental over $F$, the $F$-linear map $\pd_{\theta}^j: F[\theta]\rightarrow F[\theta]$ defined by setting \[\pd_{\theta}^j(\theta^m) = \binom{m}{j} \theta^{m-j}
	\]
	is called the \emph{$j$-th hyperdifferential operator with respect to~$\theta$}. For each $f \in F[\theta]$, we call $\pd_\theta^j(f)$ the \emph{$j$-th hyperderivative of $f$}. The definition of $\pd_\theta^j$ extends naturally to $\pd_\theta^j : \power{F}{\theta} \rightarrow \power{F}{\theta}$. The hyperdifferential operators satisfy various identities including the product rule $\pd_\theta^j(fg) = \sum_{i=0}^j \pd_\theta^i(f)\pd_\theta^{j-i}(g)$ and the composition rule $\pd_\theta^i(\pd_\theta^j(f))= \smash{\binom{i+j}{j}}\partial_\theta^{i+j}(f)$. \par
	
	The product rule extends $\pd_\theta^j$ to the Laurent series field $\laurent{F}{\theta}$ where as usual for $m > 0$, we have $\smash{\binom{-m}{j} = (-1)^j\binom{m+j-1}{j}}$. For a place $v$ of $F(\theta)$ there are unique extensions $\pd_{\theta}^j : F(\theta)_v \to F(\theta)_v$ and $\pd_{\theta}^j : F(\theta)_v^{\sep} \to F(\theta)_v^{\sep}$, where $F(\theta)_v^{\sep}$ is a separable closure of $F(\theta)_v$.

	\begin{proposition}[see {Brownawell~\cite[\S 7]{Brownawell99}, Jeong~\cite[\S 2]{Jeong11}}]\label{P:hyperderprod}
		Let $F$ be a field of characteristic $p>0$, and let $v$ be a place of $F(\theta)$. Then, for $f \in F(\theta)_v^{\sep}$, $n \geqslant 0$, and $j \geq 1$, $\pd_\theta^j : F(\theta)_v^{\sep} \to F(\theta)_v^{\sep}$, $j \geqslant 0$, satisfies
			\[
			\pd_\theta^j\bigl( f^{p^n} \bigr) =
			\begin{cases}
				\bigl(\pd_\theta^e(f)\bigr)^{p^n} & \text{if $j=ep^n$,} \\
				0 & \text{if $p^n \nmid j$.}
			\end{cases}
			\]
	\end{proposition}

	For $f \in F(\theta)_v^{\sep}$ and $n \geq 0$, we define the \emph{$d$-matrix with respect to $\theta$}, $d_{\theta, n} [f] \in \Mat_n(F(\theta)_v^{\sep})$ to be the upper-triangular $n \times n$ matrix
	\begin{equation}\label{E:dmatrix}
		d_{\theta, n}[f] := \begin{pmatrix}
			f & \pd_\theta^1(f) & \dots &\dots & \pd_\theta^{n-1}(f) \\
			& f & \pd_\theta^1(f) &  & \vdots \\
			&   & \ddots & \ddots &\vdots \\
			&   &   & \ddots & \pd_\theta^1(f)\\
			&   &   &  & f
		\end{pmatrix}.
	\end{equation}
	Using the product rule, it is easy to see that 
	$d_{\theta,n}[g] \cdot d_{\theta,n}[f] = d_{\theta,n}[gf]$. For a matrix $B := (b_{ij}) \in \Mat_{e_1\times e_2}(F(\theta)_v^\sep)$, we also define the $d$-matrix with respect to $\theta$, $d_{\theta,n}[B] \in \Mat_{ne_1\times ne_2}(F(\theta)_v^\sep)$ as in \eqref{E:dmatrix}, where we let $\pd_\theta^j(B) := (\pd_\theta^j(b_{ij})) \in \Mat_{e_1\times e_2}(F(\theta)_v^\sep)$.\par
	
	We further define \emph{partial hyperderivatives} for two independent variables $\theta$ and $t$ to be the $F$-linear maps
	\[
	\pd_{\theta}^j , \pd_t^j: F(\theta, t) \rightarrow F(\theta, t), \quad j \geq 0
	\]
	such that for $m\in \ZZ$,  we have $\pd_{\theta}^j(\theta^m) = \binom{m}{j} \theta^{m-j}$, $\pd_{t}^j(t^m) = \binom{m}{j} t^{m-j}$, and $\pd_{\theta}^j(t^m) = \pd_{t}^j(\theta^m)=0$. Thus, we have $\pd_{\theta} \circ \pd_{t} = \pd_{t} \circ \pd_{\theta}$. For $n\geq 0$, we define the $d$-matrices $d_{\theta,n}[\cdot]$ and $d_{t, n}[\cdot]$ with respect to each independent variable $\theta$ and $t$ as in \eqref{E:dmatrix}. 
 
Note that $\pd_t^j$ extends naturally to $\TT$, and $\pd_{\theta}^j$ extends to $\TT \cap \power{k_\infty^\sep}{t}$.

	\subsection{Prolongations of dual \texorpdfstring{$t$}{t}-motives}\label{S:Pro}
	
	We review the construction of prolongations of dual $t$-motives, introduced by Maurischat \cite{Maurischat18}. For a left $\ok[t, \sigma]$-module $\cM$ and $n \geq0$, we define the \emph{$n$-th prolongation of $\cM$} to be the left $\ok[t, \sigma]$-module $\rP_n\cM$ generated by symbols $D_im$, for $m \in \cM$ and $0 \leq i \leq n$ subject to the relations
	\begin{itemize}
		\item[(a)] $D_i(m_1+m_2) = D_im_1 + D_im_2$,
		\item[(b)] $D_i(a \cdot m) = \sum_{i=i_1 +i_2} \pd_t^{i_1}(a) \cdot D_{i_2}m$, 
		\item[(c)] $\sigma(a \cdot D_im) = a^{(-1)} \cdot D_i(\sigma m)$,
	\end{itemize}
	where $m, m_1, m_2 \in \cM$ and $a \in \ok[t]$.

	If $\cM$ is an $\bA$-finite dual $t$-motive, then $\rP_n\cM$ is also an $\bA$-finite dual $t$-motive \cite[Thm. 3.4]{Maurischat18}. Thus, if the entries of $\bsm = [m_1, \dots, m_r]^\tr \in \cM^r$ is a $\ok[t]$-basis of $\cM$, then a $\ok[t]$-basis of $\rP_n\cM$ is given by the entries of
	\begin{equation}\label{E:basisM}
		\bsD_n \bsm := (D_n\bsm^\tr, D_{n-1}\bsm^\tr, \dots, \dots, D_0\bsm^\tr)^\tr \in (\rP_n\cM)^{r(n+1)},
	\end{equation}
	where $D_i\bsm := (D_im_1, \dots, D_im_r)^\tr \in (\rP_n\cM)^{r}$ for each $0 \leq i \leq n$ \cite[Prop.~4.2]{Maurischat18}. Also, if $\Phi \in \GL_r(\ok[t])$ represents multiplication by $\sigma$ on $\bsm$, then
	\begin{equation}\label{E:PhiPro}
		\sigma(\bsD_n \bsm) = d_{t,n+1}[\Phi] \cdot \bsD_n \bsm.
	\end{equation}
	If $\cM$ is rigid analytically trivial with $\Psi \in \GL_r(\TT)$ so that $\Psi^{(-1)} = \Phi \Psi$, then since Frobenius twisting commutes with hyperdifferentiation with respect to $t$, we have
	\begin{equation}\label{E:ratpro}
		(d_{t,n+1}[\Psi])^{(-1)} = d_{t,n+1}[\Psi^{(-1)}] = d_{t,n+1}[\Phi \Psi] = d_{t,n+1}[\Phi] d_{t,n+1}[\Psi].
	\end{equation}
	Therefore, $\rP_n\cM$ is rigid analytically trivial. \par 
	
	Via $D_0m \mapsto m$, we see that $\rP_0\cM$ is naturally isomorphic to $\cM$, and as in \cite[Rem.~3.2]{Maurischat18}, for $0 \leq j \leq n-1$ we obtain a short exact sequence of dual $t$-motives
	\begin{equation}\label{E:motivespro}
		0 \rightarrow \rP_j\cM \rightarrow \rP_{n}\cM \xrightarrow{\boldsymbol{\rpr}_{n-j-1}} \rP_{n-j-1}\cM \rightarrow 0
	\end{equation}
	where ${\boldsymbol{\rpr}_{n-j-1}}(D_im) := D_{i-j-1}m$ for $i > j$ and ${\boldsymbol{\rpr}_{n-j-1}}(D_im) := 0$ for $i \leq j$ and $m \in \cM$.

	\section{Rigid analytic trivializations and hyperderivatives}
	\label{S:Hyperperquasiper}
	
	The goal of this section is to provide necessary background on Anderson $t$-modules for the purpose of studying Drinfeld $\bA$-modules and their prolongations, and their connection to dual $t$-motives and rigid analytic trivializations via Anderson generating functions. Then, we state Theorem~\ref{P:rigidhyper} which provides the connection between Taylor coefficients of series expansions of Anderson generating functions and hyperderivatives of periods, quasi-periods, logarithms, and quasi-logarithms of a Drinfeld $\bA$-module defined over $k^{\sep}$.
	
	\subsection{Anderson \texorpdfstring{$t$}{t}-modules, dual \texorpdfstring{$t$}{t}-motives, and Anderson generating functions} \label{S:motivemodule}
	
	For a field $K \subseteq \KK$, an \emph{Anderson $t$-module defined over $K$} is an $\FF_q$-algebra homomorphism $\phi : \bA \rightarrow \Mat_d(K[\tau])$ defined uniquely by
	\[\phi_t = B_0 + B_1\tau + \dots + B_\ell \tau^\ell,\]
	where $B_i \in \Mat_d(K)$ for $0 \leq i \leq \ell$, and $\rd \phi_t = B_0 = \theta I_d + N$ such that $I_d$ is the $d \times d$ identity matrix and $N$ is a nilpotent matrix. Then, $\phi$ defines an $\bA$-module structure on $\KK^d$ via 
	\begin{equation}\label{E:moduleaction}
		a \cdot \bsx = \phi_a(\bsx), \quad a \in \bA,\ \bsx \in \KK^d.
	\end{equation}
	We call $d$ the \emph{dimension of $\phi$}.  If $\phi_t=B_0 \in \Mat_d(K)$, then $\phi$ is said to be a {\emph{trivial}} Anderson $t$-module. A non-trivial Anderson $t$-module of dimension $1$ is called a \emph{Drinfeld $\bA$-module}. \par
	
	There exists a unique power series $\Exp_\phi(\bsz) = \sum_{i=0}^\infty C_i\bsz^{(i)} \in \power{\KK}{z_1, \dots, z_d}^d$,  $\bsz = [z_1, \dots, z_d]^\tr$ so that $C_0 = I_d$ and satisfies
	\[
	\Exp_\phi(\rd \phi_a \bsz) = \phi_a(\Exp_\phi(\bsz))
	\]
	for all $a \in \bA$. Moreover, $\Exp_\phi(\bsz)$ defines an entire function $\Exp_\phi:\KK^d \rightarrow \KK^d$. If $\Exp_\phi$ is surjective, then we say that $\phi$ is \emph{uniformizable}. The kernel $\Lambda_\phi \subseteq \KK^d$ of $\Exp_\phi$ is a free and finitely generated discrete $\bA$-submodule of $\KK^d$ through the action of $\rd\phi(\bA)$ and it is called the \emph{period lattice of $\phi$}. If $\phi$ is uniformizable, then we have an isomorphism $\KK^d/\Lambda_\phi \cong (\KK^d, \phi)$ of $\bA$-modules, where $(\KK^d, \phi)$ denotes $\KK^d$ together with the $\bA$-module structure defined in \eqref{E:moduleaction} coming from $\phi$. For more details about Anderson $t$-modules, see \cite{And86}, \cite{BPrapid}, \cite{Thakur}. \par
	
	We define the dual $t$-motive $\cM_\phi$ associated to a $t$-module $\phi$ defined over $K\subseteq \ok$ in the following way. We let $\cM_\phi:= \Mat_{1\times d}(\ok[\sigma])$. To give $\cM_\phi$ the $\ok[t, \sigma]$-module structure, set
	\begin{equation}\label{E:dualaction}
		a \cdot m = m \phi_a^*, \quad m \in \cM_\phi, a \in \bA,
	\end{equation}
	where $\phi_a^*$ is defined as in $\S$\ref{S:Not}. For each $m \in \cM_{\phi}$, by straightforward computation we obtain $(t-\theta)^d \cdot m \in \sigma \cM_{\phi}$. Thus, $\cM_\phi$ defines a dual $t$-motive and \eqref{E:dualaction} gives a unique correspondence between a $t$-module and its associated dual $t$-motive (see also \cite[\S 4.4]{BPrapid}, \cite{HartlJuschka20} and \cite[\S 2.3]{NPapanikolas21}). If $\cM_{\phi}$ is $\bA$-finite, then we say that \emph{$\phi$ is $\bA$-finite} and call the rank of $\cM_{\phi}$ as a left $\ok[t]$-module the \emph{rank of $\phi$}. The reader is directed to \cite{HartlJuschka20}, \cite[$\S$2.3]{NPapanikolas21} for more information on dual $t$-motives associated to $t$-modules. \par
	
	We conclude this subsection by introducing Anderson generating function associated to $t$-modules (see \cite{Green19b}, \cite{Maurischat19b}, \cite{NPapanikolas21} for further details). For $\bsy \in \KK^d$, we define the \emph{Anderson generating function for $\phi$} by the infinite series
	\begin{equation}\label{E:AGFG}
	\cG_{\bsy}(t) := \sum_{m=0}^{\infty} \Exp_{\phi}(\rd\phi_t^{-m-1}\bsy)t^m \in \TT^d.
	\end{equation}
We explore the properties we will use in subsections \ref{RATPro}, \ref{S:Tatem}, \ref{S:quasilogtmotive}. For clarity, we will denote by $f_y(t)$ the Anderson generating function for a Drinfeld $\bA$-module at $y \in \KK$.

\subsection{Prolongations of Drinfeld \texorpdfstring{$\bA$}{A}-modules and associated dual \texorpdfstring{$t$}{t}-motives}\label{S:Produal}
Let $\rho:\bA \rightarrow K[\tau]$ be a Drinfeld $\bA$-module defined over $K \subseteq \ok$ such that 
	\[
	\rho_t = \theta + \kappa_1\tau + \dots + \kappa_r\tau^r,
	\]
	where $\kappa_r \neq 0$. Drinfeld $\bA$-modules are uniformizable and the rank of the period lattice $\Lambda_\rho$ of $\rho$ as an $\bA$-module is~$r$. As defined above for $t$-modules, we define the dual $t$-motive $\cM_\rho:= \ok[\sigma]$. Then the set $\{m_1, m_2, \dots, m_r\}=\{1, \sigma, \dots, \sigma^{r-1}\}$ forms a $\ok[t]$-basis for $\cM_\rho$ \cite[\S3.3]{CP12}, \cite[Example 3.5.14]{NPapanikolas21}, and with respect to this basis, multiplication by $\sigma$ on $\cM_\rho$ is represented by
	\begin{equation}\label{E:Phi}
		\Phi_\rho := \begin{pmatrix}
			0 & 1 & \dots & 0 \\
			\vdots & \vdots & \ddots & \vdots \\
			0 & 0 & \dots & 1 \\
			(t-\theta)/\kappa_r^{(-r)} & -\kappa_1^{(-1)}/\kappa_r^{(-r)} & \dots & -\kappa_{r-1}^{(-r+1)}/\kappa_r^{(-r)}
		\end{pmatrix}.
	\end{equation}
	Thus, $\cM_\rho$ is $\bA$-finite. We let $M_\rho := \ok(t) \otimes_{\ok[t]} \cM_\rho$ be the pre-$t$-motive associated to $\cM_\rho$. \par
	
	For Drinfeld $\bA$-modules $\rho$ and $\rho'$ defined over $K \subseteq \KK$, a \emph{morphism} $b : \rho \rightarrow \rho'$ is a twisted polynomial $b \in \KK[\tau]$ such that $b \rho_a = \rho'_a b$ for all $a \in \bA$. We say that $b$ is defined over $L \subseteq \KK$ if $b \in L[\tau]$. A morphism  $b : \rho \rightarrow \rho'$ defined over $\ok$ induces a morphism $B : \cM_\rho \rightarrow \cM_{\rho'}$ of $\bA$-finite dual $t$-motives in the following way. If $b = \sum c_i\tau^i \in L[\tau]$, recall from $\S$\ref{S:Not} that $b^* = \sum c_i^{(-i)}\sigma^i$. Then, $B$ is the $\ok[\sigma]$-linear map such that $B(1) = b^*$ (see \cite[Lem. 2.4.2]{CP11}). 

 The map
	\begin{equation}\label{E:Endr}
	\End(\rho)~\rightarrow~\{ c\in \KK \mid c\Lambda_\rho \subseteq \Lambda_\rho\} : \sum c_i \tau^i \mapsto c_0
	\end{equation}
	is an isomorphism \cite{Drinfeld}. Throughout this paper, we identify $\End(\rho)$ with the image of this map and let $K_\rho$ denote its fraction field. We state the following result due to Anderson. 
	
	\begin{proposition}[{see Chang-Papanikolas~\cite[Prop.~3.3.2, Cor.~3.3.3]{CP12}}] \label{P:0thpro}
		The functor $\rho \rightarrow \cM_\rho$ from the category of Drinfeld $\bA$-modules defined over $K \subseteq \ok$ to the category of $\bA$-finite dual $t$-motives is fully faithful. Moreover, 
		\[
		\End(\rho) \cong \End_{\ok[t, \sigma]}(\cM_\rho), \quad K_\rho \cong \End_\cT(M_\rho),
		\]
		and $M_\rho$ is a simple left $\ok(t)[\sigma, \sigma^{-1}]$-module. 
	\end{proposition}

\begin{remark}\label{R:0thpro}
Let $i_t\in \End_{\ok[t,\sigma]}(\cM_\rho)$ be such that $i_t(1) = t\cdot 1 = \rho_t^*$. The isomorphism $\End(\rho) \cong \End_{\ok[t, \sigma]}(\cM_\rho)$ in Proposition~\ref{P:0thpro} sends $\theta \mapsto i_t$ and so, it sends $A$ to $\bA$. Thus, $K_\rho \cong \End_\cT(M_\rho)$ sends $k$ to $\bk$.
\end{remark}

	For $n \geq 0$, we define the \emph{$n$-th prolongation $t$-module $\rP_n\rho$ of $\rho$} to be the Anderson $t$-module associated to the $n$-th prolongation $\rP_n\cM_\rho$ of the $\bA$-finite dual $t$-motive $\cM_\rho$ (see for details, \cite[\S5]{Maurischat18}, \cite[\S5.2]{NPapanikolas21}). The Anderson $t$-module $\rP_n\rho : \bA \rightarrow \Mat_{n+1}(K[\tau])$ is of dimension $n+1$ and is defined by 
	\[
	(\rP_n\rho)_t =
	\rd (\rP_n\rho)_t + \diag(\kappa_1)\tau + \dots + \diag(\kappa_r)\tau^r
	\]
	where 
	\begin{equation} \label{E:constpro}
		\rd (\rP_n\rho)_t = \begin{pmatrix} 
			\theta &  &  &  & \\ 
			-1 & \ddots &&& \\ 
			0 & \ddots & \ddots & & \\ 
			\vdots & \ddots & \ddots & \ddots &  \\
			0 & \dots & 0 & -1 & \theta 
		\end{pmatrix},
	\end{equation}
	and $\diag(\kappa_i)$ is the $(n+1)\times (n+1)$ diagonal matrix with diagonal entries all equal to $\kappa_i$ for each $1 \leq i \leq r$. If we set $\cM_{\rP_n \rho} := \Mat_{1\times (n+1)}(\ok[\sigma])$ to be the dual $t$-motive associate to $\rP_n \rho$ defined as in \eqref{E:dualaction}, then by \cite[Prop.~5.2.11(b)]{NPapanikolas21} we have
	\[
	\cM_{\rP_n \rho} = \rP_n \cM_{\rho}.
	\]
	We define $\rP_nM_\rho := \ok(t) \otimes_{\ok[t]} \rP_n\cM_\rho$ to be the pre-$t$-motive associated to $\rP_n\cM_\rho$.

	\subsection{Rigid analytic trivializations}\label{RATPro}
	
	We fix our choice of Drinfeld $\bA$-module $\rho$ of rank $r$ from $\S$\ref{S:Produal} such that it is defined over $K=k^\sep$. In this subsection, we show that the $\bA$-finite dual $t$-motive $\cM_\rho$ associated to $\rho$ is rigid analytically trivial by constructing the rigid analytic trivialization $\Psi_\rho$, and then extend to the prolongation $t$-module $\rP_n \rho$. The details regarding Drinfeld $\bA$-modules can be found in \cite[\S3.4]{CP12} and \cite[Example~4.6.7]{NPapanikolas21}. \par 
	
	For $u \in \KK$, we let $f_u(t) \in \TT$ denote the Anderson generating function of $\rho$ given as in \eqref{E:AGFG}. For an $\bA$-basis $\{\lambda_1, \dots, \lambda_r\}$ of $\Lambda_\rho$, we set $f_i(t) := f_{\lambda_i}(t)$ for each $1 \leq i \leq r$. Define the matrices 
	\begin{equation}\label{E:Upsilon}
		\Upsilon := \begin{pmatrix}
			f_1 & f_1^{(1)} & \dots & f_1^{(r-1)}\\
			f_2 & f_2^{(1)} & \dots & f_2^{(r-1)}\\
			\vdots & \vdots & & \vdots\\
			f_r & f_r^{(1)} & \dots & f_r^{(r-1)}
		\end{pmatrix}, \, \text{and} \, \, 	V := \begin{pmatrix}
		\kappa_1 & \kappa_2^{(-1)} & \dots & \kappa_{r-1}^{(-r+2)} & \kappa_r^{(-r+1)}\\
		\kappa_2 & \kappa_3^{(-1)} & \dots & \kappa_{r}^{(-r+2)} & \\
		\vdots & \vdots &&&\\
		\kappa_{r-1} & \kappa_r^{(-1)}&&&\\
		\kappa_r &&&&
	\end{pmatrix}.
	\end{equation}
By \cite[$\S$3.4]{CP12} (see also~\cite[Lem. 4.3.9]{NPapanikolas21}), it follows that $\det \Upsilon \neq 0$. Set 
	\begin{equation}\label{E:ratD}
		\Psi_\rho := V^{-1}[\Upsilon^{(1)}]^{-1}.
	\end{equation}
	Then $\Psi_{\rho}^{(-1)} = \Phi_\rho \Psi_\rho$. Thus, the pre-$t$-motive $M_\rho = \ok(t) \otimes_{\ok[t]} \cM_\rho$ is rigid analytically trivial and is in the category $\cT$ of $t$-motives. \par
	
	By \eqref{E:ratpro}, the $n$-th prolongation $t$-motive $\rP_n M_\rho = \ok(t) \otimes_{\ok[t]} \rP_n\cM_\rho$ is rigid analytically trivial with rigid analytic trivialization $\Psi_{\rP_n\rho} = d_{t,n+1}[\Psi_\rho]$. Thus, 
    	\begin{equation}\label{ratpro1}
		\Psi_{\rP_n\rho} = d_{t,n+1}[V]^{-1} d_{t,n+1}[\Upsilon^{(1)}]^{-1}. 
	\end{equation}
	
	\subsection{Hyperderivatives of periods and logarithms}\label{RATProHyper}
	We continue with our choice of Drinfeld $\bA$-module $\rho$ of rank $r$ defined over $k^\sep$. Recall from $\S$\ref{S:IntroHyperPer} that a $\KK$-basis of $\rH^1_{\DR}(\rho)$ is represented by $\{\delta_1, \dots, \delta_r\}$, where $\delta_1$ is the inner biderivation such that $(\delta_1)_t = \rho_t-\theta = \kappa \tau + \dots + \kappa_r \tau^r$, and $\delta_j(t) = \tau^{j-1}$ for $2 \leq j \leq r$. Let $\rF_{\tau^{j-1}}(z)$ denote the quasi-periodic function associated to the biderivation $\delta_j: t \mapsto \tau^{j-1}$. Note that $\rF_{\delta_1}(z) = \Exp_{\rho}(z)-z$.
	Then, we have the following result which is a modified version for Drinfeld $\bA$-modules, and its proof is due to Papanikolas and the author.  
	
	\begin{theorem}[see {Namoijam-Papanikolas~\cite[Thm.~E]{NPapanikolas21}}]\label{P:rigidhyper}
		Let $\rho$ be a Drinfeld $\bA$-module defined over $k^\sep$ of rank $r$. Let $u \in \KK^d$ satisfy $\Exp_\phi(u) \in (k^\sep)^d$.  Then, for $n \geqslant 0$ 
		\begin{equation}\label{E:hyperquasilog11}
				\Span_{\ok}\bigg(\{1\} \cup\bigcup\limits_{s=0}^n  \bigcup\limits_{\ell=1}^{r} \bigl\{
				\pd_t^s(f_u^{(\ell)}(t))|_{t=\theta} \bigr\}\bigg)=\Span_{\ok}\bigg(\{1\} \cup
				\bigcup\limits_{s=0}^n \bigcup\limits_{j=1}^{r-1} \bigl\{\pd_{\theta}^s(u), \pd_\theta^s(\rF_{\tau^j}(u))\bigr\}\bigg).
		\end{equation}
		In particular, if $\{\lambda_1, \dots, \lambda_r\}$ is an $\bA$-basis of the period lattice $\Lambda_\rho$, then
		\begin{equation}\label{E:hyperquasiper11}
			\Span_{\ok} \left( d_{t, n+1}[\Psi_{\rho}]^{-1}\big|_{t=\theta} \right)
			= \Span_{\ok}\bigg(\bigcup\limits_{s=0}^{n} \bigcup\limits_{i=1}^{r} \bigcup\limits_{j=1}^{r-1} \{ \partial_\theta^s(\lambda_j), 
			\partial_\theta^s(\rF_{\tau^j}(\lambda_i))\}  \bigg).
		\end{equation}
	\end{theorem}
	
By using Theorem~\ref{T:Tannakian} and Theorem~\ref{P:rigidhyper}, computing the dimension of the Galois group $\Gamma_{\rP_n M_\rho}$ for $n \geq 1$ proves Theorem~\ref{T:Main01}. Moreover, by \eqref{E:hyperquasilog11} if we are able to construct appropriate $t$-motives whose periods span the hyperderivatives in question and determine the dimension of their associated Galois groups, then we can prove Theorem~\ref{T:Main02}.

	\section{Hyperderivatives of periods and quasi-periods}
	\label{S:Periodquasi}
	Let $\rho$ be a Drinfeld $\bA$-module of rank $r$ defined over $k^\sep$. Let $K_{\rho}$ be the fraction field of $\End(\rho)$ defined as in \eqref{E:Endr} and let $[K_{\rho}:k]=\bss$. In this section, we prove Theorem~\ref{T:Main01} (restated as Theorem~\ref{T:Main1}). To prove this theorem, we first show in Theorem~\ref{T:GeqGamma} that $\dim \Gamma_{\rP_nM_\rho}\geq (n+1)\cdot r^2/\bss$, and in Theorem~\ref{T:1LeqGamma} that $\dim \Gamma_{\rP_nM_\rho}\leq (n+1)\cdot r^2/\bss$. Moreover, in Corollary~\ref{T:LeqGamma} we explicitly compute the Galois group $\Gamma_{\rP_n M_\rho}$ for all $n \geq 1$.
	
	\subsection{The \texorpdfstring{$\fp$}{\fp}-adic Tate module and Anderson generation functions}\label{S:Tatem}
	
	Let $\phi$ be a uniformizable and $\bA$-finite Anderson $t$-module of dimension $d$ and rank~$r$. For any $a \in \bA$, the torsion $\bA$-module $\phi[a] := \{\bsx \in \KK^d \ | \ \phi_a(\bsx) = 0\}$ is isomorphic to $\big(\bA/(a)\big)^{\oplus r}$ (see \cite{And86}, \cite[Thm.~7.2.1]{Thakur}). For a non-zero prime $\mathfrak{p}$ of $\bA$, we define the \emph{$\mathfrak{p}$-adic Tate module} 
	\[
	T_{\mathfrak{p}}(\phi) := \lim\limits_{\substack{\longleftarrow \\ m}} \phi[\mathfrak{p}^m] \cong \bA_{\mathfrak{p}}^{\oplus r},
	\]
	where $\bA_{\mathfrak{p}}$ is the completion of $\bA$ at $\mathfrak{p}$. Now, we fix a Drinfeld $\bA$-module $\rho$ of rank~$r$. If $\rho$ is defined over $K$ such that $k \subseteq K \subseteq \ok$ and $[K:k] < \infty$, then note that every element of $\rho[\mathfrak{p}^m]$ is separable over $K$. Thus, the absolute Galois group $\Gal(K^\sep/K)$ of the separable closure of $K$ inside $\ok$ acts on $T_{\mathfrak{p}}(\rho)$, defining a representation
	\[
	\varphi_{\mathfrak{p}} : \Gal(K^\sep/K) \rightarrow \Aut(T_{\mathfrak{p}}(\rho)) \cong \GL_r(\bA_{\mathfrak{p}}).
	\]
Set $\mathrm{p}:=\mathfrak{p}(\theta) \in A$. We fix  an $\bA$-basis $\{\lambda_1, \dots, \lambda_r\}$ of $\Lambda_\rho$ and define 
	\[
	\xi_{i,m} := \Exp_\rho \bigg(\frac{\lambda_i}{\rp^{m+1}}\bigg) \in \rho[\fp^{m+1}],
	\]
	for each $1 \leq i \leq r$ and $m \geq 0$. Then, $\{x_1, \dots, x_r\}$ is an $\bA_{\fp}$-basis of $T_{\fp}(\rho)$, where we set $x_i := (\xi_{i,0}, \xi_{i,1}, \xi_{i,2}, \dots )$. Set $\bx := [x_1, \dots, x_r]^\tr$. Then, for each $\epsilon \in \Gal(K^\sep/K)$ there exists $g_\epsilon \in \GL_r(\bA_{\mathfrak{p}})$ such that 
	\begin{equation} \label{E:gepsilon}
		\varphi_\fp(\epsilon) \bx = g_\epsilon \bx.
	\end{equation}

	\begin{theorem} [{Maurischat-Perkins~\cite[Theorem~1.2]{MauPer20}}]\label{T:GrHomD}
		Let $\rho$ be a Drinfeld $\bA$-module defined over $K$ such that $k \subseteq K \subseteq \ok$ and $[K:k] < \infty$. 
		Let $\bk_\fp$ be the field of fractions of $\bA_{\fp}$.
		For each $\epsilon \in  \Gal(K^\sep/K)$, let $g_{\epsilon} \in \GL_r(\bA_\fp)$ be as in \eqref{E:gepsilon}. Then, the assignment $\epsilon \mapsto g_\epsilon$ induces a group homomorphism 
		\[
		\beta_0 : \Gal(K^\sep/K) \rightarrow \Gamma_{\Psi_{\rho}}(\bA_\fp) := \GL_r(\bA_\fp)\cap\Gamma_{\Psi_{\rho}}(\bk_\fp).
		\]
	\end{theorem}
	
Note that in the case of $\fp =t$, Theorem~\ref{T:GrHomD} was first proved by Chang and Papanikolas (see \cite[Thm. 3.5.1]{CP12}). 

For the remainder of this subsection, we fix $n \geq 0$. By \cite[Prop.~5.2.16]{NPapanikolas21}, we have that for $\bsz = [z_0, \dots, z_{n}]^\tr$, 
	\begin{equation}\label{E:exppro}
		\Exp_{\rP_n\rho}(\bsz) = [\Exp_\rho(z_0), \dots, \Exp_\rho(z_{n})]^\tr.
	\end{equation}
For $u \in \KK$, set
	\begin{equation}\label{E:ucol}
	(u)_{j} :=  [0, \dots, 0, u, 0, \dots, 0]^\tr =\begin{psmallmatrix}
	    0\\ \vdots \\ 0 \\ u\\0\\ \vdots \\ 0
	\end{psmallmatrix} \in \KK^{n+1},
\end{equation}
where $u$ is in the $j$-th entry and all other entries are $0$. By \eqref{E:exppro}, using the $\bA$-basis $\{\lambda_1, \dots, \lambda_r\}$ of $\Lambda_\rho$, an $\bA$-basis of the period lattice $\Lambda_{\rP_n\rho}$ of $\rP_n\rho$ is
\[\{ (\lambda_i)_{j} \ | \ 1 \leq i \leq r~\textup{and}~1 \leq j \leq n+1\}.\]

We define 
	\[
	\chi_{i,m} := \Exp_\rho \bigg(\frac{\lambda_i}{\theta^{m+1}}\bigg), 
	\]
	for each $1 \leq i \leq r$ and $m \geq 0$. By \eqref{E:AGFG}, the Anderson generating function $f_i(t):=f_{\lambda_i}(t)$ of $\rho$ with respect to~$\lambda_i$ is
	\[
	f_i(t) = \sum\limits_{m=0}^\infty \Exp_{\rho}\bigg(\frac{\lambda_i}{\theta^{m+1}} \bigg)t^m = \sum\limits_{m=0}^\infty \chi_{i,m}t^m 
	\in \TT \cap \power{K^\sep}{t}.
	\]
For each $1 \leq i \leq r$ and $1 \leq j \leq n+1$ we let $\bcG_{i,j}(t) := \bcG_{(\lambda_i)_j}(t)$ denote the Anderson generating function of $\rP_n\rho$ with respect to $(\lambda_i)_j$. Then, by \eqref{E:AGFG} we have
	\[
	\bcG_{i,j}(t) = \sum\limits_{m=0}^\infty \Exp_{\rP_n\rho}\big( (\rd(\rP_n\rho)_t)^{-m-1} (\lambda_i)_{j} \big)t^m 
	\in \TT^{n+1} \cap \power{K^\sep}{t}^{n+1}.
	\]
	Observe that in \eqref{E:constpro}, the subdiagonal entries of $\rd(\rP_n\rho)_t$ are $-\pd_\theta^1(\theta)$. Also, $(-1)^c \pd_\theta^c (\theta)=0$ for $c \geq 2$. Moreover, the $e$-th subdiagonal entries of  $\rd(\rP_n\rho)_t^{-1}$ are $(-1)^e\pd_\theta^e(\theta^{-1})$ and so, by the product rule of hyperderivatives, for $h \in \ZZ$ we have 
	\begin{equation*} \label{E:power}
		(\rd(\rP_n\rho)_t)^{h} = \begin{pmatrix}
			\theta^{h} & & & & &\\
			-\pd_\theta^1(\theta^{h}) & \theta^{h} &&&&\\
			\pd_\theta^2(\theta^{h}) & -\pd_\theta^1(\theta^{h}) & \theta^{h} &&&\\
			\vdots&&\ddots&\ddots&&\\
			\vdots&&&\ddots&\ddots& \\
			(-1)^n\pd_\theta^n(\theta^{h}) & \dots &\dots& \pd_\theta^2(\theta^{h})& -\pd_\theta^1(\theta^{h}) & \theta^{h}
		\end{pmatrix}.
	\end{equation*}
	
	\noindent
Note that for $m \geq 0$ and $c \geq 1$, $(-1)^c \pd_\theta^c (\theta^{-m-1}) =  \binom{m+c}{c} \theta^{-m-1-c}$. Then, it follows by using \eqref{E:exppro} that
	\begin{equation*}
		\begin{split}
		&	\bcG_{i,j}(t) = \sum\limits_{m=0}^{\infty}\left(0, \dots, 0, \chi_{i,m}, \binom{m+1}{1} \chi_{i,m+1}, \dots, \binom{m+(n+1-j)}{n+1-j} (\chi_{i,m+(n+1-j)})\right)^{\tr}t^m\\
			&=\left(0, \dots, 0, \sum\limits_{m=0}^{\infty}\chi_{i,m}t^{m}, \sum\limits_{m=1}^{\infty}\chi_{i,m} \binom{m}{1}t^{m-1}, \dots, \sum\limits_{m=n+1-j}^{\infty}\chi_{i,m}\binom{m}{n+j-1}t^{m-(n+j-1)}\right)^{\tr}.
		\end{split}
	\end{equation*}
	Thus,
	\begin{equation}\label{E:AGFPro}
		\bcG_{i,j}(t) =  \left(0, \dots, 0, f_i, \pd_t^1(f_i), \dots, \pd_t^{n+1-j}(f_i)\right)^{\tr} \in \TT^{n+1}.
	\end{equation}   

For our purpose, we consider the Galois group $\Gamma_{\Psi_{\rP_n \rho}}$ and its principal homogeneous space $Z_{\Psi_{\rP_n \rho}}$ as in $\S$\ref{S:DiffGalGr}, and we prove the following result.	
	\begin{theorem} \label{T:GrHom}
		Let $\rho$ be a Drinfeld $\bA$-module defined over $K$ such that $k \subseteq K \subseteq \ok$ and $[K:k] < \infty$, and for $n \geq 0$ let $\rP_n\rho$ be its $n$-th prolongation $t$-module. 
		Let $\bk_\fp$ be the fraction field of $\bA_{\fp}$.
		For each $\epsilon \in  \Gal(K^\sep/K)$, let $g_{\epsilon} \in \GL_r(\bA_\fp)$ be as in \eqref{E:gepsilon}.
		Then, the assignment $\epsilon \mapsto d_{t,n+1}[g_\epsilon]$ induces a group homomorphism 
		\[
		\beta_n : \Gal(K^\sep/K) \rightarrow \Gamma_{\Psi_{\rP_n\rho}}(\bA_\fp) := \GL_{(n+1)r}(\bA_\fp)\cap\Gamma_{\Psi_{\rP_n\rho}}(\bk_\fp).
		\]
	\end{theorem}
We follow the methods used by Maurischat and Perkins \cite{MauPer20}. Let $\overline{\FF_q}$ denote an algebraic closure of $\FF_q$ inside $\KK$ and let $\zeta \in \overline{\FF_q}$ be a root of $\fp$. We define the $\KK$-algebra map $\mathcal{D}_\zeta:\TT \rightarrow \power{\KK}{X}$ by
\[
g \mapsto \sum_{m=0}^{\infty} \pd_t^m(g)|_{t=\zeta}X^m.
\]
By \cite[Lem.~2.2]{MauPer20}, the map $\mathcal{D}_\zeta:\bA \rightarrow \power{\KK}{X}$ extends to an isomorphism $\mathcal{D}_\zeta:\bA_\fp \rightarrow \power{\FF_q(\zeta)}{X}$. 

	The Galois group $\Gal(K^\sep/K)$ acts on $\power{K^\sep}{X}$ by acting on each coefficient. 
We now consider the Galois action on Anderson generating functions of $\rP_n\rho$ and their Frobenius twists. 

\begin{proposition}[{cf.~\cite[Proposition~4.2]{MauPer20}}]\label{P:AGFz}
    	For each $\epsilon \in \Gal(K^\sep/K)$, let $g_\epsilon \in \GL_r(\bA_\fp)$ be defined as in \eqref{E:gepsilon}. Let $\bcG := [\cG_{1,1}, \dots, \cG_{r,1}, \dots, \dots, \cG_{1,n+1}, \dots, \cG_{r,n+1}]^\tr \in \Mat_{r(n+1)\times n+1}(\TT)$. Then, 
		\[ \epsilon(\mathcal{D}_{\zeta}(\bcG)) = \mathcal{D}_{\epsilon(\zeta)}(d_{t, n+1}[g_\epsilon])\cdot\mathcal{D}_{\epsilon(\zeta)}(\bcG), \] where \[\epsilon(\mathcal{D}_{\zeta}(\bcG)) = [\epsilon(\mathcal{D}_{\zeta}(\cG_{1,1})), \dots, \epsilon(\mathcal{D}_{\zeta}(\cG_{r,1})), \dots, \dots, \epsilon(\mathcal{D}_{\zeta}(\cG_{1,n+1})), \dots, \epsilon(\mathcal{D}_{\zeta}(\cG_{r,n+1}))]^\tr.\] 
\end{proposition}

\begin{proof}
  Note that by \eqref{E:AGFPro}, the $j$-th column of $\cG$ for $1\leq j \leq n+1$ is 
  \[
  [\pd_t^{j-1}(f_1), \dots, \pd_t^{j-1}(f_r), \pd_t^{j-2}(f_1), \dots, \pd_t^{j-2}(f_r), \dots \dots,  f_1, \dots, f_r, 0 \dots \dots 0]^{\tr} \in \TT^{r(n+1)}.
  \]
Then, for $m_1, m_2 \in \NN$ and $1 \leq i \leq r$, since $\pd_t^{m_1}(\pd_t^{m_2}(f_i)) = \binom{m_1+m_2}{m_1}\pd_t^{m_1+m_2}(f_i)$, the result follows by using \cite[Lem.~4.1]{MauPer20}.
\end{proof}

\begin{proposition}[{cf.~\cite[Proposition~5.1]{MauPer20}}]\label{Cor.:twistepsilon}
		For $1\leq i,j \leq r$, define $\Upsilon \in \Mat_r (\TT)$ so that $\Upsilon_{ij} := f_i^{(j-1)}(t)$ as in \eqref{E:Upsilon}. Then, for any $\epsilon \in \Gal(K^\sep/K)$ and $g_\epsilon \in \GL_r(\bA_\fp)$ as in \eqref{E:gepsilon}, we have
		\[
		\epsilon\bigl(\mathcal{D}_{\zeta}\left(d_{t,n+1}[\Upsilon]^{(1)}\right)\bigr) = \mathcal{D}_{\epsilon(\zeta)}\left(d_{t,n+1}[g_\epsilon]\right)\cdot \mathcal{D}_{\epsilon(\zeta)}\left(d_{t,n+1}[\Upsilon]^{(1)}\right),
		\]
		and
	\[
		\epsilon\bigl(\mathcal{D}_{\zeta}\left(\Psi_{\rP_n\rho}\right)\bigr) = \mathcal{D}_{\epsilon(\zeta)}\left(\Psi_{\rP_n\rho}\right)\cdot \mathcal{D}_{\epsilon(\zeta)}\left(d_{t,n+1}[g_{\epsilon}]\right)^{-1}.
		\]
	\end{proposition}
	\begin{proof}
Since Frobenius twisting commutes with hyperdifferentiation with respect to $t$, we see  by using \eqref{E:AGFPro} that
for $1\leq j \leq r$ and $0\leq \ell \leq n$, the $(\ell r+j)$-th column of	
	$d_{t,n+1}[\Upsilon^{(1)}]$ is given by the $j$-th Frobenius twist of the $(\ell+1)$-th column of $\cG$. Moreover, by \eqref{ratpro1} we have $\Psi_{\rP_n\rho} = d_{t,n+1}[V]^{-1} d_{t,n+1}[\Upsilon^{(1)}]^{-1}$. Then by using Proposition~\ref{P:AGFz}, the results follow by a straightforward adaptation of the proof of \cite[Prop.~5.1]{MauPer20}.
	\end{proof}

By an abuse of the notation $\mathcal{D}_\zeta$, we consider the homomorphism $\mathcal{D}_\zeta: \TT \otimes_{\bA}\bA_\fp \rightarrow \power{\KK}{X}$ defined by 
\[
\sum_{i}g_i \otimes b_i \mapsto \sum_i \mathcal{D}_\zeta(g_i) \cdot \mathcal{D}_\zeta(b_i).
\]
Note that $\mathcal{D}_\zeta$ is injective on $\TT$, and so it extends to $\LL\otimes_{\TT}(\TT\otimes_{\bA}\bA_{\fp}) \cong \LL\otimes_{\bk}\bk_{\fp}$, that is, to a ring homomorphism
\[
\widetilde{\mathcal{D}}_\zeta: \LL\otimes_{\bk}\bk_{\fp} \rightarrow \laurent{\KK}{X}.
\]

	\begin{proof}[Proof of Theorem~\ref{T:GrHom}]
		Let $S \subseteq K^\per(t)[Y, 1/\det Y]$ denote a finite set of generators of the defining ideal of $Z_{\Psi_{\rP_n\rho}}$, where $K^\per$ is the perfect closure of $K$ in $\KK$. Then, for any $h \in S$, we have $h(\Psi_{\rP_n \rho}) = 0$. 
If $\Psi_{\rP_n \rho} \cdot d_{t,n+1}[g_\epsilon]^{-1} \in Z_{\Psi_{\rP_n\rho}} (\laurent{\KK}{t})$, then by Theorem~\ref{Thm.:diffP} we have $d_{t,n+1}[g_\epsilon]^{-1} \in \Gamma_{\Psi_{\rP_n\rho}}(\bk_\fp)$. Thus, to prove our result, we will show that $h(\Psi_{\rP_n \rho} \cdot  d_{t,n+1}[g_{\epsilon}]^{-1})=0$ for every $h \in S$. The proof follows by a straightforward adaptation of the proof of \cite[Thm.~1.2]{MauPer20}, but for completeness we provide a proof.

For $h \in S$, let $h_{\zeta} \in \laurent{\KK}{X}[Y, 1/\det Y]$ denote its image after mapping its coefficients via the map $\widetilde{\mathcal{D}}_{\zeta}$. Then,
\begin{equation}
\begin{split}
	\widetilde{\mathcal{D}}_\zeta \left(h(\Psi_{\rP_n \rho} \cdot  d_{t,n+1}[g_{\epsilon}]^{-1})\right) &=  h_\zeta\left(\mathcal{D}_{\zeta}(\Psi_{\rP_n\rho})\cdot \mathcal{D}_{\zeta}(d_{t,n+1}[g_\epsilon]^{-1})\right)\\
	& =h_\zeta\left(\epsilon(\mathcal{D}_{\epsilon^{-1}(\zeta)}(\Psi_{\rP_n\rho}))\right)\\
	&= \epsilon\left(h_{\epsilon^{-1}(\zeta)}\left(\mathcal{D}_{\epsilon^{-1}(\zeta)}(\Psi_{\rP_n\rho})\right)\right)\\
		&= \epsilon\left(\widetilde{\mathcal{D}}_{\epsilon^{-1}(\zeta)}(h(\Psi_{\rP_n\rho}))\right)\\
		&=0,
		\end{split}
		\end{equation}
		where the second equality is by Proposition~\ref{Cor.:twistepsilon} and we obtain the third equality since the coefficients of $h$ are in $K^\per(t)$. Since $\zeta$ is an arbitrary root of $\fp$, it follows by \cite[Lem.~5.3]{MauPer20} that $h(\Psi_{\rP_n \rho} \cdot d_{t,n+1}[g_{\epsilon}]^{-1})=0$.
	\end{proof}
	
	\subsection{Elements of \texorpdfstring{$\Gamma_{{\rP_n M_\rho}}$}{GPn}}
	
Let $\rho$ a Drinfeld $\bA$-module of rank $r$ defined over $k^\sep$, and consider the $t$-motive $M_\rho$ associated to $\rho$ (see $\S$\ref{S:motivemodule}). In this subsection, for $n \geq 1$ we study the structure of the Galois group $\Gamma_{\rP_{n} M_\rho}$ of the $n$-th prolongation $t$-motive $\rP_{n} M_\rho$. We let $\End_\cT(\rP_{n} M_\rho)$ denote the ring of endomorphisms of $\rP_{n} M_\rho$ and set $\bK_\rho := \End_\cT(M_\rho)$. If the entries of $\bsm \in \Mat_{r \times 1} (M_\rho)$ form a $\ok(t)$-basis of $M_\rho$, then the entries of $\bsD_n \bsm$ form a $\ok(t)$-basis of $\rP_{n} M_\rho$ as in \eqref{E:basisM}. Given $h \in \End_\cT(\rP_{n} M_\rho)$, let $\rH \in \Mat_{r(n+1)}(\ok(t))$ be such that $h(\bsD_{n} \bsm) = \rH\bsD_{n} \bsm$. Since $h \sigma = \sigma h$ and $\Phi_{\rP_{n} \rho} = d_{t,n+1}[\Phi_\rho]$, we have 
	\begin{align*}
		d_{t,n+1}[\Phi_\rho] \rH = \rH^{(-1)}d_{t,n+1}[\Phi_\rho].
	\end{align*}
	
	\noindent From this, we see that $\sigma$ fixes $d_{t,n+1}[\Psi_\rho]^{-1}\rH d_{t,n+1}[\Psi_\rho]$, and thus $d_{t,n+1}[\Psi_\rho]^{-1}\rH d_{t,n+1}[\Psi_\rho] \in \Mat_{r(n+1)}(\bk)$. We have thus defined the following injective map:
	\[
	\End_\cT(\rP_n M_\rho) \rightarrow \End((\rP_n M_\rho)^B) = \Mat_{r(n+1)}(\bk),
	\]
	\begin{equation} \label{E:End}
		h \mapsto h^B := d_{t,n+1}[\Psi_\rho]^{-1}\rH d_{t,n+1}[\Psi_\rho].
	\end{equation}
	Since the tautological representation $\varpi_n: \Gamma_{\rP_n M_\rho} \rightarrow \GL((\rP_n M_\rho)^B)$ is functorial in $\rP_n M_\rho$ \cite[Thm.~4.5.3]{P08}, for any $\bk$-algebra $\rR$ and $\mu \in \Gamma_{\rP_n M_\rho}(\rR)$, it follows that we have the following commutative diagram:
	\begin{equation} \label{E:Com}
		\begin{tikzcd}
			\rR \otimes_{\bk} (\rP_n M_\rho)^B \arrow{r}{\varpi_n^{\rR}(\mu)} \arrow{d}{1\otimes h^B} & \rR \otimes_{\bk} (\rP_n M_\rho)^B \arrow{d}{1 \otimes h^B} \\
			\rR \otimes_{\bk} (\rP_n M_\rho)^B \arrow{r}{\varpi_n^{\rR}(\mu)} & \rR \otimes_{\bk} (\rP_n M_\rho)^B
		\end{tikzcd}.
	\end{equation}

	\begin{proposition}\label{P:End}
		Given $f \in \bK_\rho$, let $\rF \in \Mat_r(\ok(t))$ satisfy $f(\bsm) = \rF \bsm$. Also, for $n \geq 1$ let $h \in \End_\cT(\rP_n M_\rho)$ be such that $h(\bsD_n \bsm) = \rH \bsD_n \bsm$, where $\rH = (\rH_{ij}) \in \Mat_{r(n+1)}(\ok(t))$ and each $\rH_{ij}$ is an $r \times r$ block for $1 \leq i,j \leq n+1$. Then,
		\begin{enumerate}
			\item[(a)] 
			For $n \geq 1$ there exists $g \in \End_\cT(\rP_n M_\rho)$ such that $g(\bsD_n \bsm) = d_{t,n+1}[\rF]\bsD_n \bsm. $
			\item[(b)] For $0 \leq j \leq n-1$, the matrix $\rH_j := (\rH_{uv})\in \Mat_{r(j+1)} (\ok(t))$, $j+1 \leq u \leq n+1, 1 \leq v \leq j+1$ formed by the lower left $r(j+1) \times r(j+1)$ square of $\rH$ represents an element of $\End_\cT(\rP_j M_\rho)$.
		\end{enumerate}
	\end{proposition}
	\begin{proof}
		For part (a), since $f \sigma = \sigma f$, we have $\Phi_\rho \rF = \rF^{(-1)} \Phi_\rho$. Since multiplication by $\sigma$ on $\rP_n M_\rho$ is represented by $\Phi_{\rP_n \rho} = d_{t,n+1}[\Phi_\rho]$, the proof of (a) follows from the observation that 
		\[
		d_{t,n+1}[\Phi_\rho] d_{t,n+1}[\rF] = d_{t,n+1}[\rF]^{(-1)} d_{t,n+1}[\Phi_\rho].
		\]
		For part (b), using $d_{t,n+1}[\Phi_\rho] \rH = \rH^{(-1)} d_{t,n+1}[\Phi_\rho]$ and the definition of $d$-matrices, we see that for $0 \leq j \leq n-1$,
		\[
		d_{t,j+1}[\Phi_\rho] \rH_j = \rH_j^{(-1)} d_{t,j+1}[\Phi_\rho],
		\]
		and the result follows. 
	\end{proof}

 For any $n \geq 1$ and $0\leq j \leq n-1$, since $\rP_{n-j-1}M_\rho$ is a sub-$t$-motive of $\rP_n M_\rho$, we have a surjective map of affine group schemes over $\bk$,
	\begin{equation}\label{E:SESschemes11gen}
		\boldsymbol{\pi}_{n-j-1}: \Gamma_{\rP_n M_\rho} \twoheadrightarrow \Gamma_{\rP_{n-j-1}M_\rho}.
	\end{equation}
 
	We are now ready to prove the main result of this subsection. 
	
	\begin{theorem} \label{T:Element}
		For each $n \geq 1$ and any $\bk$-algebra $\rR$, an element of $\Gamma_{\rP_n M_\rho}(\rR)$ is of the form 
		\begin{equation}\label{E:elementn}
		\mu_n = \begin{pmatrix}
			\gamma_0 & \gamma_1  & \dots & \gamma_{n-1}& \gamma_{n} \\ 
			&\gamma_0 & \gamma_1 & \ddots & \gamma_{n-1}\\
			&  & \ddots & \ddots &\vdots \\
			&&&\ddots& \gamma_1 \\
			&&&&\gamma_0
		\end{pmatrix},
		\end{equation}
		where for each $0 \leq i \leq n$, $\gamma_i$ is an $r \times r$ block. Furthermore, for $0\leq j \leq n-1$, the matrix $\mu_{n-j-1}$ formed by the upper left $r(n-j) \times r(n-j)$ square is an element of $\Gamma_{\rP_{n-j-1} M_\rho}(\rR)$. In particular, the map $\boldsymbol{\pi}_{n-j-1}^{(\rR)}: \Gamma_{\rP_n M_\rho}(\rR) \twoheadrightarrow \Gamma_{\rP_{n-j-1}M_\rho}(\rR)$ maps an element $\mu_n$ of $\Gamma_{\rP_n M_\rho}(\rR)$ to the matrix $\mu_{n-j-1}$.
	\end{theorem}
	
	\begin{proof}
		Since the prolongation of an $\bA$-finite dual $t$-motive is also an $\bA$-finite dual $t$-motive, by \eqref{E:motivespro} for any $n \geq 1$ and $0 \leq j \leq n-1$ we obtain a short exact sequence of $t$-motives
		\begin{equation}\label{E:motivespro1}
			0 \rightarrow \rP_j M_\rho \xrightarrow{\iota} \rP_{n}M_\rho \xrightarrow{\boldsymbol{\rpr}_{n-j-1}} \rP_{n-j-1}M_\rho \rightarrow 0,
		\end{equation}
		\noindent
		where $\boldsymbol{\rpr}_{n-j-1}(D_im) := D_{i-j-1}m$ for $i > j$ and $\boldsymbol{\rpr}_{n-j-1}(D_im) := 0$ for $i \leq j$ and $m \in M_\rho$, and $\iota$ is the inclusion map. Note that $\rP_0 M_\rho \cong M_\rho$ via $D_0m \mapsto m$ for all $m \in M_\rho$. \par
		
		For any $\bk$-algebra $\rR$, we recall the action of $\Gamma_{\rP_n M_\rho}(\rR)$ on $\rR \otimes_\bk (\rP_n M_\rho)^B$ from \cite[\S4.5]{P08}. 
		Since $\Psi_{\rP_n \rho}= d_{t, n+1}[\Psi_\rho]$, the entries of $\bsu_n := d_{t,n+1}[\Psi_\rho]^{-1} \bsD_n \bsm$ form a $\bk$-basis of $(\rP_n M_\rho)^B$ \cite[Prop. 3.3.9]{P08} and similarly for $0 \leq j \leq n-1$, we have that the entries of $\bsu_{n-j-1} := d_{t,n-j}[\Psi_\rho]^{-1} \bsD_{n-j-1} \bsm$ form a $\bk$-basis of $(\rP_{n-j-1} M_\rho)^B$. For any $\mu_n \in \Gamma_{\rP_n M_\rho}(\rR)$ and any $a_i \in \Mat_{1\times r}(\rR)$, $0 \leq i \leq n$, the action of $\mu_n$ on $(a_0, \dots, a_n) \cdot \bsu_n \in \rR \otimes_{\bk} (\rP_n M_\rho)^B$ is 
		\begin{equation}\label{E:action2}
			(a_0, \dots, a_n) \cdot d_{t,n+1}[\Psi_\rho]^{-1} \bsD_n \bsm\mapsto (a_0, \dots, a_n) \cdot \mu_n^{-1}d_{t,n+1}[\Psi_\rho]^{-1} \bsD_n \bsm.
		\end{equation}

		We first restrict the action of $\mu_n$ to $\rR\otimes_\bk(\rP_j M_\rho)^B$ via the map $\iota$ in \eqref{E:motivespro1}. So, we take $a_0, \dots, a_{n-j-1} = 0$ and set $\mu_n^{-1} := (\rB_{uw}),  1 \leq u, w \leq n+1$ where each $\rB_{uw}$ is an $r \times r$ block. By $\iota$ in \eqref{E:motivespro1}, we see that $\mu_n$ leaves $(\rP_j M_\rho)^B$ invariant and thus the blocks 
		\[\rB_{n-j+v,1} = \rB_{n-j+v,2} = \dots = \rB_{n-j+v,n-j} = \mathbf{0}, \quad \text{for} \, \, 1 \leq v \leq j+1.\]  
		Moreover, since the non-zero $a_i$'s were chosen arbitrarily, we see that the matrix formed by the lower right $r(j+1) \times r(j+1)$ square is an element of $\Gamma_{\rP_j M_\rho}(\rR)$. Varying $j$ from $0$ to $n-1$, we see that $\mu_n^{-1}$ is a block upper triangular matrix and that the matrices formed by the lower right $r(j+1) \times r(j + 1)$ square is an element of $\Gamma_{\rP_j M_\rho}(\rR)$ for each $0 \leq j \leq n-1$. \par
		
		We return to arbitrary $a_i \in \Mat_{1\times r}(\rR)$, $0 \leq i \leq n$. We restrict the action of $\mu_n$ to $\rR~\otimes_\bk~(\rP_{n-j-1} M_\rho)^B$ via the map $\boldsymbol{\rpr}_{n-j-1}$ in \eqref{E:motivespro1}. Through $\boldsymbol{\rpr}_{n-j-1}$, we see that $\mu_n$ leaves $(\rP_{n-j-1} M_\rho)^B$ invariant and so the matrix $\mu_{n-j-1}$ formed by the upper left $r(n-j) \times r(n-j)$ square of $\mu_n$ is an element of $\Gamma_{\rP_{n-j -1} M_\rho}(\rR)$. Varying $j$ from $0$ to $n-1$, we see that the matrices $\mu_{n-j-1}$ formed by the upper left $r(n-j) \times r(n-j)$ square of $\mu_n$ is an element of $\Gamma_{\rP_{n-j-1} M_\rho}(\rR)$ for each $0 \leq j \leq n-1$. \par
		
		Now, we let $h \in \End_\cT(\rP_n M_\rho)$ be such that for $\rH \in \Mat_{r(n+1)}(\ok(t))$ we have $h(\bsD_n \bsm) = \rH \bsD_n \bsm$. Let $\rH := (\rH_{iw})$, where each $(\rH_{iw})$ is an $r\times r$ block. For $0 \leq j \leq n-1$, let $\rH_j := (\rH_{uv})\in \Mat_{r(j+1)} (\ok(t))$, $j+1 \leq u \leq n+1, 1 \leq v \leq j+1$ be the matrix formed by the lower left $r(j+1) \times r(j+1)$ square of $\rH$. Using the definition of $d$-matrices, we see that the matrix formed by the lower left $r(j+1) \times r(j+1)$ square of $d_{t,n+1}[\Psi_\rho]^{-1} \rH d_{t,n+1}[\Psi_\rho]$ is $d_{t,j+1}[\Psi_\rho]^{-1} \rH_j ~d_{t,j+1}[\Psi_\rho]$. By Proposition~\ref{P:End}(b), we have that $d_{t,j+1}[\Psi_\rho]^{-1} \rH_j ~d_{t,j+1}[\Psi_\rho]$ is an element in the image of the natural embedding \eqref{E:End} for the $j$-th prolongation. Thus, by using the commutative diagram \eqref{E:Com} for the $n$-th and the $(n-1)$-th prolongations, we see that since $\mu_n$ is upper triangular, the matrices formed by the lower right $rn \times rn$ square and the upper left $rn \times rn$ square of $\mu_n$ are equal. Now, comparing each $r \times r$ block in this equality, we get the required result. 
	\end{proof}
	
	\subsection{Lower bound on the dimension of \texorpdfstring{$\Gamma_{\rP_nM_\rho}$}{GPn}}
	For this subsection, the reader is directed to Appendix~\ref{DAG} for details about differential algebra and differential algebraic geometry in characteristic $p>0$. We note that the purpose of Appendix~\ref{DAG} is for use in this subsection to prove Theorem~\ref{T:GeqGamma}. For a non-zero prime $\fp \in \bA$, let $\bA_\fp$ denote the completion of $\bA$ at $\fp$, and let $\bk_\fp$ be the fraction field of $\bA_\fp$. By the properties of hyperderivatives (see $\S$\ref{HyperD}) we see that $(\bk_\fp, \pd_t)$, where $\pd_t$ represents hyperdifferentiation with respect to $t$, is a $\pd_t$-field. Using Theorem~\ref{T:Element}, by a slight abuse of notation, we make the choice to let the coordinates of $\Gamma_{{\rP_n M_\rho}}$ be 
	\begin{equation}\label{E:CoordinatesX}
		\bX := \begin{pmatrix}
			\bX_0 & \bX_1 & \dots&\bX_n \\
			&    \ddots     & \ddots&\vdots\\
			&&\ddots&\bX_1\\
			&&&\bX_0
		\end{pmatrix},
	\end{equation}
	\noindent where $\bX_h := ((X_h)_{i, j})$, an $r \times r$ matrix for $0 \leq h \leq n$. We set $\pd_t^\ell(\bX_h):=(\pd_t^\ell((X_h)_{i,j}))$ and 
	\[\vect(\bX_h) := [ (X_h)_{1,1}, \dots, (X_h)_{r,1}, (X_h)_{1,2}, \dots, (X_h)_{r,2}, \dots, \dots, (X_h)_{1,r}, \dots, (X_h)_{r,r}]^\tr,
	\]
	which consists of all entries of $\bX_h$ lined up in a column vector. \par 
	
	Let $0\leq \alpha \leq n$. As in Appendix~\ref{S: Kolchin+}, we define $\bk_\fp\{\bX_0, \dots, \bX_\alpha \}$ to be the $\pd_t$-polynomial ring over $\bk_\fp$ with entries of each $\bX_h$ for $0 \leq h \leq \alpha $ as $\pd_t$-indeterminates. We also define $\bk_\fp\{\bX_0, \dots, \bX_\alpha, 1/\det \bX_0\}$ to be the localization of $\bk_\fp\{\bX_0, \dots, \bX_\alpha \}$ at $\det \bX_0$. We define $\bk_\fp[\bX_0, \dots, \bX_\alpha ]$ to be the usual polynomial ring over $\bk_\fp$ with entries of each $\bX_h$ for $h=0, \dots, \alpha $ as indeterminates, and $\bk_\fp[\bX_0, \dots, \bX_\alpha , 1/\det \bX_0]$ to be the localization of $\bk_\fp[\bX_0, \dots, \bX_\alpha ]$ at $\det \bX_0$.

	We define the centralizer $\Cent_{\GL_r/\bk}(\bK_\rho)$ to be the algebraic group over $\bk$ such that for any $\bk$-algebra $\rR$,
		\[
		\Cent_{\GL_r/\bk}(\bK_\rho)(\rR):= \{\gamma \in \GL_r(\rR) \mid \gamma g = g \gamma \, \, \textup{for all} \, \, g\in \rR\otimes_\bk \bK_\rho \subseteq \Mat_r(\rR)\}.
		\]
		By \cite[Thm.~0.2]{Pink97} and \cite[Thm.~0.2]{PinRut09}, the image $\im\beta_0$ of the homomorphism $\beta_0$ in Theorem~\ref{T:GrHomD} is equal to $\Cent_{\GL_r(\bA_\fp)}(\bK_\rho)$ for all but finitely many primes of $\bA$. Therefore, let $\fp \in \bA$ be a non-zero prime such that $\im \beta_0=\Cent_{\GL_r(\bA_\fp)}(\bK_\rho)$. Then, by \cite[Thm.~3.5.4]{CP12} we see that \begin{equation}\label{E:imageAd}
		\Gamma_{M_\rho}(\bA_\fp)=\Cent_{\GL_r(\bA_\fp)}(\bK_\rho)=~\im \beta_0.
		\end{equation}

	\begin{theorem}\label{T:GeqGamma}
		Fix $n\geq 1$. Let $\rho$ be a Drinfeld $\bA$-module of rank $r$ defined over $k^{\textup{sep}}$ and $\rP_n \rho$ be its associated $n$-th prolongation $t$-module. Let $M_\rho$ and $\rP_{n} M_\rho$ be the $t$-motives corresponding to $\rho$ and $\rP_{n}\rho$ respectively. Let $K_{\rho}$ be the fraction field of $\End(\rho)$ defined as in \eqref{E:Endr} and suppose that $[K_\rho:k]=\bss$. 
		Then, 
		\[
		\dim \Gamma_{{\rP_n M_\rho}} \geq (n+1)\frac{r^2}{\bss}.
		\]
	\end{theorem}
	
	\begin{proof}
		By Theorem~\ref{T:GrHom}, we see that the Zariski closure $\overline{\im \beta_n}^Z$ of $\im \beta_n$ is an algebraic subgroup of $\Gamma_{{\rP_n M_\rho}}$. Therefore, our task is to prove that $\dim (\overline{\im \beta_n}^Z) = (n+1)r^2/\bss$. By \cite[Thm.~3.5.4]{CP12}, we have $\Gamma_{M_\rho} =\Cent_{\GL_r/\bk}(\bK_\rho)$ and $\dim \Gamma_{M_\rho} = r^2/\bss$. Since the defining polynomials of $\Cent_{\Mat_r(\bk)}(\bK_\rho)=\Lie \Gamma_{M_\rho}$ are homogeneous degree one polynomials, let its defining equations be as follows:
		\begin{equation}\label{E:dim0nomat}
			\sum\limits_{i,j=1}^r  (b_u)_{i,j}(X_0)_{i,j}=0, \quad (b_u)_{i,j}\in \bk, u=1, \dots, r^2-r^2/\bss,
		\end{equation}
		which can be written as 
		\begin{equation}\label{E:dim0}
			\bBB \cdot \vect(\bX_0) = {\bf{0}},
		\end{equation}
		\noindent where we set $\bBB$ to be the  $(r^2-r^2/\bss) \times r^2$ matrix of full rank with $(b_u)_{ij}$ as the $u \times ((j-1)r+i)$-th entry. 
		We see that $\rank \bBB=r^2-\dim \Gamma_{M_\rho} = r^2-r^2/\bss$. 
		Therefore, the defining ideal of $\Gamma_{M_\rho}$ is the ideal generated by the entries of $\bBB \cdot \vect(\bX_0)$ in $\bk[\bX_0, 1/\det\bX_0]$, the coordinate ring of $\GL_{r}/\bk$. \par

 For $0\leq \alpha \leq n$, we define a monomial order on $\bk_\fp\{\bX_0, \dots, \bX_{\alpha}\}$ and use the division algorithm \cite[Prop.~1.9]{IimaYoshino} on it. We denote by $\smash{\ZZ_{\geq 0}^{(\infty)}}$ the set of all sequences $(a_1, a_2, a_3, \dots \dots)$ of non-negative integers such that $a_i = 0$ for all but finitely many $i \geq 1$. Any monomial in $\bk_\fp\{\bX_0, \dots, \bX_{\alpha}\}$ can be described uniquely as $\bX^{\bbb}~=~\prod \pd_t^\ell((X_h)_{i,j})^{(b_{h,\ell})_{i,j}}$ for some $\bbb= (\bbb_{0,0}, \bbb_{0,1} \dots, \bbb_{1,0}, \bbb_{1,1}, \dots, \dots, \bbb_{{\alpha},0}, \bbb_{{\alpha},1}, \dots) \in\ZZ_{\geq 0}^{(\infty)}$,  where $\bbb_{h, \ell}=\vect(\left((b_{h,\ell})_{i,j}\right))^\tr= [ (b_{h,\ell})_{1,1}, \dots, (b_{h,\ell})_{r,1}, (b_{h,\ell})_{1,2}, \dots, (b_{h,\ell})_{r,2}, \dots, \dots, (b_{h,\ell})_{1,r}, \dots, (b_{h,\ell})_{r,r}]$ for $0 \leq h \leq {\alpha}$ and $\ell~\in~\ZZ_{\geq0}$ 
			such that $\left((b_{h,\ell})_{i,j}\right)$ is an $r \times r$ matrix and $(b_{h,\ell})_{i.j} = 0$ for all but a finite number of $h, \ell, i, j$. We define a monomial order on $\bk_\fp\{\bX_0, \dots, \bX_{\alpha}\}$ as in \cite[Def.~1.1]{IimaYoshino} in the following way: \begin{itemize}
				\item  we set $\pd_t^{\ell}((X_h)_{1,1})<\dots<\pd_t^{\ell}((X_h)_{r,1})< 
				\dots \dots< \pd_t^{\ell}((X_h)_{1,r}), \dots< \pd_t^{\ell}((X_h)_{r,r})$, 
				\item we set $\pd_t^{\ell}((X_h)_{i_1,j_1})<\pd_t^{\ell+1}((X_h)_{i_2,j_2}),$
				\item we set $\pd_t^{\ell_1}((X_h)_{i_1,j_1}) < \pd_t^{\ell_2}((X_{h+1})_{i_2,j_2})$,
				\item we take the pure lexicographic order defined such that $\bX^\bbb <\bX^{\bf{c}}$ if the left-most nonzero component of $\bbb-\bf{c}$ is negative,
			\end{itemize}
			where $\bbb, \textbf{c} \in \ZZ_{\geq0}^{(\infty)}$, $\ell, \ell_1, \ell_2 \in \ZZ_{\geq0}$,  $i, j, i_1, i_2, j_1, j_2 \in\{0, \dots, r\}$ and $0 \leq h \leq {\alpha}$.

   Let $\mathfrak{I}(\im \beta_0)$ denote the defining $\bk_\fp$-$\pd_t$-ideal of $\im \beta_0$ in $\bk_\fp\{\bX_0, 1/\det \bX_0\}$, and let $\mathfrak{D}(\bBB \cdot \vect(\bX_0))$ denote the $\pd_t$-ideal in $\bk_\fp\{\bX_0, 1/\det \bX_0\}$ generated by the homogeneous degree one polynomials given by the entries of $\bBB \cdot \vect(\bX_0)$. Also, let $\mathfrak{R}(\bBB\cdot \vect(\bX_0))$ denote the radical $\pd_t$-ideal in $\bk_\fp\{\bX_0,  1/\det \bX_0\}$ generated by the entries of $\bBB \cdot \vect(\bX_0)$.
   
     \begin{claim1}\label{Claim:12} 
  We claim that  $\mathfrak{I}(\im\beta_0)=\mathfrak{D}(\bBB\cdot \vect(\bX_0))$.
\end{claim1}

\begin{proof}
By Proposition~\ref{P:raddiff}, we have $\mathfrak{D}(\bBB\cdot \vect(\bX_0))=\mathfrak{R}(\bBB\cdot \vect(\bX_0))$. Clearly, we have $\mathfrak{D}(\bBB\cdot \vect(\bX_0)) \subseteq \mathfrak{I}(\im \beta_0)$. 
   To show that  $\mathfrak{I}(\im\beta_0)\subseteq \mathfrak{D}(\bBB\cdot \vect(\bX_0))$, let $P \in \mathfrak{I}(\im\beta_0) \subseteq \bk_\fp\{\bX_0, 1/\det \bX_0\}$.
Let $(X_0)_{\vartheta_u, \omega_u}$ denote the leading variable of $\sum_{i,j=1}^r  (b_u)_{i,j}(X_0)_{i,j}$ for $u=1, \dots, r^2-r^2/\bss$ with respect to the monomial order above. This means that for $\ell> \vartheta_u, h>\omega_u$ the coefficients $(b_u)_{\ell h}=0$. Moreover, by clearing denominators, we may assume that each $(b_u)_{i,j} \in \bA$. Thus, the defining polynomials of $\im\beta_0$ are now as follows:
\begin{equation}\label{E:dim0nomat1}
\sum_{i=1}^{\vartheta_u}\sum_{j=1}^{\omega_u} (b_u)_{i,j}(X_0)_{i,j}=0, \quad (b_u)_{i,j}\in \bA, u=1, \dots, r^2-r^2/\bss,
		\end{equation}
Since the rank of $\bBB$ is full, we may pick $(b_u)_{i,j}$ so that for each $u=1, \dots, r^2-r^2/\bss-1$
    \[
    (X_0)_{\vartheta_u, \omega_u} < (X_0)_{\vartheta_{u+1}, \omega_{u+1}}.
    \]
By using the division algorithm \cite[Prop.~1.9]{IimaYoshino}, we can write 
    \[
    P = \sum_{u=1}^{r^2-r^2/\bss}\sum_{\ell=0}^{\mu_u} \pd_t^{\ell} \left(\sum_{i=1}^{\vartheta_u}\sum_{j=1}^{\omega_u} (b_u)_{i,j}(X_0)_{i,j}\right)\cdot z_{\ell,u}+ S,
    \]
where $\mu_u$ is the largest number such that $\pd_t^{\mu_u}\left((X_0)_{\vartheta_u, \omega_u}\right)$ occurs as a variable in $P$, each $z_{\ell,u} \in \bk_\fp\{\bX_0, 1/\det \bX_0\}$, and the remainder $S$ is an element of $\mathfrak{I}(\im \beta_0) \setminus \mathfrak{D}(\bBB \cdot \vect(\bX_0))$. Note that the variables $\pd_t^{\ell}\left((X_0)_{\vartheta_u, \omega_u}\right)$ do not occur in $S$.

Suppose that $S \neq 0$. Then, note that there exist $\alpha \geq 0$ and $m \geq \alpha$ such that 
\[
S \in \bk_\fp[\pd_t^\alpha((X_0)_{1,1}), \dots, \pd_t^\alpha((X_0)_{r,r}),  \dots, \dots,   \pd_t^m((X_0)_{1,1}), \dots, \pd_t^m((X_0)_{r,r})],
\]
when $S$ is regarded as a usual polynomial in the variables $\{\pd_t^{\ell}((X_0)_{i,j})\, |\, \alpha\leq \ell \leq m, \, 1\leq i,j, \leq r\}$ over $\bk_\fp$. Suppose $\pd_t^\alpha((X_0)_{v_1, v_2})$ for some $1 \leq v_1, v_2 \leq r$ is the smallest, with respect to the above monomial order, among the variables $\pd_t^\alpha((X_0)_{i,j})$ occurring in $S$. We will show that the coefficients of $S$ as a polynomial in the single variable $\pd_t^\alpha((X_0)_{v_1, v_2})$ over the ring $\bk_\fp[\pd_t^{\alpha}((X_0)_{\gamma_1, \gamma_2}), \pd_t^\ell((X_0)_{i,j}) \, | \, v_1< \gamma_1\leq r, v_2< \gamma_2\leq r, 1\leq i,j\leq r, \alpha< \ell \leq m]$ are in $\mathfrak{I}(\im\beta_0)$ as well.

    We pick $\mathrm{v}>1$ such that $q^\mathrm{v}>m$. Consider $\mathfrak{f} \in \bA_\fp^{q^\mathrm{v}}$ of the form 
    \begin{equation}\label{E:elementformim}
    \mathfrak{f} = \mathfrak{g}\cdot \prod_{u=1}^{r^2-r^2/\bss} (b_u)_{\vartheta_u, \omega_u}^{q^\mathrm{v}},
    \end{equation}
    where $\mathfrak{g} \in \bA_\fp^{q^\mathrm{v}}$, $\mathfrak{g}|_{t=0} = 0$ and each $(b_u)_{\vartheta_u, \omega_u} \in \bA$ is the coefficient of $(X_0)_{\vartheta_u, \omega_u}$ in \eqref{E:dim0nomat1}. Note that $\mathfrak{f}|_{t=0} = 0$. Then, for $\alpha\leq \ell \leq m$ by using the product rule for hyperderivatives and Proposition~\ref{P:hyperderprod} we have
    \[
    \pd_t^\ell(t^{\alpha} \cdot \mathfrak{f} ) =   \pd_t^\ell(t^{\alpha}) \cdot \mathfrak{f} = \begin{cases}
        \mathfrak{f} & \text{ for } \ell=\alpha\\
        0 & \text{ for } \alpha < \ell \leq m
    \end{cases}.
    \]
For $\mathfrak{f}$ as in \eqref{E:elementformim}, consider $ \mathfrak{G}=(\mathfrak{f}_{i,j})\in \Mat_r(\bA_\fp)$ where we set $\mathfrak{f}_{v_1, v_2}=t^\alpha \cdot \mathfrak{f}$, for $(i,j) \neq (v_1, v_2), (\vartheta_u, \omega_u)$, $u=1, \dots, r^2-r^2/s$, we set $\mathfrak{f}_{i,j} =t^{\alpha-1}\cdot \mathfrak{f}$ (or $\mathfrak{f}_{i,j} =0$ in the case $\alpha=0$), and finally we pick the entries $\mathfrak{f}_{\vartheta_u, \omega_u} \in \bA_\fp$ for each $u=1, \dots, r^2-r^2/\bss$ such that 
\[
(b_u)_{\vartheta_u, \omega_u}\cdot  \mathfrak{f}_{\vartheta_u, \omega_u} = -\left(\sum_{i=1}^{\vartheta_u-1}\sum_{j=1}^{\omega_u-1} (b_u)_{i,j}\cdot \mathfrak{f}_{i,j}\right).
\]
Note that each $\mathfrak{f}_{\vartheta_u, \omega_u}|_{t=0} =0$. Then, $\mathfrak{G}$ satisfies \eqref{E:dim0}, that is,
\[
\bBB \cdot \vect(\mathfrak{G}) =0.
\] 
Since $\mathfrak{f}_{i,j}|_{t=0}=0$ for all $1\leq i,j\leq r$, for any $\mathfrak{C} \in \Cent_{\GL_r(\bA_\fp)}(\mathbf{K}_\rho)=~\im \beta_0$, we see that $\mathfrak{C}+\mathfrak{G} \in \GL_r(\bA_\fp)$. Moreover, $\mathfrak{C}+\mathfrak{G}$ satisfies $\bBB \cdot \vect(\mathfrak{C}+\mathfrak{G}) =0$, and so 
\begin{equation}\label{E:CplusD}
\mathfrak{C}+\mathfrak{G}  =(\mathfrak{e}_{i,j}) \in \Cent_{\GL_r(\bA_\fp)}(\mathbf{K}_\rho)=~\im \beta_0.
\end{equation}

To prove $S=0$, we adapt an argument of Maurischat (see \cite[Cor.~6.4]{Maurischat19a}).  For any $\mathfrak{C} = (\mathfrak{c}_{i,j}) \in \Cent_{\GL_r(\bA_\fp)}(\mathbf{K}_\rho)=~\im \beta_0$, consider the polynomial $W_\mathfrak{C}(Y) \in \bk_\fp[Y]$ created from $S$ by making the following assignments to the variables
    \begin{align*}
   \pd_t^{\alpha}((X_0)_{v_1,v_2}) &=\pd_t^{\alpha}(\mathfrak{c}_{v_1, v_2})+Y,\\
   \pd_t^{\alpha}((X_0)_{\gamma_1,\gamma_2}) &=\pd_t^{\alpha}(\mathfrak{c}_{\gamma_1, \gamma_2}), \\
   \pd_t^{\ell}((X_0)_{i,j}) &=\pd_t^{\ell}(\mathfrak{c}_{i,j})
    \end{align*}
    for $v_1< \gamma_1\leq r, v_2< \gamma_2\leq r, 1\leq i,j\leq r$, and $\alpha< \ell\leq m$. Note that for $\mathfrak{C}+\mathfrak{G}$ in \eqref{E:CplusD}
\begin{align*}
\pd_t^\alpha (\mathfrak{e}_{v_1,v_2}) &=\pd_t^\alpha (\mathfrak{c}_{v_1,v_2} + t^\alpha \cdot \mathfrak{f}) = \pd_t^\alpha (\mathfrak{c}_{v_1,v_2})+ \mathfrak{f},\\
\pd_t^\alpha (\mathfrak{e}_{\gamma_1,\gamma_2}) &=\pd_t^\alpha (\mathfrak{c}_{\gamma_1,\gamma_2} + t^{\alpha-1} \cdot \mathfrak{f}) = \pd_t^\alpha (\mathfrak{c}_{\gamma_1,\gamma_2})
\end{align*}
for $v_1< \gamma_1\leq r, v_2< \gamma_2\leq r$, and
\[\pd_t^\ell (\mathfrak{e}_{i,j}) =\pd_t^\ell(\mathfrak{c}_{i,j} + t^{\alpha-1} \cdot \mathfrak{f}) = \pd_t^\ell (\mathfrak{c}_{i,j}),
\]
for $\alpha< \ell\leq m$ and $1\leq i,j\leq r$ such that $(i,j) \neq (\vartheta_u, \omega_u)$ where $u=1, \dots, r^2-r^2/s$. Thus, since the variables $\pd_t^{\ell}\left((X_0)_{\vartheta_u, \omega_u}\right)$ do not occur in $S$, we see that $W_{\mathfrak{C}}(\mathfrak{f})$ is equal to the evaluation of $S$ at the element $\mathfrak{C} + \mathfrak{G} \in \im \beta_0$ and so,
\[
W_{\mathfrak{C}}(\mathfrak{f})=0.
\]
This implies that for all $\mathfrak{C} \in \mathfrak{I}(\im\beta_0)$, the single variable polynomial $W_\mathfrak{C}(Y)$ has infinitely many solutions $\mathfrak{f} \in \bA_\fp^{q^\mathrm{v}}$ of the form \eqref{E:elementformim} and so, $W_\mathfrak{C}(Y)$ is identically $0$. Note that $W_\mathfrak{C}(\pd_t^\alpha((X_0)_{v_1, v_2})-\pd_t^\alpha(\mathfrak{c}_{v_1,v_2}))$ is simply the polynomial in the variable $\pd_t^{\alpha}((X_0)_{v_1,v_2})$ obtained from $S$ by letting
   \begin{align*}
   \pd_t^{\alpha}((X_0)_{\gamma_1,\gamma_2}) =\pd_t^{\alpha}(\mathfrak{c}_{\gamma_1, \gamma_2}), \, \, 
   \pd_t^{\ell}((X_0)_{i,j}) =\pd_t^{\ell}(\mathfrak{c}_{i,j})
    \end{align*}
    for $v_1< \gamma_1\leq r, v_2< \gamma_2\leq r, 1\leq i,j\leq r$, and $\alpha< \ell\leq m$. Since, for all $\mathfrak{C} \in \im\beta_0$,
\begin{equation*}
\begin{split}
0 = W_\mathfrak{C}(\pd_t^\alpha((X_0)_{v_1, v_2})-\pd_t^\alpha(\mathfrak{c}_{v_1,v_2})),
\end{split}
\end{equation*}
this implies that the coefficients of $\pd_t^\alpha((X_0)_{v_1, v_2})$ in the polynomial $S$ also lie in $\mathfrak{I}(\im\beta_0)$. If $S'$ denotes such a coefficient and if $\pd_t^{\varkappa}((X_0)_{a_1, a_2})$ is the smallest variable with respect to the monomial order above occurring in $S'$, then applying to $S'$ the same process above, the coefficients of $\pd_t^{\varkappa}((X_0)_{a_1, a_2})$ in the polynomial $S'$ also lie in $\mathfrak{I}(\im\beta_0)$. Continuing like this, there is a non-zero element of $\bk_\fp$ which is an element of $\mathfrak{I}(\im\beta_0)$, which gives a contradiction to $\im \beta_0 \neq \emptyset$. Thus, $S=0$.
  \end{proof}

Set
		\[
		T :=\mathfrak{R}(\bBB \cdot \vect(\bX_0), \vect(\pd_t^1(\bX_0)-(\bX_1)), \vect(\pd_t^2(\bX_0)-(\bX_2)), \dots, \vect(\pd_t^n(\bX_0)-(\bX_n)))
		\] 
 to be  the radical $\pd_t$-ideal in $\bk_\fp\{\bX_0, \dots, \bX_n,  1/\det \bX_0\}$ generated by the entries of $\bBB \cdot \vect(\bX_0), \vect(\pd_t^1(\bX_0)-(\bX_1)), \vect(\pd_t^2(\bX_0)-(\bX_2)), \dots, \vect(\pd_t^n(\bX_0)-(\bX_n))$, which are homogeneous degree one $\pd_t$-polynomials. Then, Proposition~\ref{P:raddiff} implies that 
		\begin{equation}\label{E:TDdefiningideal}
			T=\mathfrak{D}(\bBB \cdot \vect(\bX_0), \vect(\pd_t^1(\bX_0)-(\bX_1)), \vect(\pd_t^2(\bX_0)-(\bX_2)), \dots, \vect(\pd_t^n(\bX_0)-(\bX_n))),
		\end{equation}
		the $\pd_t$-ideal in $\bk_\fp\{\bX_0, \dots, \bX_n,  1/\det \bX_0\}$ generated by the set of homogeneous degree one $\pd_t$-polynomials given by the entries of $\bBB \cdot \vect(\bX_0), \vect(\pd_t^1(\bX_0)-(\bX_1)), \vect(\pd_t^2(\bX_0)-(\bX_2)), \dots, \vect(\pd_t^n(\bX_0)-(\bX_n))$. 
  
		Let $\mathfrak{I}(\im \beta_n)$ denote the defining $\bk_\fp$-$\pd_t$-ideal of $\im \beta_n$ in $\bk_\fp\{\bX_0,\dots, \bX_n, 1/\det \bX_0\}$. 
  
		\begin{claim1}\label{C:Claim1Diff}
			We claim that $T= \mathfrak{I}(\im \beta_n)$.
		\end{claim1}
		
		\begin{proof}[Proof of Claim~\ref{C:Claim1Diff}]
 By Theorem~\ref{T:GrHom}, clearly $T\subseteq \mathfrak{I}(\im \beta_n)$. 
			To show $\mathfrak{I}(\im \beta_n) \subseteq T$,
	let $F \in \mathfrak{I}(\im \beta_n) \subseteq \bk_\fp\{\bX_0, \dots, \bX_n,  1/\det \bX_0\}$. Note that for $1 \leq h \leq n$, we have $\pd_t^\ell(\pd_t^h((X_0)_{i,j})) < \pd_t^\ell((X_h)_{i,j})$ and so the leading monomial of each $\pd_t^\ell(\pd_t^h((X_0)_{i,j})-(X_h)_{i,j})$ is $\pd_t^\ell((X_h)_{i,j})$. Then, by using the division algorithm \cite[Prop.~1.9]{IimaYoshino} we see that 
			\begin{equation}\label{E:Divalg}
				F= \sum\limits_{i,j=1}^r\sum\limits_{h=1}^n \sum\limits_{\ell=0}^{m_{h,i,j}}  \pd_t^\ell\left(\pd_t^h ((X_0)_{i,j}) - (X_h)_{i,j}\right)\cdot (w_{h,\ell})_{i,j} + H,
			\end{equation}
			where $m_{h,i,j}$ is the largest number such that $\pd_t^{m_{h,i,j}}((X_h)_{i,j})$ occurs as a variable in $F$, each $(w_{h,\ell})_{i,j} \in \bk_\fp\{\bX_0, \dots, \bX_n,  1/\det \bX_0\}$, and the remainder $H=H(\bX_0)$ is an element of $\bk_\fp\{\bX_0,  1/\det \bX_0\}$. Note that for $g_\epsilon$, $\im \beta_n$, and $\im \beta_0$ as in Theorem~\ref{T:GrHom}, there is a surjective map 
			\[
			{\im \beta_n} \twoheadrightarrow {\im \beta_0}
			\]
			given by $d_{t, n+1}[g_\epsilon] \mapsto g_\epsilon$. Moreover we have $F(d_{t, n+1}[g_\epsilon]) = 0$. Since $T\subseteq \mathfrak{I}(\im \beta_n)$ and $\sum_{i,j=1}^r\sum_{h=1}^n \sum_{\ell=0}^{m_{h,i,j}}  \pd_t^\ell\left(\pd_t^h ((X_0)_{i,j}) - (X_h)_{i,j}\right)\cdot (w_{h,\ell})_{i,j} \in~T$, we obtain from \eqref{E:Divalg} that $H(g_\epsilon)=0$. Thus, $H(\bX_0)$ is an element of $\mathfrak{I}(\im \beta_0) = \mathfrak{D}(\bBB \cdot \vect(\bX_0))$.
This proves our claim. Therefore, $\mathfrak{I}(\im \beta_n)= T$. 
\end{proof}

		We are now ready to compute $\overline{\im \beta_n}^Z$. 
		\begin{claim1}\label{C:Claim2Diff}
			The defining equations of $\overline{\im \beta_n}^Z$ are given by
			\begin{equation}\label{E:dimell}
				d_{t,n+1}[\bBB] \cdot \vect([\bX_n, \dots, \bX_0]^\tr) = \mathbf{0}.
			\end{equation}
		\end{claim1}
		
		\begin{proof}[Proof of Claim~\ref{C:Claim2Diff}]
			Based on Lemma~\ref{L:ZariskiKolchin}, we can find the defining equations of $\overline{\im \beta_n}^Z$ if we determine
			\begin{equation}\label{E:idealdelnnodel}
				\bT:= T \cap \bk_\fp[\bX_0, \bX_1, \dots, \bX_n, 1/\det \bX_0].
			\end{equation}
			
			\noindent By the preceding arguments, an element of $F \in T= \mathfrak{I}(\overline{\im \beta_n}^\pd)$ is of the form \eqref{E:Divalg}, where $H \in \mathfrak{D}(\bBB \cdot \vect(\bX_0))\subseteq \bk_\fp\{\bX_0,  1/\det \bX_0\}$. Suppose
			\[
			H = \sum\limits_{i,j=1}^r \sum\limits_{u=1}^{r^2-r^2/\bss}\sum\limits_{\varkappa=0}^{v_u}   c_{u,\varkappa}\cdot \pd_t^{\varkappa}((b_{u})_{i,j}(X_0)_{i,j}),
			\]
			where $v_u \geq 0$ and $c_{u,\varkappa} \in \bk_\fp\{\bX_0, 1/\det \bX_0\}$ for each $1 \leq u \leq r^2-r^2/s$. By the product rule of hyperderivatives, we have $\pd_t^{\varkappa}((b_{u})_{i,j}(X_0)_{i,j}) = \sum_{\alpha=0}^{\varkappa}\pd_t^{\varkappa-\alpha}((b_{u})_{i,j})\cdot \pd_t^{\alpha}((X_0)_{i,j}))$, and so rewriting $F$ we have
		\begin{align*}
				\begin{split}\label{eq:1}
				F=
				&	\sum\limits_{i,j=1}^r \Bigg(\sum\limits_{h=1}^n\sum\limits_{\ell=0}^{m_{h,i,j}}(w_{h,\ell})_{i,j}\cdot \pd_t^\ell\big(\pd_t^h ((X_0)_{i,j}) - (X_h)_{i,j}\big)  \\
					&+\sum\limits_{u=1}^{r^2-r^2/\bss} \Big(\sum\limits_{\varkappa=0}^{v_u}c_{u,\varkappa} \cdot \pd_t^{\varkappa}((b_{u})_{i,j})\cdot (X_0)_{i,j}+\sum\limits_{\varkappa=1}^{v_u}\sum_{\alpha=1}^{\varkappa} c_{u,\varkappa} \cdot \pd_t^{\varkappa-\alpha}((b_{u})_{i,j})\cdot \pd_t^{\alpha}((X_0)_{i,j})\Big)\Bigg),
				\end{split}
			\end{align*}
			where $(w_{h,\ell})_{i,j} \in \bk_\fp\{\bX_0, \dots, \bX_n,  1/\det \bX_0\}$ and $m_{h,i,j} \in \ZZ_{\geq0}$ for $1 \leq h \leq n$, $1 \leq i,j\leq r$. 
   
			Suppose that $F \in \bT \subseteq \bk_\fp[\bX_0, \dots, \bX_n, 1/\det \bX_0]$. Then, since $H \in \bk_\fp\{\bX_0, 1/\det \bX_0\}$, we obtain
					\[
  m_{h,i,j} = 0.
		\]
		Additionally, for each $1\leq i,j\leq r$ we have 
		\[
		\sum_{u=1}^{r^2-r^2/\bss} \sum\limits_{\varkappa=1}^{v_u}\sum_{\alpha=1}^{\varkappa}  c_{u,\varkappa}\cdot \pd_t^{\varkappa-\alpha}((b_{u})_{i,j})\cdot \pd_t^{\alpha}((X_0)_{i,j})) +\sum_{h=1}^n(w_{h,0})_{i,j}\cdot \pd_t^h ((X_0)_{i,j})=0.
		\]
From this, we see that $v_u \leq n$, and for $h > v_u$ we have $(w_{h,0})_{i,j}=0$. Moreover, since $\sum_{\varkappa=1}^{v_u}\sum_{\alpha=1}^{\varkappa}  c_{u,\varkappa}\cdot \pd_t^{\varkappa-\alpha}((b_{u})_{i,j})\cdot \pd_t^{\alpha}((X_0)_{i,j})) = \sum_{h=1}^{v_u} \sum_{\gamma=h}^{v_u} c_{u,\gamma} \cdot \pd_t^{\gamma-h}((b_u)_{i,j}) \cdot \pd_t^h((X_0)_{i,j})$ we have for $1 \leq h \leq v_u$,
		\[(w_{h,0})_{i,j} =-	\sum_{u=1}^{r^2-r^2/\bss}  \sum_{\gamma=h}^{v_u} c_{u,\gamma} \cdot \pd_t^{\gamma-h}((b_u)_{i,j}).
		\]
		Thus, $F \in \bT$ is of the form 
			\[
			\begin{split}
			    F& =\sum_{i,j=1}^r \sum\limits_{u=1}^{r^2-r^2/\bss} \Big(\sum\limits_{\varkappa=0}^{v_u}c_{u,\varkappa} \cdot \pd_t^{\varkappa}((b_{u})_{i,j})\cdot (X_0)_{i,j}+\sum_{h=1}^{v_u} \sum_{\gamma=h}^{v_u} c_{u,\gamma} \cdot \pd_t^{\gamma-h}((b_u)_{i,j}) \cdot (X_h)_{i,j}\Big)\\
			    & = \sum\limits_{u=1}^{r^2-r^2/\bss} \Big(\sum\limits_{\varkappa=0}^{v_u}c_{u,\varkappa} \cdot \pd_t^{\varkappa}(B_u)\cdot \vect(\bX_0) +\sum_{h=1}^{v_u} \sum_{\gamma=h}^{v_u} c_{u,\gamma} \cdot \pd_t^{\gamma-h}(B_u) \cdot \vect(\bX_h)\Big), 				\end{split}
			\] 
			where $B_u$ is the $u$-th row of $\bBB$ and each $c_{u,\varkappa}\in \bk_\fp[\bX_0, 1/\det \bX_0]$. Varying $u$ from $1$ to $r^2-r^2/\bss$ and varying each $v_u$ from $0$ to $n$, we see that the ideal in $\bk_\fp[\bX_0, \dots, \bX_n, 1/\det \bX_0]$ generated by \eqref{E:idealdelnnodel} is the same as the ideal generated by \[\left\{\sum_{h=0}^n\pd_t^{n-h}(B_{u})\cdot \vect(\bX_h), \quad u=1, \dots, r^2-r^2/\bss
			\right\},\]
			which can be written as 
			\begin{equation*}\label{E: XXX}
				d_{t,n+1}[\bBB] \cdot \vect([\bX_n, \dots, \bX_0]^\tr),
			\end{equation*}
			\noindent where we define $\vect([\bX_n, \dots, \bX_0]^\tr) := [(\vect{\bX_n})^\tr, \dots, (\vect{\bX_0})^\tr]^\tr$. Since, by its definition, $d_{t, n+1}[\bBB]$ is a block upper triangular matrix with all diagonal blocks equal to $\bBB$, we have that $\rank d_{t, n+1}[\bBB] \geq (n+1) \cdot \rank \bBB = (n+1) \cdot (r^2-r^2/\bss)$. Also, since $d_{t, n+1}[\bBB]$ is an $(n+1) \cdot (r^2-r^2/\bss) \times (n+1) \cdot r^2$ matrix, we have that $\rank d_{t, n+1}[\bBB] \leq (n+1) \cdot (r^2-r^2/\bss)$ and so $\rank d_{t,n+1}[\bBB] = (n+1) \cdot (r^2-r^2/\bss)$. Since $\rank d_{t,n+1}[\bBB]$ is full, we see that 
			\begin{equation*}\label{E:dimell1}
				d_{t,n+1}[\bBB] \cdot \vect([\bX_n, \dots, \bX_0]^\tr) = \mathbf{0}
			\end{equation*}
			are the defining equations of $\overline{\im \beta_n}^Z$. 
		\end{proof}
		
		Since each $(b_u)_{ij}$ is an element of $\bk$, we see that each entry of $d_{t,n+1}[\bBB]$ is an element of $\bk$ and so, $\overline{\im \beta_n}^Z$ is defined over $\bk$. Moreover, \begin{equation}\label{E:dimzar}
			\dim \overline{\im \beta_n}^Z =(n+1)\cdot r^2-\rank d_{t,n+1}[\bBB]=(n+1)\cdot r^2-(n+~1) \cdot~(r^2-~r^2/\bss)=~(n+1) \cdot r^2/\bss,
		\end{equation}
		which gives the desired result. 
	\end{proof}

	\subsection{Upper bound on the dimension of \texorpdfstring{$\Gamma_{{\rP_n M_\rho}}$}{GPn}}\label{S:Dimension}
	
	Recall from Theorem~\ref{T:Element} that for any $\bk$-algebra $\rR$ and $n \geq 1$, 
	an element $\mu_n$ of $\Gamma_{\rP_n M_{\rho}}(\rR)$ is of the form as in \eqref{E:elementn}.

	Note that by \eqref{E:SESschemes11gen}, we have a short exact sequence of affine group schemes over $\bk$,
	\begin{equation}\label{E:SESschemes11}
		1 \rightarrow Q_n \rightarrow \Gamma_{\rP_n M_\rho} \xrightarrow{\boldsymbol{\pi}_{n-1}} \Gamma_{\rP_{n-1}M_\rho} \rightarrow 1,
	\end{equation}
	where, by Theorem~\ref{T:Element}, $\boldsymbol{\pi}_{n-1}^{(\rR)}: \Gamma_{\rP_n M_\rho}(\rR) \twoheadrightarrow \Gamma_{\rP_{n-1}M_\rho}(\rR)$ maps $\mu_n$ to the matrix $\mu_{n-1}$ formed by the upper left $rn \times rn$ square. Consider 
 \begin{equation}\label{E:Actionprolo1}
		\nu= \begin{pmatrix}
			\Iden_r & 0 & \dots & 0 &  \bv\\
			& \Iden_r & 0 & \dots & 0\\
			&      &  \ddots &\ddots&\vdots\\
			&&&\ddots&0 \\
			&&&&\Iden_r
		\end{pmatrix} \in \GL_{(n+1)r}(\rR),
	\end{equation}
 where $\bv \in \Mat_r(\rR)$. 
Then, an element of $Q_n(\rR)$ is of the form \eqref{E:Actionprolo1}. It can easily be checked that 
	\begin{equation}\label{E:conjaction}
		\mu_n\nu\mu_n^{-1} =  \begin{pmatrix}
			\Iden_r & 0 & \dots & 0 &  \gamma_0\bv\gamma_0^{-1}\\
			& \Iden_r & 0 & \dots & 0\\
			&      &  \ddots &\ddots&\vdots\\
			&&&\ddots&0 \\
			&&&&\Iden_r
		\end{pmatrix}.
	\end{equation}
	Note that $\rP_0 M_\rho$ is simply $M_\rho$ via the map $D_0 m \mapsto m$ for all $m \in M_\rho$ and so, $M_\rho$ is a sub-$t$-motive of $\rP_n M_\rho$. Thus, similarly, by \eqref{E:SESschemes11gen} there is a surjective map of affine group schemes over $\bk$,
	\[
	\boldsymbol{\pi}_0:\Gamma_{\rP_n M_\rho} \twoheadrightarrow \Gamma_{M_\rho},
	\]
	where, by Theorem~\ref{T:Element}, $\boldsymbol{\pi}_0^{(\rR)}:\Gamma_{\rP_n M_\rho}(\rR) \rightarrow \Gamma_{M_\rho}(\rR)$ is the map 
	given by $\mu_n \mapsto \gamma_0$. Thus, via conjugation there is a left action of $\Gamma_{M_\rho}$ on $Q_n$ given by \eqref{E:conjaction}. 

  Set $\bK_\rho := \End_\cT(M_\rho)$ and for a $\bk$-algebra $\rR$, define  
  \[
	\Cent_{\Mat_r/\bk}(\bK_\rho)(\rR):= \{\gamma \in \Mat_r(\rR) \mid \gamma g = g \gamma \, \, \textup{for all} \, \, \\g\in\rR\otimes_\bk \bK_\rho \subseteq \Mat_r(\rR)\}.
	\]
\begin{lemma}\label{L:KerC}
For $n\geq 1$, let $\nu \in Q_n(\rR)$ be as in \eqref{E:Actionprolo1}. Then, 
	\begin{equation*}\label{E:KerC}
		\bv \in \Cent_{\Mat_r/\bk}(\bK_\rho)(\rR).
	\end{equation*}
	 \end{lemma}

  \begin{proof}
		The entries of $\bsu_n = d_{t,n+1}[\Psi_\rho]^{-1} \bsD_n \bsm$ form a $\bk$-basis of $(\rP_n M_\rho)^B$ (see \cite[Prop. 3.3.9]{P08}). Recall the action of $\Gamma_{\rP_n M_\rho}(\rR)$ on $\rR~\otimes_\bk~(\rP_n M_\rho)^B$ from \cite[\S4.5]{P08} (see also \eqref{E:action2}) as follows: for any $\mu_n \in \Gamma_{\rP_n M_\rho}(\rR)$ and any $a_i \in \Mat_{1\times r}(\rR)$, $0 \leq i \leq n$, the action of $\mu_n$ on $(a_0, \dots, a_n) \cdot \bsu_n \in \rR \otimes_{\bk} (\rP_n M_\rho)^B$ is 
		\begin{equation}\label{E:action201}
			\varpi_n^{\rR}(\mu_n):(a_0, \dots, a_n) \cdot d_{t,n+1}[\Psi_\rho]^{-1} \bsD_n \bsm\mapsto (a_0, \dots, a_n) \cdot \mu_n^{-1}d_{t,n+1}[\Psi_\rho]^{-1} \bsD_n \bsm.
		\end{equation}
Given $f \in \bK_\rho$, let $\rF \in \Mat_r(\ok(t))$ satisfy $f(\bsm) = \rF \bsm$. By Proposition~\ref{P:End}(a), for $n \geq 1$ there exists $g \in \End_\cT(\rP_n M_\rho)$ such that $g(\bsD_n \bsm) = d_{t,n+1}[\rF]\bsD_n \bsm$ and so, $d_{t,n+1}[\Psi_\rho]^{-1}d_{t,n+1}[\rF] d_{t,n+1}[\Psi_\rho]=d_{t,n+1}[\Psi_\rho^{-1}\rF\Psi_\rho]$ is an element in the image of the natural embedding \eqref{E:End}.  
		\noindent
  Then by \eqref{E:action201} and the commutative diagram \eqref{E:Com}, we have
  \[
d_{t,n+1}[\Psi_\rho^{-1}\rF\Psi_{\rho}]\nu = \nu d_{t,n+1}[\Psi_\rho^{-1}\rF\Psi_{\rho}]
  \]
  This gives
  \[
 \Psi_\rho^{-1}\rF\Psi_{\rho}\bv =  
 \bv\Psi_\rho^{-1}\rF\Psi_{\rho}
  \]
and the desired result follows.
  \end{proof}

	\begin{theorem}\label{T:1LeqGamma}
		Let $\rho$ be a Drinfeld $\bA$-module of rank~$r$ defined over $k^{\textup{sep}}$ and for $n \geq 1$, let $\rP_{n}\rho$ be its associated $n$-th prolongation $t$-module. Let $M_\rho$ and $\rP_{n} M_\rho$ be the $t$-motives corresponding to $\rho$ and $\rP_{n}\rho$ respectively. Let $K_{\rho}$ be the fraction field of $\End(\rho)$ defined as in \eqref{E:Endr} and suppose that $[K_{\rho}:k]=\bss$. Then $\dim \Gamma_{\rP_{n} M_\rho} \leq (n+1)\cdot r^2/\bss$.
	\end{theorem}

 \begin{remark}
     The author thanks the referee for sharing the ideas of the following proof which is an improvement on the ideas used in a previous proof the author obtained. The author's previous proof required a lemma proving smoothness of $Q_n$. This is no longer required and  has been removed.
 \end{remark}
	
	\begin{proof}
 By Proposition \ref{P:0thpro} and Remark~\ref{R:0thpro}, we see that $[\bK_\rho:\bk]=\bss$ and so, $\Cent_{\Mat_r/\bk}(\bK_\rho)$ is an additive group scheme of dimension $r^2/\bss$ over $\bk$  \cite[Thm.~3.15(3)]{FarbDennis}. 

As $Q_n$ is defined as the kernel in \eqref{E:SESschemes11}, $Q_n$ is a closed subgroup of $\Gamma_{\rP_n M_\rho}$. Consider the closed immersion $\Mat_{r}/\bk \lhook\joinrel\xrightarrow{}  \GL_{(n+1)r}/\bk$ defined by $\bv \mapsto \nu$, where $\nu$ is of the form \eqref{E:Actionprolo1}. Note that $Q_n \subseteq \GL_{(n+1)r}/\bk$ is  isomorphic to its preimage under this closed immersion. Thus, $Q_n$ is closed in $\Mat_{r}/\bk$, and hence closed in $\Cent_{\Mat_{r}/\bk}(\bK_\rho)$  by Lemma~\ref{L:KerC}. This implies that $\dim Q_n \leq \dim \Cent_{\Mat_{r}/\bk}(\bK_\rho) =r^2/\bss$. 
		
		Now, by \eqref{E:SESschemes11} our task is to prove that $\dim Q_n + \dim \Gamma_{\rP_{n-1} M_\rho}\leq (n+1) \cdot r^2/\bss$, which we  show by induction. For the base case $n=1$, since $\dim \Gamma_{M_\rho}=r^2/\bss$ \cite[Thm.~3.5.4]{CP12} we see that $\dim Q_1 + \dim \Gamma_{M_\rho}\leq \dim \Cent_{\Mat_r/\bk}(\bK_\rho)+ \dim \Gamma_{M_\rho}= 2\cdot r^2/\bss$. Suppose we have shown that $\dim \Gamma_{\rP_{n-1} M_\rho} \leq n \cdot r^2/\bss$. By the same argument as in the base case, we obtain $ \dim Q_n + \dim \Gamma_{\rP_{n-1} M_\rho} \leq \dim \Cent_{\Mat_r/\bk}(\bK_\rho) +\dim \Gamma_{\rP_{n-1} M_\rho}= (n+1)\cdot r^2/\bss$.
	\end{proof}
	
	\begin{corollary}\label{T:LeqGamma}
		Let $\rho$ be a Drinfeld $\bA$-module of rank~$r$ defined over $k^{\textup{sep}}$ and for $n \geq 1$, let $\rP_{n}\rho$ be its associated $n$-th prolongation $t$-module. Let $M_\rho$ and $\rP_{n} M_\rho$ be the $t$-motives corresponding to $\rho$ and $\rP_{n}\rho$ respectively. Let $\overline{\im \beta_{n}}^Z$ be the Zariski closure of $\im \beta_{n}$, where $\beta_{n}$ is as in Theorem~\ref{T:GrHom}. Let $K_{\rho}$ be the fraction field of $\End(\rho)$ defined as in \eqref{E:Endr} and suppose that $[K_{\rho}:k]=\bss$. Then $\dim \Gamma_{\rP_{n} M_\rho} = (n+1)\cdot r^2/\bss$ and 
		\[\overline{\im \beta_{n}}^Z/\bk = \Gamma_{\rP_{n} M_\rho}.\]
	\end{corollary}
	\begin{proof}
		We obtain $\dim \Gamma_{\rP_{n} M_\rho} = (n+1)\cdot r^2/\bss$ by combining Theorem~\ref{T:GeqGamma} and Theorem~\ref{T:1LeqGamma}. By \eqref{E:dimzar} we see that $\dim \overline{\im \beta_n}^Z = \dim \Gamma_{\rP_n M_\rho}$. Then, since $\Gamma_{\rP_n M_\rho}$ is connected and smooth by Theorem~\ref{Thm.:diffP}(b), we have $\overline{\im \beta_n}^Z/\bk = \Gamma_{\rP_n M_\rho}$.
	\end{proof}

 \begin{remark}
     By Corollary~\ref{T:LeqGamma}, we see that $\dim Q_n = \dim \Cent_{\Mat_{r}/\bk}(\bK_\rho)$. Since the defining polynomials of $\Cent_{\Mat_{r}/\bk}(\bK_\rho)$ are degree one polynomials, it is connected and smooth. Thus, $Q_n = \Cent_{\Mat_{r}/\bk}(\bK_\rho)$.
 \end{remark}

	\subsection{Algebraic Independence of periods and quasi-periods}
	The following result proves Theorem~\ref{T:Main01}. 
	
	\begin{theorem}\label{T:Main1}
		Fix $n \geq 1$. Let $\rho$ be a Drinfeld $\bA$-module of rank~$r$ defined over $k^{\textup{sep}}$ and $\rP_n \rho$ be its associated $n$-th prolongation $t$-module. Let $K_{\rho}$ be the fraction field of $\End(\rho)$ defined as in \eqref{E:Endr} and suppose that $K_\rho$ is separable over $k$. Let $M_\rho$ and $\rP_n M_\rho$ be the $t$-motives corresponding to $\rho$ and $\rP_n \rho$ respectively. Then, $\trdeg_{\ok} \ok(\Psi_{\rP_n \rho}(\theta)) = (n+1) \cdot r^2/\bss$, where $\bss = [K_\rho:k]$. In particular, 
		\[
		\trdeg_{\ok} \, \ok \bigg(\bigcup\limits_{s=1}^{n} \bigcup\limits_{i=1}^{r-1}\bigcup\limits_{j=1}^r \{\lambda_j, F_{\tau^i}(\lambda_j), \pd_\theta^s(\lambda_j), \pd_\theta^s(F_{\tau^i}(\lambda_j))\}\bigg) = (n+1) \cdot r^2/\bss.
		\]
	\end{theorem}
	
	\begin{proof}
	By Theorem~\ref{T:Tannakian}, we have $\dim \Gamma_{\rP_n M_\rho} = \trdeg_{\ok} \ok(\Psi_{\rP_n \rho}|_{t=\theta})$. Since $\Psi_{\rP_n \rho} = d_{t,n+1}[\Psi_{\rho}]$, the result follows from Theorem~\ref{P:rigidhyper} and Corollary~\ref{T:LeqGamma}.
	\end{proof}

	\section{Hyperderivatives of logarithms and quasi-logarithms}
	\label{S:Hyperlogquasilog}
	
	In this section, we prove Theorem~\ref{T:Main02} (restated as Theorem~\ref{T:Main2}) and Corollary~\ref{C:Main02}. We fix a Drinfeld $\bA$-module $\rho$ of rank $r$ defined over $k^\sep$ and an $\bA$-basis $\{\lambda_1, \dots, \lambda_r\}$ of $\Lambda_\rho$ as in $\S$\ref{S:Produal}. Let $M_\rho$ be the $t$-motive associated to $\rho$ along with a fixed $\ok(t)$-basis $\{m_1, \dots, m_r\} \subseteq M_\rho$, multiplication by $\sigma$ given by $\Phi_\rho$ as in \eqref{E:Phi}, and rigid analytic trivialization $\Psi_\rho$ as in \eqref{E:ratD}. For each $n \geq 0$, let $\rP_n M_\rho$ be the $t$-motive corresponding to the $n$-th prolongation $\rP_n \rho$ of $\rho$ as in $\S$\ref{S:Produal}. Note that $\rP_0 M_\rho$ is simply $M_\rho$ via the map $D_0 m \mapsto m$ for all $m \in M_\rho$. If $\bsm = (m_1, \dots, m_r)^{\tr}$, then a $\ok(t)$-basis of $\rP_n M_\rho$ is given by the entries of $\bsD_n \bsm \in \Mat_{(n+1)r\times 1}(\rP_n M_\rho)$ (see \eqref{E:basisM}) such that multiplication by $\sigma$ is given by $\Phi_{\rP_n \rho} = d_{t,n+1}[\Phi_\rho]$ (see \eqref{E:PhiPro}) with rigid analytic trivialization $\Psi_{\rP_n \rho} = d_{t,n+1}[\Psi_\rho]$ (see \eqref{E:ratpro}). We also set $\bK_\rho:= \End_\cT(M_\rho)$ and let $K_\rho$ denote the fraction field of $\End(\rho)$.
	
	\subsection{\texorpdfstring{$t$}{t}-motives and quasi-logarithms}\label{S:quasilogtmotive}
	Given $u \in \KK$ such that $\Exp_\rho(u) = \alpha \in k^\sep$, let $f_u(t)$ be the Anderson generating function of $\rho$ with respect to $u$ given as in \eqref{E:AGFG}. Then, for $n \geq 1$ we see that the Anderson generating function of $\rP_n\rho$ with respect to $\bsu_n := [u, 0, \dots, 0]^{\tr} \in \KK^{n+1}$ is $\cG_{\bsu_n}(t)= [f_u(t), \pd_t^1(f_u(t)), \dots, \pd_t^{n}(f_u(t))]^{\tr}$ (see~\eqref{E:AGFPro}). Moreover, by \eqref{E:exppro}, $\Exp_{\rP_n\rho}(\bsu_n)~=~[\Exp_\rho(u), 0 \dots, 0]^{\tr}~=~[\alpha, 0 \dots, 0]^{\tr}~\in~(k^\sep)^{n+1}$. We define 
	\[
	\bs_\alpha := \begin{pmatrix}
		-(\kappa_1f_u^{(1)}(t) + \dots + \kappa_{r-1}f_u^{(r-1)}(t) + \kappa_rf_u^{(r)}(t)) \\
		-(\kappa_2^{(-1)}f_u^{(1)}(t) + \dots + \kappa_{r-1}^{(-1)}f_u^{(r-2)}(t) + \kappa_r^{(-1)}f_u^{(r-1)}(t)) \\
		-(\kappa_3^{(-2)}f_u^{(1)}(t) + \dots + \kappa_{r-1}^{(-2)}f_u^{(r-3)}(t) + \kappa_r^{(-2)}f_u^{(r-2)}(t)) \\
		\vdots \\
		-\kappa_r^{(-r+1)}f_u^{(1)}(t)
	\end{pmatrix}^{\tr} \in \Mat_{1 \times r}(\TT),
	\]
	and let $\bh_{\alpha, n} := (\alpha, 0, \dots, 0)  \in \Mat_{1\times (n+1)r}(k^\sep)$. Let $\rF_{\delta}$ be the quasi-periodic function associated to $\rho$-biderivation $\delta$, where $\delta_t= \kappa_1\tau+\dots +\kappa_{r-1}\tau^{r-1}+\kappa_r\tau^r = \rho_t-\theta$. Then, by \cite[Prop. 3.2.2]{BP02} (see also \cite[Prop.~4.3.5(a)]{NPapanikolas21}) we obtain
	\begin{equation}\label{E:RAT111}
	-u+\alpha=\rF_{\delta}(u)=\kappa_1f_u^{(1)}(\theta) + \dots + \kappa_{r-1}f_u^{(r-1)}(\theta) + \kappa_r f_u^{(r)}(\theta).
		\end{equation}

 We now define the pre-$t$-motive $Y_{\alpha,n}$ of dimension $(n+1)r+1$ over $\ok(t)$ such that multiplication by $\sigma$ is given by $\Phi_{\alpha,n} := \begin{psmallmatrix}
		\Phi_{\rP_n \rho} & \mathbf{0} \\
		\bh_{\alpha, n} & 1
	\end{psmallmatrix} \in \Mat_{(n+1)r+1}(\ok[t])$. If we set $\bg_{\alpha, n} := (\bs_\alpha, \pd_t^1(\bs_\alpha), \dots, \pd_t^{n}(\bs_\alpha))$, where the hyperderivatives are taken entry-wise, then we have
	$ \bg_{\alpha, n}^{(-1)} \Phi_{\rP_n \rho} = \bg_{\alpha, n} + \bh_{\alpha, n}$.
	We set $\Psi_{\alpha,n} := \begin{psmallmatrix}
		\Psi_{\rP_n \rho} & \mathbf{0} \\
		\bg_{\alpha,n} \Psi_{\rP_n \rho} & 1
	\end{psmallmatrix} \in \Mat_{(n+1)r+1}(\TT)$ to obtain $\Psi_{\alpha,n}^{(-1)} = \Phi_{\alpha,n} \Psi_{\alpha, n}$. Thus, $Y_{\alpha,n}$ is rigid analytically trivial. The reader may consult \cite[Lem.~5.3.25]{NPapanikolas21} for motivation behind the construction of $\bg_{\alpha,n}$ and $\bh_{\alpha,n}$.
	
	\begin{proposition}[{cf.~\cite[Prop.~6.1.3]{P08}}]
	The pre-$t$-motive $Y_{\alpha,n}$ is a $t$-motive.
	\end{proposition}
	
	\begin{proof}
Set $\cN:= \Mat_{1\times (n+1)r+1}(\ok[t])$ and let $\bse := [\bse_1, \dots, \bse_{(n+1)r+1}]^\tr$ be its standard $\ok[t]$-basis. We give $\cN$ a left $\ok[t, \sigma]$-module structure by setting $\sigma \bse = (t-\theta)\Phi_{\alpha,n}\bse$. Let $\cC$ be the $\bA$-finite dual $t$-motive associated to the Carlitz module $\mathfrak{C}$ (rank 1 Drinfeld $\bA$-module) given by $\mathfrak{C}_t = \theta + \tau$ and let $C := \ok(t) \otimes_{\ok[t]} \cC$ be the corresponding pre-$t$-motive. We obtain the following short exact sequence of $\ok[t, \sigma]$-modules:
		\begin{equation}\label{SEStensor}
			0\rightarrow \cC \otimes_{\ok[t]} \rP_n\cM_\rho\rightarrow \cN \rightarrow \cC \rightarrow 0.
		\end{equation}
		Since $\cC$ and $\cC \otimes_{\ok[t]} \rP_n\cM_\rho$ are finitely generated left $\ok[\sigma]$-modules, it follows from \cite[Prop.~4.3.2]{ABP04} that $\cN$ is free and finitely generated as a left $\ok[\sigma]$-module. Since $\cC\otimes_{\ok[t]}\rP_n\cM_\rho$ is an $\bA$-finite dual $t$-motive, we have $(t-\theta)^{v_1}(\cC\otimes_{\ok[t]}\rP_n\cM_\rho) \subseteq \sigma(\cC\otimes_{\ok[t]}\rP_n\cM_\rho)$ for  some $v_1 \in \NN$. Moreover, $(t-\theta)\cC=\sigma \cC$ and so, by \eqref{SEStensor} we obtain $(t-\theta)^{v_2}\cN \subseteq \sigma\cN$ for $v_2 \in \NN$ sufficiently large. Thus, we see that $\cN$ is an $\bA$-finite dual $t$-motive. Then, it follows from the discussion in \cite[$\S$3.4.10]{P08} that $Y_{\alpha,n}$ is a $t$-motive.
	\end{proof}
	
	\subsection{Non-triviality in \texorpdfstring{$\Ext_{\cT}^1(\mathbf{1}, \rP_n M_\rho)$}{EPn}}\label{S:Ext} 
	We continue with the $t$-motive $Y_{\alpha,n}$ from the previous subsection. Let $\mathbf{1}$ denote the trivial object of the category $\cT$ from $\S$\ref{S:motives}. Note that $Y_{\alpha,n}$ represents a class in $\Ext_\cT^1(\mathbf{1},\rP_n M_\rho)$. Suppose $e \in \End_\cT(M_\rho)$ and let $\rE \in \Mat_r(\ok(t))$ be such that $e(\bsm) = \rE \bsm$. If we set 
	\begin{equation}\label{E:Maprho}
		\bE := \begin{psmallmatrix} \mathbf{0} &\dots &\mathbf{0}& \rE \\
			& \ddots & \ddots&\mathbf{0} \\
			&&\ddots&\vdots \\
			&&&\mathbf{0}
		\end{psmallmatrix} \in \Mat_{(n+1)r}(\ok(t)),
	\end{equation}
	then one checks easily that $\bE$ represents an element $\be$ of $\End_\cT(\rP_n M_\rho)$. For classes $Y_1$ and $Y_2$ in $\Ext_\cT^1(\mathbf{1},\rP_n M_\rho)$, if multiplication by $\sigma$ on suitable $\ok(t)$-bases are represented by $\begin{psmallmatrix}\Phi_{\rP_n \rho} & \mathbf{0}\\ \bv_1 & 1\end{psmallmatrix}$ and $\begin{psmallmatrix}\Phi_{\rP_n \rho} & \mathbf{0}\\ \bv_2 & 1\end{psmallmatrix}$ respectively, then their Baer sum in $\Ext_\cT^1(\mathbf{1},\rP_n M_\rho)$ is achieved by the matrix $\begin{psmallmatrix}\Phi_{\rP_n \rho} & \mathbf{0}\\ \bv_1+\bv_2 & 1\end{psmallmatrix}$. Moreover, we see that multiplication by $\sigma$ on a $\ok(t)$-basis of the pushout $\be_*Y_1$ is represented by $\begin{psmallmatrix}\Phi_{\rP_n \rho} & \mathbf{0}\\ \bv_1 \bE & 1\end{psmallmatrix}$. 

 Note that if $[K_\rho:k]=\bss$, then $\{\lambda_1, \dots , \lambda_r\}$ span a $K_\rho$-vector space of dimension $r/\bss$.
	
	\begin{theorem}\label{T:Trivial}
		Suppose $u_1, \dots, u_w \in \KK$ such that $\Exp_\rho(u_i) = \alpha_i \in k^\sep$ for each $1\leq i\leq w$ and $\dim_{K_\rho} \Span_{K_\rho}(\lambda_1, \dots, \lambda_r, u_1, \dots, u_w) = r/\bss+w$, where $[K_\rho:k]=\bss$. For $n \geq 1$, we let $Y_{i, n}:=Y_{\alpha_i,n}$ be defined as in $\S$\ref{S:quasilogtmotive}. Then, for $e_1, \dots, e_w \in \bK_\rho$, not all zero,  $S:= \be_{1*}Y_{1,n} + \dots + \be_{w*}Y_{w,n}$ is non-trivial in $\Ext_\cT^1(\mathbf{1},\rP_n M_\rho)$, where each $\be_{i} \in \End_\cT(\rP_n M_\rho)$ corresponds to $e_i$ as in \eqref{E:Maprho}.
	\end{theorem}
	
	\begin{proof}
		We adapt the ideas of the proof of \cite[Thm.~4.2.2]{CP12}. For each $1\leq i\leq w$, we let $\bh_{i,n}:=\bh_{\alpha_i, n}$ and $\bg_{i,n}:=\bg_{\alpha_i, n}$. Fix $\rE_i \in \Mat_r(\ok(t))$ so that $e_i(\bsm) = \rE_i \bsm$ for each $1\leq i\leq w$. Then $\be_i(\bsD_n \bsm) = \bE_i \cdot \bsD_n \bsm$, where $\bE_i$ is as in \eqref{E:Maprho} with $\rE_i=\rE$. By choosing an appropriate $\ok(t)$-basis~$\bs$ for $S$, multiplication by $\sigma$ on $\bs$ is represented by \[\Phi_S~:=~\begin{pmatrix}
			\Phi_{\rP_n \rho} & \mathbf{0} \\
			\sum_{i=1}^w\bh_{i,n} \bE_i & 1
		\end{pmatrix} \in \GL_{(n+1)r+1}(\ok(t)), \] and a corresponding rigid analytic trivialization is represented by \[\Psi_S~:=~\begin{pmatrix}
			\Psi_{\rP_n \rho} & \mathbf{0} \\
			\sum_{i=1}^w\bg_{i, n} \bE_i \Psi_{\rP_n \rho}& 1
		\end{pmatrix} \in \GL_{(n+1) r+1}(\LL).\]
		Suppose on the contrary that $S$ is trivial in $\Ext_\cT^1(\mathbf{1},\rP_n M_\rho)$. Then, there exists another $\ok(t)$-basis $\bs'$ of $S$ such that $\sigma \bs'= (\Phi_{\rP_n\rho} \oplus (1))\bs'$, where $\Phi_{\rP_n\rho} \oplus (1)$ is the block diagonal matrix with $\Phi_{\rP_n\rho}$ and $1$ in the diagonal blocks and all other entries are zero. If we let $\gamma = \begin{psmallmatrix} \Iden_{(n+1)r} & \mathbf{0}\\ \gamma_0 \dots \gamma_n&1\end{psmallmatrix} \in \GL_{(n+1)r+1}(\ok(t))$,
		where $\gamma_j := (\gamma_{j1},\dots ,\gamma_{jr})$ for each $0\leq j \leq n$ be the matrix such that $\bs' := \gamma\bs$, then we obtain 
		\begin{equation}\label{E:Spe1}
			\gamma^{(-1)}\Phi_S = (\Phi_{\rP_n\rho} \oplus (1)) \gamma.
		\end{equation}
		Note from \cite[Proof of Prop.~3.4.5]{P08} that all denominators of entries of $\gamma$ are in $\bA$ and so in particular, for each $0\leq j \leq n$ the entries of  $\gamma_{j}$ are regular at $t = \theta, \theta^q, \theta^{q^2}, \dots$\,. Using $\Phi_{\rP_n\rho} = d_{t,n+1}[\Phi_\rho]$, the $((n+1)r+1,(n-j)r+1)$-th entry of \eqref{E:Spe1} for each $1 \leq j \leq n$ is 
		\begin{equation*}\label{E:entry}
			\sum\limits_{h=0}^{n-j}\gamma_{h, r}^{(-1)} \pd_t^{n-j-h}\left((t-\theta)/\kappa_r^{(-r)}\right) = \gamma_{n-j,1}, 
		\end{equation*}
		and the $((n+1)r+1,n r+1)$-th entry is 
		\begin{equation*}\label{E:EntryL}
			\sum\limits_{h=0}^n \gamma_{h, r}^{(-1)} \pd_{t}^{n-h}\left((t-\theta)/\kappa_{r}^{(-r)}\right) +\sum\limits_{i=1}^w \alpha_i (\rE_i)_{11} = \gamma_{n,1}.
		\end{equation*}
		For each $0 \leq j \leq n$, applying $(-1)^j \pd_t^j(\cdot)$ to the $((n+1)r+1,(n-j)r+1)$-th entry and then adding them (we also use the product rule of hyperderivatives and the property $\pd_t^v\pd_t^w (f(t)) = \binom{v+w}{v} \pd_t^{v+w}(f(t))$), we obtain
		\begin{equation}\label{E:pdgammadelta}
			\sum_{j=0}^n(-1)^j\pd_t^j(\gamma_{n-j,r})^{(-1)}(t-\theta)/\kappa_r^{(-r)} + \sum_{i=1}^w\alpha_i(\rE_i)_{11}=\sum_{j=0}^n(-1)^j\pd_t^j(\gamma_{n-j,1}).
		\end{equation}
		Specializing both sides of this equation at $t=\theta$, we obtain 
		\begin{equation}\label{E:hyperE11}
			\sum_{j=0}^n(-1)^j\pd_t^j(\gamma_{n-j,1})(\theta)=\sum_{i=1}^w\alpha_i(\rE_i)_{11}(\theta).
		\end{equation}
		By \eqref{E:Spe1}, we also have $(\gamma \Psi_S)^{(-1)} = (\Phi_{\rP_n\rho} \oplus (1)) (\gamma \Psi_S)$ and so by \cite[$\S$4.1.6]{P08}, for some 
		$\delta = \begin{psmallmatrix} \Iden_{(n+1)r} & \mathbf{0}\\ \delta_0 \dots \delta_n&1\end{psmallmatrix}~\in~\GL_{(n+1)r+1}(\bk)$
		where $\delta_j~:=~(\delta_{j1},\dots ,\delta_{jr})$ for each $0\leq j \leq n$,  we have 
		\begin{equation}\label{E:gammadeltapro}
			\gamma \Psi_S = (\Psi_{\rP_n\rho} \oplus (1))\delta.
		\end{equation}
		Since $\Psi_{\rP_n\rho} = d_{t,n+1}[\Psi_\rho]$, by applying to \eqref{E:gammadeltapro} the same methods applied on \eqref{E:Spe1} to obtain \eqref{E:pdgammadelta}, it follows that 
		\begin{equation}\label{ExtProof2}
			\sum_{j=0}^n(-1)^j\pd_t^j(\gamma_{n-j}) + \sum_{i=1}^w\bs_i \rE_i = \sum_{j=0}^n(-1)^j\pd_t^j(\delta_{n-j})\Psi_\rho^{-1},
		\end{equation}
	where for $\pd_t^j(\gamma_{n-j})$ and $\pd_t^j(\delta_{n-j})$, the hyperderivatives are taken entry-wise. Since for each $1\leq i \leq w$ the first entry of $\bs_{i}(\theta)$ is $u_i-\alpha_i$ by \eqref{E:RAT111}, using \cite[Prop.~4.1.1(b)]{CP12} and specializing both sides of \eqref{ExtProof2} at $t=\theta$, we see that 
		\begin{equation*}\label{E:Spe}
			\sum\limits_{j=0}^n(-1)^j\pd_t^j(\gamma_{n-j,1})(\theta) + \sum_{i=1}^w(u_i-\alpha_i)(\rE_i)_{11}(\theta) =-\sum\limits_{m=1}^r \sum\limits_{j=0}^n(-1)^j\pd_t^j(\delta_{n-j,m})(\theta)\lambda_m,
		\end{equation*}
		and so from \eqref{E:hyperE11} we have
		\[
		\sum\limits_{m=1}^r \sum\limits_{j=0}^n(-1)^j\pd_t^j(\delta_{n-j,m})(\theta)\lambda_m+ \sum_{i=1}^w(\rE_i)_{11}(\theta)u_i=0.
		\]
		Since $e_1, \dots, e_w$ are not all zero, $\rE_i$ is nonzero for some $1 \leq i \leq w$. Moreover, by Proposition \ref{P:0thpro} we see that $\bK_\rho \cong K_\rho$ and so $\rE_i$ is invertible. By \cite[Prop.~4.1.1(b),(c)]{CP12} we get $(\rE_i)_{11}(\theta)\in K_\rho^\times$ and thus we get a contradiction to the assumption that $\{u_1, \dots, u_w\}$ is $K_\rho$-linearly independent from each other and is $K_\rho$-linearly independent from $\{\lambda_1, \dots, \lambda_r\}$.
	\end{proof}

	\subsection{Construction of the \texorpdfstring{$t$}{t}-motives \texorpdfstring{$Y$}{Y} and \texorpdfstring{$N_n$}{N_n}}\label{S:YandN}
	
	In this subsection, we construct a $t$-motive that is suitable for the investigation of the hyperderivatives of logarithms and quasi-logarithms of the Drinfeld $\bA$-module $\rho$, and the study of its Galois group. Suppose that we have $u_1, \dots, u_w  \in \KK$ with $\Exp_\rho(u_i) = \alpha_i \in k^\sep$ for each $1 \leq i \leq w$. For $n \geq 0$, we let $\bh_{\alpha_i}:=\bh_{\alpha_i, n}, \bg_{\alpha_i}:=\bg_{\alpha_i, n}, Y_{i, n}:=Y_{\alpha_i,n}, \Phi_{i, n}:=\Phi_{\alpha_i,n}$ and $\Psi_{i, n}:=\Psi_{\alpha_i,n}$ defined as in $\S$\ref{S:quasilogtmotive}. The matrix $\Psi_n := \oplus_{i=1}^w \Psi_{i, n}$ is a rigid analytic trivialization for $Y_n := \oplus_{i=1}^w Y_{i, n}$. \par

	Define the $t$-motive $N_n$ such that multiplication by $\sigma$ on a $\ok(t)$-basis is given by $\Phi_{N_n} \in \GL_{(n+1)rw+1}(\ok(t))$ along with rigid analytic trivialization $\Psi_{N_n} \in \GL_{(n+1)rw+1}(\TT)$ such that
	\begin{equation}\label{E:PhiPsiN}
		\Phi_{N_n} := \begin{pmatrix} \Phi_{\rP_n \rho} & & &\\ & \ddots &&\\ &&\Phi_{\rP_n \rho} &\\ \bh_{\alpha_1} & \dots & \bh_{\alpha_w} & 1 \end{pmatrix}, \quad \text{and} \quad \Psi_{N_n} := \begin{pmatrix} \Psi_{\rP_n \rho} & & &\\ & \ddots &&\\ &&\Psi_{\rP_n \rho} &\\ \bg_{\alpha_1}\Psi_{\rP_n \rho} & \dots & \bg_{\alpha_w}\Psi_{\rP_n \rho} & 1 \end{pmatrix}.
	\end{equation}

Similar to $n=0$ case \cite[\S5.1]{CP12}, $N_n$ is an extension of $\mathbf{1}$ by $(\rP_n M_\rho)^w$ which is a pullback of the surjective map $Y_n \twoheadrightarrow \mathbf{1}^w$ and the diagonal map $\mathbf{1} \rightarrow \mathbf{1}^w$. Thus, the two $t$-motives $Y_n$ and $N_n$ generate the same Tannakian subcategory of $\cT$ and hence the Galois groups $\Gamma_{Y_n}$ and $\Gamma_{N_n}$ are isomorphic. For any $\bk$-algebra $\rR$, an element of $\Gamma_{N_n}(\rR)$ is of the form 
	\begin{equation}\label{E:emelemtNn}
	\nu = \begin{pmatrix} \mu &&& \\ & \ddots && \\ && \mu &\\ \bv_1 &\dots & \bv_w & 1\end{pmatrix},
	\end{equation}
	where $\mu \in \Gamma_{\rP_n M_\rho}(\rR)$ and for each $1 \leq i \leq w$, we have $\bv_i = (\bv_{i,1}, \dots, \bv_{i,n+1})$ such that $\bv_{i,h} \in \GG_a^r(\rR)= \Mat_{1\times r}(\rR)$ for each $0 \leq h \leq n$. Since $(\rP_n M_\rho)^w$ is a sub-$t$-motive of $N_n$, we have the following short exact sequence of affine group schemes over $\bk$,
	\begin{equation}\label{E:SESlog}
		1 \rightarrow X_n \rightarrow \Gamma_{N_n} \xrightarrow{\pi_n} \Gamma_{\rP_n M_\rho} \rightarrow 1,
	\end{equation}
	where $\pi_n^{(\rR)}:\Gamma_{N_n}(\rR)\rightarrow \Gamma_{\rP_nM_\rho}(\rR)$ is the map $\nu \mapsto \mu$ (cf.~\cite[p.138]{CP12}). 
	It can be checked directly that via conjugation, \eqref{E:SESlog} gives an action of any $\mu \in \Gamma_{\rP_n M_\rho}(\rR)$ on 
	\[\bv = \left( \begin{array}{cccc} \Iden_{(n+1)r} & &&\\ & \ddots & &\\ &&\Iden_{(n+1)r} & \\ \bu_1 & \dots & \bu_w & 1 \end{array} \right) \in X_n(\rR)
	\]
	given by
	\begin{equation}\label{E:Actionlog}
		\nu \bv \nu^{-1}=\left( \begin{array}{cccc} \Iden_{(n+1)r} & &&\\ & \ddots & &\\ &&\Iden_{(n+1)r} & \\ \bu_1\mu^{-1} & \dots & \bu_w \mu^{-1}  & 1 \end{array} \right).
	\end{equation}

 For $n \geq 0$, recall from \eqref{E:basisM} that if the entries of $\bsm \in \Mat_{r\times 1}(M_\rho)$ form a $\ok(t)$-basis of $M_\rho$, then the entries of $\bsD_n \bsm$ form a $\ok(t)$-basis of $\rP_n M_\rho$. Let $[\bsD_n \bsm^\tr, y]^\tr$ be a $\ok(t)$-basis of $N_n$. Then, the entries of $\Psi_{N_n}^{-1}[\bsD_n \bsm^\tr, y]^\tr$ form a $\FF_q(t)$-basis of $N_n^B$ \cite[Prop. 3.3.9]{P08}. By construction, $\rP_jM_\rho$ is a sub-$t$-motive of $N_n$ for each $0\leq j \leq n-1$ and we have a short exact sequence of $t$-motives
		\begin{equation}\label{E:ProSESLog}
			0 \rightarrow \rP_jM_\rho \xrightarrow{\iota} N_n \xrightarrow{\boldsymbol{\mathrm{Pr}}_{n-j-1}} N_{n-j-1} \rightarrow 0,
		\end{equation}
		where $\boldsymbol{\mathrm{Pr}}_{n-j-1}(D_hm) := D_{h-j-1}m$ for $h > j$, $\boldsymbol{\mathrm{Pr}}_{n-j-1}(D_hm) := 0$ for $h \leq j$ and $m \in M_\rho$, and $\boldsymbol{\mathrm{Pr}}_{n-j-1}(x) := x$ for $x \in N_n / \rP_n M_\rho$. Thus, as $t$-motives 
		$	N_n / \rP_jM_\rho \cong N_{n-j-1}$
		and so, $N_{n-j-1}$ is an object in the Tannakian category $\cT_{N_n}$. Therefore, we have a surjective map of affine group schemes $\Gamma_{N_n} \twoheadrightarrow \Gamma_{N_{n-j-1}}$. We now determine this surjective map. For any $\bk$-algebra $\rR$, we recall the action of $\Gamma_{N_n}(\rR)$ on $\rR \otimes_\bk N_n^B$ from \cite[\S4.5]{P08} as follows: for any $\nu_n \in \Gamma_{N_n}(\rR)$, $b \in \rR$ and $a_h \in \Mat_{1\times r}(\rR)$ where $0 \leq h \leq n$, the action of $\nu_n$ on $(a_0, \dots, a_n, b) \cdot \Psi_{N_n}^{-1}[\bsD_n \bsm^\tr, y]^\tr \in R \otimes_{\bk} N_n^B$ is 
		\begin{equation}\label{E:action22}
			(a_0, \dots, a_n, b) \cdot \Psi_{N_n}^{-1}[\bsD_n \bsm^\tr, y]^\tr \mapsto (a_0, \dots, a_n, b) \cdot \nu_n^{-1} \Psi_{N_n}^{-1}[\bsD_n \bsm^\tr, y]^\tr.
		\end{equation}
		\noindent
		Note that $\Psi_{N_n}^{-1}[\bsD_n \bsm^\tr, y]^\tr =[(d_{t,n+1}[\Psi_\rho]^{-1} \bsD_n \bsm)^\tr, -\bg_{\alpha_1} D_n\bsm +y]^\tr$ by the definition of $\Psi_{N_n}$ (see \eqref{E:PhiPsiN}). We restrict the action of $\nu_n$ to $\rR~\otimes_\bk~N_{n-j-1}^B$ via the map $\boldsymbol{\mathrm{Pr}}_{n-j-1}$ in \eqref{E:ProSESLog}. Note that an element of $\Gamma_{N_{n}}(\rR)$ is of the form $\begin{psmallmatrix}\mu_n & {\bf 0}\\ \bsw_n & 1\end{psmallmatrix}$, where $\mu_n \in \Gamma_{\rP_n M_\rho}(\rR)$ and $\bsw_n =(w_0, \dots, w_n)$ such that each $w_h \in \GG_a^{r}(\rR)= \Mat_{1\times r}(\rR)$. Through $\boldsymbol{\mathrm{Pr}}_{n-j-1}$, we see that $\nu_n$ leaves $N_{n-j-1}^B$ invariant and so for $\nu_n = \begin{psmallmatrix}\mu_n & {\bf 0}\\ \bsw_n & 1\end{psmallmatrix} \in \Gamma_{N_n}(\rR)$, we obtain 
  \begin{equation}\label{E:NntoN}
		\nu_{n-j-1} = \begin{pmatrix}\mu_{n-j-1} & {\bf 0}\\ \bsw_{n-j-1} & 1\end{pmatrix} \in \Gamma_{N_{n-j-1}}(\rR),
		\end{equation}
		where $\mu_{n-j-1}$ is the matrix formed by the upper left $r(n-j)\times r(n-j)$ square of $\mu_n$ and $\bsw_{n-j-1} =(w_0, \dots, w_{n-j-1})$. Note that by Theorem~\ref{T:Element}, we have $\mu_{n-j-1} \in \Gamma_{\rP_{n-j-1} M_\rho}(\rR)$. 
		Thus, the surjective map $\Xi_{n-j-1}:\Gamma_{N_n} \rightarrow \Gamma_{N_{n-j-1}}$ is given by (cf.~\cite[Prop.~3.1.2]{CP11}) \begin{equation}\label{E:mapNtoN_0}
			\Xi_{n-j-1}^{(\rR)}:\nu_n \mapsto \nu_{n-j-1}.
		\end{equation}
	
	\begin{lemma}\label{L:defined1}
		Let $n \geq 1$. If $K_\rho$ is separable over $k$, then $X_n$ in \eqref{E:SESlog} is $\bk$-smooth. 
	\end{lemma}
	
	\begin{proof}
	We adapt the ideas of the proof of \cite[Prop.~4.1.2]{CP11} and the proof of a lemma from a preliminary version of \cite{CP12} (Lemma~5.1.3: arXiv:1005.5120v1). By \cite[Cor. 12.1.3]{Springer98} it suffices to show that for $n \geq 1$, the induced tangent map $\rd\pi_n$ at the identity is surjective onto $\Lie \Gamma_{\rP_{n} M_\rho}$. We prove this for $w=1$ as the argument used in this case can be applied in a straightforward manner to prove the arbitrary $w$ case. We leave this task to the reader. Since $\bK_\rho$ is separable over $\bk$ (by hypothesis, Proposition~\ref{P:0thpro}, and  Remark~\ref{R:0thpro}), we see from \cite[Cor.~3.5.6]{CP12} and \cite[p.61 Problem~14]{Waterhouse} that through conjugation by some $\rJ \in \GL_r({\bk^\sep})$, we have an isomorphism
		\[
		\Gamma_{M_\rho} \times_{\bk} \bK_\rho \xrightarrow{\cong} \prod_{i=1}^{\bss} (\GL_{r/\bss}/\bK_\rho)_i,
		\]
		where 
		\[
		\prod_{i=1}^{\bss} (\GL_{r/\bss}/\bK_\rho)_i:= \left\{ \begin{pmatrix} \GL_{r/\bss} &&\\ & \ddots&\\ &&\GL_{r/\bss} \end{pmatrix} \right\} ,
		\]
		and $(\GL_{r/\bss}/\bK_\rho)_i$ is the canonical embedding of $\GL_{r/\bss}/\bK_\rho$ into the $i$-th diagonal block matrix of $\GL_{r}/\bK_\rho$. Making a change of basis, we obtain
		\[
		\Gamma_{M_\rho} \times_{\bk} \overline{\bk} \xrightarrow{\cong} \prod_{i=1}^{\bss} (\GL_{r/\bss}/\overline{\bk})_i.
		\] 
  For $n\geq 1$, it follows that via conjugation by $d_{t,n+1}[\rJ] \in \GL_{(n+1)r}({\bk^\sep})$ on $\Gamma_{\rP_{n} M_\rho}$, we obtain $\overline{\Gamma}_{\rP_n M_\rho}$, an algebraic subgroup of $\GL_{(n+1)r}/\overline{\bk}$, such that there is an isomorphism 
		\begin{equation}\label{E:abar}
   \Gamma_{\rP_n M_\rho} \times_{{\bk}}~\overline{\bk}\xrightarrow{\cong} \overline{\Gamma}_{\rP_n M_\rho}.
		\end{equation}
 		Let $\left(\oplus_{i=1}^w d_{t,n+1}[\rJ]\right) \oplus (1) \in \GL_{(n+1)rw+1}(\bk^\sep)$ be the block diagonal matrix with $d_{t,n+1}[\rJ]$ in the first $w$ diagonal blocks and $1$ in the last diagonal, and all other entries are zero. Then, via conjugation by $\left(\oplus_{i=1}^w d_{t,n+1}[\rJ]\right) \oplus (1)$ on $\Gamma_{N_n}$ we obtain $\overline{\Gamma}_{N_n}$ such that we have an isomorphism $\Gamma_{N_n} \times_\bk \overline{\bk} \cong \overline{\Gamma}_{N_n}$. Moreover, $\overline{\Gamma}_{N_n}$ is an algebraic subgroup of $\GL_{(n+1)rw+1}/\overline{\bk}$ such that $\overline{\pi}_n: \overline{\Gamma}_{N_n} \rightarrow \overline{\Gamma}_{\rP_n M_\rho}$ induced by $\pi_n$ in \eqref{E:SESlog} is surjective. Thus, we are reduced to proving that the induced tangent map $\rd\overline{\pi}_n: \Lie\overline{\Gamma}_{N_n} \rightarrow \Lie\overline{\Gamma}_{\rP_n M_\rho}$ is surjective. 
		
First we determine $\overline{\Gamma}_{\rP_n M_\rho}$. Recall $\bX$, the coordinates of $\Gamma_{\rP_n M_\rho}$ from \eqref{E:CoordinatesX}. Since $d_{t,n+1}[\rJ]$ and its inverse are block upper triangular matrices, similar to $\bX$ we make the choice to let the coordinates of $\overline{\Gamma}_{\rP_n M_\rho}$ be
		\[
		\bY := \begin{pmatrix}
			\bY_0 & \bY_1 & \dots&\bY_n \\
			        & \bY_0&\ddots&\vdots\\
			&&\ddots&\bY_1\\
			&&&\bY_0
		\end{pmatrix},
		\]
		where $\bY_h := ((Y_h)_{ij})$, an $r \times r$ matrix for $0 \leq h \leq n$. Then, by construction we have $\bX = d_{t, n+1}[\rJ] \bY d_{t, n+1}[\rJ]^{-1}$ and so for each $0 \leq w \leq n$, we obtain 
		\[\bX_w=  \sum_{\substack{w_1+w_2=w \\ w_1, w_2 \geq 0}} \sum\limits_{h=0}^{w_1} \pd_t^{w_{1}-h}(\rJ)\cdot \bY_{w_2}\cdot (\pd_t^{h}(\rJ))^{-1},
		\]
		where the hyperderivatives are taken entry-wise. Then, we have 
		\begin{equation*}
			\begin{split}
				\vect(\bX_w) &= \sum_{\substack{w_1+w_2=w \\ w_1, w_2 \geq 0}} \sum\limits_{h=0}^{w_1} \left([(\pd_t^{h}(\rJ))^{-1}]^\tr \otimes  \pd_t^{w_1-h}(\rJ)\right) \cdot \vect(\bY_{w_2})\\
				&=  \sum_{\substack{w_1+w_2=w \\ w_1, w_2 \geq 0}} \pd_t^{w_1}\left((\rJ^{-1})^\tr \otimes \rJ\right) \cdot \vect(\bY_{w_2}),
			\end{split}
		\end{equation*}
		where we obtain the first equality by using properties of the Kronecker product and the second equality by further applying the product rule for hyperderivatives. This implies
		\begin{equation}\label{E:XandY}
			\vect( [\bX_n, \dots, \bX_0]^\tr) = d_{t, n+1}[(\rJ^{-1})^\tr \otimes \rJ] \cdot \vect( [\bY_n, \dots, \bY_0]^\tr) 		\end{equation}
		where we set $\vect( [\bX_n, \dots, \bX_0]^\tr) := [(\vect{\bX_n})^\tr, \dots, (\vect{\bX_0})^\tr]^\tr$,
  and we further define $\vect( [\bY_n, \dots, \bY_0]^\tr)$ similarly. For $0\leq i\leq n$, let $\overline{\bk}[\bY_0, \dots, \bY_i, 1/\det \bY_0]$ denote the localization of $\overline{\bk}[\bY_0, \dots, \bY_i]$ at $\det \bY_0$. Then, by \eqref{E:dimell}, Corollary~\ref{T:LeqGamma}, and \eqref{E:XandY}, the defining ideal of 
  $\overline{\Gamma}_{\rP_n M_\rho}$ 
  via the isomorphism \eqref{E:abar}  is the ideal in $\overline{\bk}[\bY_0, \dots, \bY_n, 1/\det \bY_0]$ generated by the entries of
		\begin{equation}\label{E:defeqsep}
			d_{t, n+1}[\bBB\cdot \left((\rJ^{-1})^\tr \otimes \rJ\right)] \cdot \vect( [\bY_n, \dots, \bY_0]^\tr). 
	\end{equation}

   		It is clear by observing $\prod_{i=1}^{\bss} (\GL_{r/\bss}/\overline{\bk})_i$ that for $\bY_0= ((Y_0)_{i,j})$, the defining ideal of $\prod_{i=1}^{\bss} (\GL_{r/\bss}/\overline{\bk})_i$ is the ideal in $\overline{\bk}[\bY_0, 1/\det \bY_0]$ generated by 
			\begin{equation}\label{E:dim00}
				\left\{(Y_0)_{i,j} \, \mid \,  (i,j) \neq (ur/\bss+v_1, ur/\bss +v_2),  0 \leq u \leq \bss-1 \, \textup{and} \, 1 \leq v_1, v_2 \leq r/\bss\right\}.
			\end{equation}
   		\noindent Moreover, by \eqref{E:dim0} and \eqref{E:XandY}, the defining ideal of $\prod_{i=1}^{\bss} (\GL_{r/\bss}/\overline{\bk})_i$ is also generated by the entries of 
			\begin{equation}\label{E:dim0sep}
				(\bBB\cdot ((\rJ^{-1})^\tr \otimes \rJ)) \cdot \vect(\bY_0). 
			\end{equation}
      \noindent By \eqref{E:dim00}, in the defining ideal of $\prod_{i=1}^{\bss} (\GL_{r/\bss}/\overline{\bk})_i$, there are no linear relations among 
   \begin{equation}\label{E:dim00not}
			\left\{(Y_0)_{i,j}  \, \mid \,  (i,j) = (ur/\bss+v_1, ur/\bss +v_2), 0 \leq u \leq {\bss}-1 \, \textup{and} \, 1 \leq v_1, v_2 \leq  r/\bss\right\}.
			\end{equation}
			Since \eqref{E:dim0sep} also generate the defining ideal of $\prod_{i=1}^{\bss} (\GL_{r/\bss}/\overline{\bk})_i$, 
we see that the entries of $\bBB\cdot ((\rJ^{-1})^\tr \otimes \rJ)$ that give linear relations among the variables in \eqref{E:dim00not} are all zero. Therefore, the hyperderivatives of these entries are also all zero. Using this and using \eqref{E:dim00}, for $\gamma \in \prod_{i=1}^{\bss} (\GL_{r/\bss}/\overline{\bk})_i$ and for $0\leq \ell \leq n$, we see that 
   
			\begin{equation}\label{E:hyperelementpro1}
				\pd_t^\ell\left(\bBB\cdot ((\rJ^{-1})^\tr \otimes \rJ)\right)\cdot \gamma 
				= {\bf{0}}.
			\end{equation}
			
			\noindent Moreover, by \eqref{E:dim0sep}, for $1\leq h \leq n$, the defining ideal of $\prod_{i=1}^s (\Mat_{r/\bss}/\overline{\bk})_i$ is the ideal in $\overline{\bk}[\bY_h]$ generated by the entries of
			\begin{equation*}\label{E:dim0sep11}
				(\bBB\cdot ((\rJ^{-1})^\tr \otimes \rJ)) \cdot \vect(\bY_h),
			\end{equation*}
   and similar to \eqref{E:hyperelementpro1}, for $\gamma' \in \prod_{i=1}^{\bss} (\Mat_{r/\bss}/\overline{\bk})_i$ and for $0\leq \ell \leq n$, we see that
   \[
   \pd_t^\ell\left(\bBB\cdot ((\rJ^{-1})^\tr \otimes \rJ)\right)\cdot \gamma' 
				= {\bf{0}}.
   \]
			Therefore, for all $\gamma_0\in \prod\limits_{i=1}^s(\GL_{r/\bss}/\overline{\bk})_i$ and $\gamma_h \in \prod\limits_{i=1}^s(\Mat_{r/\bss}/\overline{\bk})_i$ where $1\leq h \leq n$, we have 
			\begin{equation*}
				d_{t,n+1}[\bBB\cdot ((\rJ^{-1})^\tr \otimes \rJ)] \cdot \left( [\gamma_n, \dots, \gamma_0]^\tr\right) ={\bf{0}}.
			\end{equation*}
			Thus, by \eqref{E:defeqsep} we have
			\begin{equation}\label{E:Gammasepn}
				\overline{\Gamma}_{\rP_n M_\rho} = 
   \left\{ 
			\begin{pmatrix}
				\begin{array}{c|c|c|c} 
					\gamma_0 & \gamma_1 & \dots & \gamma_{n}\\
					\hline 
					& \gamma_0 &\ddots &\vdots\\
					\hline 
					&&\ddots&\gamma_1\\
					\hline 
					&&&\gamma_0
				\end{array} 
			\end{pmatrix}: \begin{aligned} \gamma_0 \in \prod\limits_{i=1}^s(\GL_{r/\bss}/\overline{\bk})_i, \, &\gamma_h \in \prod\limits_{i=1}^s(\Mat_{r/\bss}/\overline{\bk})_i,\\  & 1 \leq h \leq n
			\end{aligned}
			\right\}, 
		\end{equation}
  where for each $i$, $(\GL_{r/\bss}/\overline{\bk})_i$ and $(\Mat_{r/\bss}/\overline{\bk})_i$ are the canonical embeddings of $\GL_{r/\bss}/\overline{\bk}$ and $\Mat_{r/\bss}/\overline{\bk}$ respectively into the $i$-th diagonal block matrices of $\GL_{r}/\overline{\bk}$ and $\Mat_{r}/\overline{\bk}$.

We are now ready to prove that the induced tangent map $\rd\overline{\pi}_n: \Lie\overline{\Gamma}_{N_n} \rightarrow \Lie\overline{\Gamma}_{\rP_n M_\rho}$ is surjective. Let $w=1$ and consider the short exact sequence of linear algebraic groups
		\begin{equation}\label{E:overlineSES1}
			1\rightarrow \overline{X}_n\rightarrow \overline{\Gamma}_{N_n} \xrightarrow{\overline{\pi}_n} \overline{\Gamma}_{\rP_n M_\rho} \rightarrow 1.
		\end{equation}
  First suppose $n =1$. Then, by \eqref{E:Gammasepn}, 
 \begin{equation}\label{E:Gammasep1}
			\overline{\Gamma}_{\rP_1M_\rho} = \left\{ 
			\begin{pmatrix}
					\gamma_0 & \gamma_1\\
					{\bf{0}} & \gamma_0
			\end{pmatrix} \, : \, \gamma_0 \in \prod\limits_{i=1}^s(\GL_{r/\bss}/\overline{\bk})_i, \, \gamma_1 \in \prod\limits_{i=1}^s(\Mat_{r/\bss}/\overline{\bk})_i
			\right\}, 
		\end{equation}
  and by \eqref{E:emelemtNn},
\begin{equation}\label{E:GammasepNn1}
  \overline{\Gamma}_{N_1} \subseteq \left\{ \begin{pmatrix}
\gamma_0 & \gamma_1 & {\bf{0}}\\
{\bf{0}} &\gamma_0 & {\bf{0}}\\
\bz_0 & \bz_1 &1 \\ 
\end{pmatrix}  \, : \, \begin{pmatrix}
					\gamma_0 & \gamma_1\\
					{\bf{0}} & \gamma_0
			\end{pmatrix} \in \overline{\Gamma}_{\rP_1M_\rho}, \, \bz_{0}, \bz_1 \in \GG_a^r
				\right\}.
  \end{equation}
From $\overline{\pi}_1$, we see that $\overline{X}_{1}$ is contained in the $2r$-dimensional additive group
		\[
		G:= \left\{ 
		\begin{pmatrix}
			\Iden_{r/\bss} &   &  &\\ 
			&\ddots &  & \\ 
			&&\Iden_{r/\bss} &\\
			\bv_1 &\dots &\bv_{2\bss}&1
		\end{pmatrix} \, : \, \bv_i \in \GG_a^{r/\bss}
		\right\},
		\]
		where we call $\bv_1, \dots, \bv_{2\bss}$ the coordinates of $G$. We see that via conjugation,  $\overline{X}_1(\overline{\bk})$ has a $\overline{\Gamma}_{\rP_1 M_\rho}(\overline{\bk})$-module structure coming from \eqref{E:overlineSES1} (see \eqref{E:Actionlog}). Using \eqref{E:Gammasep1} and this module structure, one checks easily that there is a natural decomposition $\overline{X}_1(\overline{\bk}) = \prod_{i=1}^{2\bss}W_i$ such that each $W_i$ is either zero or $\overline{\bk}^{r/\bss}$. Fix any $1 \leq i \leq \bss$. For any $\xi_i \in \GL_{r/\bss}(\overline{\bk})$, we let 
		\[
		\overline{\xi}_i = \begin{psmallmatrix} 
			\Iden_{r/\bss} &   &   &   && \mathbf{0}  &   &&&& \\ 
			&\ddots & &&  &&\ddots&&&&\\ 
			&& \xi_i &&& && \mathbf{0} &&&\\
			&& &\ddots&& &&&\ddots&&\\
			&&&&\Iden_{r/\bss}& &&&&\mathbf{0} &\\
			\hline
			&&&&&\Iden_{r/\bss} &&&& &\\ 
			&&&&&& \ddots &&& &\\ 
			&&&&&   & &\xi_i&& &\\ 
			&&&&&    &&& \ddots & &\\ 
			&&&&& && & &\Iden_{r/\bss}&\\
			\bu_1&\dots & \bu_i& \dots&\bu_s& \bu_{s+1} & \dots& \bu_{s+i} &\dots&\bu_{2s}&1
		\end{psmallmatrix}
		\in \overline{\Gamma}_{N_1}(\overline{\bk})
		\]
		be an arbitrary element, which by \eqref{E:overlineSES1}  for $n=1$ and \eqref{E:Gammasep1} is a pre-image of the matrix formed by the upper left $2r\times 2r$ square of $\overline{\xi_i}$ under the map $\overline{\pi}_1$. For each $j\neq i$ with $1 \leq j \leq \bss$, we claim that if $\bu_j \neq {\bf{0}}$ and $\bu_{\bss+j} \neq {\bf{0}}$, then $W_j = W_{\bss+j}= \overline{\bk}^{r/\bss}$. To prove this claim, assuming that $\bu_j \neq {\bf{0}}$ and $\bu_{\bss+j} \neq {\bf{0}}$ we pick $\delta_j \in \GL_{r/\bss}(\overline{\bk})$ so that $\bu_j\delta_j-\bu_j \neq {\bf{0}}$ and $\bu_{\bss+j}\delta_j-\bu_{\bss+j} \neq {\bf{0}}$, and let $\overline{\delta}_j \in \overline{\Gamma}_{N_1}(\overline{\bk})$ be such that 
		\[
		\overline{\pi}_1(\overline{\delta}_j) = \begin{psmallmatrix} 
			\Iden_{r/\bss} &   &   &   && \mathbf{0}  &   &&& \\ 
			&\ddots & &&  &&\ddots&&&\\ 
			&& \delta_j &&& && \mathbf{0} &&\\
			&& &\ddots&& &&&\ddots&\\
			&&&&\Iden_{r/\bss}& &&&&\mathbf{0} \\
			\hline
			&&&&&\Iden_{r/\bss} &&&& \\ 
			&&&&&& \ddots &&& \\ 
			&&&&&   & &\delta_j&& \\ 
			&&&&&    &&& \ddots & \\ 
			&&&&& && & &\Iden_{r/\bss}
		\end{psmallmatrix}
		\in \overline{\Gamma}_{\rP_1 M_\rho}(\overline{\bk}).\]
		Then one checks directly that $\overline{\delta}_j^{-1}\overline{\xi}_i\overline{\delta}_j\overline{\xi}_i^{-1}$ is an element of $\overline{X}_1(\overline{\bk})$ and its $\bv_j$ and  $\bv_{\bss+j} $ coordinate vectors respectively are $\bu_j\delta_j-\bu_j$ and $\bu_{\bss+j}\delta_j-\bu_{\bss+j}$, and so it follows that $W_j =W_{\bss+j}= \overline{\bk}^{r/\bss}$. Therefore, multiplying $\overline{\xi}_i$ by a suitable element of $\overline{X}_1(\overline{\bk})$ we get an element of the form
		\begin{equation}\label{E:MatrixGammaprime}
		\overline{\xi}_i'= \begin{psmallmatrix} 
			\Iden_{r/\bss} &   &   &   && {\bf 0}  &   &&&& \\ 
			&\ddots & &&  &&\ddots&&&&\\ 
			&& \xi_i &&& && \mathbf{0} &&&\\
			&& &\ddots&& &&&\ddots&&\\
			&&&&\Iden_{r/\bss}& &&&&{\bf 0} &\\
			\hline
			&&&&&\Iden_{r/\bss} &&&& &\\ 
			&&&&&& \ddots &&& &\\ 
			&&&&&   & &\xi_i&& &\\ 
			&&&&&    &&& \ddots & &\\ 
			&&&&& && & &\Iden_{r/\bss}&\\
			\mathbf{0}&\dots & \bu_i& \dots&\mathbf{0}&\mathbf{0} & \dots& \bu_{s+i} &\dots&\mathbf{0}&1
		\end{psmallmatrix}
		\in \overline{\Gamma}_{N_1}(\overline{\bk}).
		\end{equation}
		For any $\bbb_i \in \Mat_{r/\bss}(\overline{\bk})$, by using a similar method as above where we take an element of the form $\overline{\delta}_j$, we obtain an element of the form
		\begin{equation}\label{E:Matrixbprime}
		\overline{\bbb}_i'= \begin{psmallmatrix} 
			\Iden_{r/\bss} &   &   &   && {\bf 0}  &   &&&& \\ 
			&\ddots & &&  &&\ddots&&&&\\ 
			&& \Iden_{r/\bss} &&& && \bbb_i &&&\\
			&& &\ddots&& &&&\ddots&&\\
			&&&&\Iden_{r/\bss}& &&&&{\bf 0} &\\
			\hline
			&&&&&\Iden_{r/\bss} &&&& &\\ 
			&&&&&& \ddots &&& &\\ 
			&&&&&   & &\Iden_{r/\bss}&& &\\ 
			&&&&&    &&& \ddots & &\\ 
			&&&&& && & &\Iden_{r/\bss}&\\
			\mathbf{0}&\dots & \bw_i& \dots&\mathbf{0}&\mathbf{0} & \dots& \bw_{s+i} &\dots&\mathbf{0}&1
		\end{psmallmatrix}
		\in \overline{\Gamma}_{N_1}(\overline{\bk}),
		\end{equation}
		which is a pre-image of the matrix formed by the upper left $2r \times 2r$ square of $\overline{\bbb}_i'$ under the map $\overline{\pi}_1$. Let $\overline{H}_{1,i}$ be the Zariski closure inside $\overline{\Gamma}_{N_1}$ of the subgroup generated by all $ \overline{\xi}_i'$ with $\xi_i$ running over all elements of $\GL_{r/\bss}(\overline{\bk})$ and all $\overline{\bbb}_i'$ with $\bbb_i$ running over all elements of $\Mat_{r/\bss}(\overline{\bk})$. For each $1\leq i \leq \bss$, let
  \begin{equation}\label{E:GammaP1i}
  \left(\overline{\Gamma}_{\rP_1 M_\rho}/\overline{\bk}\right)_i := \left\{ \begin{pmatrix} \gamma_0 & \gamma_1 \\ \mathbf{0} & \gamma_0\end{pmatrix} \, : \, \gamma_0 \in (\GL_{r/\bss}/\overline{\bk})_i; \, \gamma_1 \in (\Mat_{r/\bss}/\overline{\bk})_i \right\}.
  \end{equation}
  Note that $\dim \overline{H}_{1,i}\leq 2r^2/\bss^2+2r/\bss$. 
		
		First suppose that $\dim \overline{H}_{1,i}= 2r^2/\bss^2+2r/\bss$. Then, we could simply take $\overline{\xi}_i'$ and $\overline{\bbb}_i'$ so that $\bu_i, \bu_{\bss+i}, \bw_i$ and $\bw_{\bss+i}$ are zero. Taking the Zariski closure $\overline{S}_{1,i}$ inside $\overline{\Gamma}_{N_1}$ of the subgroup generated by all such $\overline{\xi}_i'$ and $\bbb_i'$ with $\xi_i$ and $\bbb_i$ running over all elements of $\GL_{r/\bss}(\overline{\bk})$ and $\Mat_{r/\bss}(\overline{\bk})$ respectively, we obtain  
		\begin{equation}\label{E:Xsplitsmoothlog}
			\overline{S}_{1,i} = \left\{ \begin{pmatrix} \nu_i & \mathbf{0} \\ \mathbf{0} & 1 \end{pmatrix} \, : \,\nu_i \in (\overline{\Gamma}_{\rP_1 M_\rho}/\overline{\bk})_i \right\}.  	\end{equation}
   Thus, $\overline{\Gamma}_{N_1}$ contains a copy of $(\overline{\Gamma}_{\rP_1 M_\rho}/\overline{\bk})_i$ and so, restricting $\rd \overline{\pi}_1$ to $\Lie \overline{S}_{1,i}$, we obtain a surjection onto $\Lie (\overline{\Gamma}_{\rP_1 M_\rho}/\overline{\bk})_i$. As we vary all $1\leq i\leq \bss$, the surjection of $\rd \overline{{\pi}}_{1}$ follows. \par

		Next, suppose that $\dim \overline{H}_{1,i}< 2r^2/\bss^2+2r/\bss$. Then, via $\overline{{\pi}}_1$ we have a short exact sequence
		\[
		1 \rightarrow \overline{Q}_{1,i} \rightarrow \overline{H}_{1,i} \xrightarrow{{\overline{\pi}}_{1,i}} (\overline{\Gamma}_{\rP_1M_\rho}/\overline{\bk})_i \rightarrow 1,
		\] 
		where $\dim \overline{Q}_{1,i}< 2r/\bss$ and $\overline{Q}_{1,i}$ is contained in an additive subgroup of $G$ whose $\bv_j$ coordinate vector is zero for all $j\neq i, \bss+i$, that is, 
  \begin{equation}\label{E:LQ1i}
\overline{Q}_{1,i} \subseteq 
			\left\{{\begin{psmallmatrix} 
				\Iden_{r/\bss} &   &   &   && \mathbf{0}  &   &&&& \\ 
				&\ddots & &&  &&\ddots&&&&\\ 
				&& \Iden_{r/\bss} &&& && \mathbf{0} &&&\\
				&& &\ddots&& &&&\ddots&&\\
				&&&&\Iden_{r/\bss}& &&&&\mathbf{0} &\\
				\hline
				&&&&&\Iden_{r/\bss} &&&& &\\ 
				&&&&&& \ddots &&& &\\ 
				&&&&&   & &\Iden_{r/\bss}&& &\\ 
				&&&&&    &&& \ddots & &\\ 
				&&&&& && & &\Iden_{r/\bss}&\\
				{\bf{0}}&\dots & \bv_i& \dots&{\bf{0}}& {\bf{0}} & \dots& \bv_{\bss+i} &\dots&{\bf{0}}&1
			\end{psmallmatrix}} : \begin{aligned} 
				&\bv_i, \bv_{\bss+i} \in \GG_a^{r/\bss}, \, \\
				& \bv_i = (\bv_{i,1}, \dots, \bv_{i,r/\bss}),\, \\
								& \bv_{\bss+i} = (\bv_{\bss+i,1}, \dots, \bv_{\bss+i,r/\bss})
			\end{aligned}
			\right\}.
			\end{equation}

	\begin{claim3}\label{CL:CLaim1}
  For $\overline{Q}_{1,i}$  if some entry of the $\bv_{i}$ coordinate vector is non-zero or $\dim \overline{Q}_{1,i}\neq r/\bss$,
			 then $\dim \overline{Q}_{1,i} = 0$.
		\end{claim3}
		
		\begin{proof}[Proof of Claim~\ref{CL:CLaim1}]
			We follow the argument of the proof of \cite[Lem. 4.1.1]{CP11}. Suppose $\dim \overline{Q}_{1,i} = m$, where $1\leq m < 2r/\bss$. Note that $\overline{Q}_{1,i}$ is a vector group. 
If $\bv_{i,j}, \bv_{\bss+i,j}\neq 0$ for all $1 \leq j \leq r/\bss$, let $\mu \in \overline{Q}_{1,i}(\overline{\bk})$ such that all the entries of $\mu$ in the $\bv_i$ coordinate vector are non-zero. For $a \in \overline{\bk}$, $a\neq 0,1$ and $1\leq \ell \leq r/\bss$, pick $\eta_{\ell}, \varkappa_{\ell} \in \overline{H}_{1,i}(\overline{\bk})$ such that $\overline{\pi}_{1,i}(\eta_\ell), \overline{\pi}_{1,i}(\varkappa_\ell) \in (\overline{\Gamma}_{\rP_1 M_\rho}(\overline{\bk}))_i$ where
			\begin{equation}\label{E:Matrixelement}
			\overline{\pi}_{1,i}(\eta_\ell) =
			{\begin{psmallmatrix} 
				\Iden_{r/\bss} &   &   &   && \mathbf{0}  &   &&& \\ 
				&\ddots & &&  &&\ddots&&&\\ 
				&& \mathfrak{a}_{\ell} &&& && \mathbf{0} &&\\
				&& &\ddots&& &&&\ddots&\\
				&&&&\Iden_{r/\bss}& &&&&\mathbf{0} \\
				\hline
				&&&&&\Iden_{r/\bss} &&&& \\ 
				&&&&&& \ddots &&& \\ 
				&&&&&   & &\mathfrak{a}_{\ell}&& \\ 
				&&&&&    &&& \ddots & \\ 
				&&&&& && & &\Iden_{r/\bss}\\
			\end{psmallmatrix}}, \, \, 
			\text{for} \,  
		\mathfrak{a}_{\ell} := \begin{psmallmatrix}
				1 &&&&\\
				&\ddots &&&\\
				&&a &&\\
                && &\ddots&\\
				&&&&1\\
			\end{psmallmatrix},
			\end{equation}
   \[
			\overline{\pi}_{1,i}(\varkappa_{\ell}) =
			{\begin{psmallmatrix} 
				\Iden_{r/\bss} &   &   &   && \mathbf{0}  &   &&& \\ 
				&\ddots & &&  &&\ddots&&&\\ 
				&& \Iden_{r/\bss} &&& && \mathfrak{b}_{\ell} &&\\
				&& &\ddots&& &&&\ddots&\\
				&&&&\Iden_{r/\bss}& &&&&\mathbf{0} \\
				\hline
				&&&&&\Iden_{r/\bss} &&&& \\ 
				&&&&&& \ddots &&& \\ 
				&&&&&   & &\Iden_{r/\bss}&& \\ 
				&&&&&    &&& \ddots & \\ 
				&&&&& && & &\Iden_{r/\bss}\\
			\end{psmallmatrix}}, \, \, 
			\text{for} \,  
		\mathfrak{b}_{\ell} :=  \begin{psmallmatrix}
				0 &&&&\\
				&\ddots &&&\\
				&&a &&\\
                && &\ddots&\\
				&&&&0\\
			\end{psmallmatrix},
			\]
  where $\mathfrak{a}_{\ell}$ and $\mathfrak{b}_{\ell}$ are $r/\bss \times r/\bss$, and $a$ is in the $\ell$-th diagonal entries. One checks directly that the $2r/\bss$ vectors $\eta_\ell^{-1}\mu \eta_\ell$, $\varkappa_\ell^{-1}\mu \varkappa_\ell$ where $1\leq \ell \leq m$ are $\overline{\bk}$-linearly independent in $\overline{Q}_{1,i}(\overline{\bk})$, which contradicts the assumption $\dim \overline{Q}_{1,i}=m<2r/\bss$. Thus, $\bv_{i,u}=0$ for some $1\leq u \leq r/\bss$. Since $m \neq 0$, at least one of $\bv_{i,j}, \bv_{\bss+i,j}$ for some $1 \leq j \leq r/\bss$ is non-zero, say $\bv_{i,v}$ or $\bv_{\bss+i,v}$.

  Let $\textbf{P}_{u,v}$ be the permutation matrix obtained by switching the $((i-1)r/\bss+u)$-th column and the $((i-1)r/\bss+v)$-column of the $r\times r$ identity matrix. Pick $\gamma \in \overline{H}_{1,i}(\overline{\bk})$ such that
			\begin{equation}\label{E:permutaionuvmatrixG}
			\overline{\pi}_{1,i}(\gamma) = \begin{pmatrix}
				\textbf{P}_{u,v} & \\
				&\textbf{P}_{u,v}
			\end{pmatrix}\in (\overline{\Gamma}_{\rP_1 M_\rho}(\overline{\bk}))_i.
			\end{equation}

If $\bv_{i,v}$ is non-zero, then since $\gamma^{-1}\overline{Q}_{1,i}\gamma \subseteq \overline{Q}_{1,i}$ we get a contradiction to $\bv_{i,u}=0$. Therefore, $\dim \overline{Q}_{1,i}=0$.

Next suppose $\bv_{\bss+i,v}$ is non-zero but $\bv_{i,j}=0$ for all $1\leq j\leq r/\bss$. Then by hypothesis, $m<r/\bss$. If $\bv_{\bss+i,j}\neq 0$ for all $1 \leq j \leq r/\bss$, let $\vartheta \in \overline{Q}_{1,i}(\overline{\bk})$ such that all the entries of $\vartheta$ in the $\bv_{\bss+i}$ coordinate vector are non-zero. Then, one checks directly that for $\eta_{\ell}$ as in \eqref{E:Matrixelement}, the $r/\bss$ vectors $\eta_\ell^{-1}\vartheta \eta_\ell$ are $\overline{\bk}$-linearly independent in $\overline{Q}_{1,i}(\overline{\bk})$, which contradicts the assumption $\dim Q_{1,i}=m<r/\bss$. Thus, $\bv_{\bss+i,u}=0$ for some $1\leq u \leq r/\bss$. Then, since $\bv_{\bss+i,v}$ is non-zero and $\gamma^{-1}Q_{1,i}\gamma \subseteq \overline{Q}_{1,i}$ for $\gamma$ as in \eqref{E:permutaionuvmatrixG}, we get a contradiction to $\bv_{\bss+i,u}=0$. Therefore, $\dim \overline{Q}_{1,i}=0$.
		\end{proof}

		\begin{claim3}\label{CL:Claim2}
			If $\dim \overline{Q}_{1,i}=0$, then $\rd \overline{{\pi}}_{1,i}: \Lie \overline{H}_{1,i} \rightarrow \Lie (\overline{\Gamma}_{\rP_1M_\rho}/\overline{\bk})_i$ is surjective.
		\end{claim3}
		
		\begin{proof}[Proof of Claim~\ref{CL:Claim2}]
			To prove that $\rd \overline{{\pi}}_{1,i}$ is surjective, we follow the argument of the proof of \cite[Prop.~4.1.2]{CP11}. We let the coordinates of $\overline{H}_{1,i}$ 
			be as follows:
			\begin{equation}\label{E:coordinatesX1}
				\bZ_1 :=\begin{pmatrix}
					\mathcal{Z}_0 & \mathcal{Z}_1 &\bf{0}\\
					& \mathcal{Z}_0 &{\bf{0}}\\
					\mathcal{W}_0 & \mathcal{W}_1 &1
				\end{pmatrix},
			\end{equation}
			where 
			\[
			\mathcal{Z}_0 = \begin{psmallmatrix}
				\Iden_{r/\bss} &&&&\\
				& \ddots &&&\\
				&&(Z_0) &&\\
				&&&\ddots&\\
				&&&&\Iden_{r/\bss}
			\end{psmallmatrix},
			\quad
			\mathcal{Z}_1 = \begin{psmallmatrix}
				{\bf{0}} &&&&\\
				& \ddots &&&\\
				&&(Z_1) &&\\
				&&&\ddots&\\
				&&&&{\bf{0}} 
			\end{psmallmatrix},
			\]
			\noindent such that $(Z_0)$ and $(Z_1)$ are the coordinates of $\GL_{r/\bss}$ and $\Mat_{r/\bss}$ respectively. For each $h = 0,1$, we define $(Z_h)$ to be the $r/\bss \times r/\bss$ block $((Z_h)_{a,b})$ for $1 \leq a,b \leq r/\bss$ and $\mathcal{W}_h := (0, \dots, 0,(W_h),0, \dots 0)$, where we set $(W_h):= (W_{h,1}, \dots, W_{h,r/\bss})$. For $1\leq u, v \leq r/\bss$, we define the following one-dimensional subgroups of $\overline{\Gamma}_{\rP_1 M_\rho}$: 
			\begin{equation}\label{E:onedimgroup}
				T_{uv} := \left\{\begin{pmatrix}
					\mathcal{B}_{uv} & {\bf{0}}\\
					{\bf{0}} & \mathcal{B}_{uv}
				\end{pmatrix}\right\},
				\quad  U_{uv} := \left\{\begin{pmatrix}
					\Iden_r &\mathcal{C}_{uv}\\
					{\bf{0}}& \Iden_r
				\end{pmatrix}\right\},
			\end{equation}
			where we set
			\begin{equation}\label{E:mathcalBC}
				\mathcal{B}_{uv} := \begin{psmallmatrix}
					\Iden_{r/\bss} &&&&\\
					& \ddots &&&\\
					&&B_{uv} &&\\
					&&&\ddots &\\
					&&&&\Iden_{r/\bss}
				\end{psmallmatrix}, \quad \mathcal{C}_{uv} := \begin{psmallmatrix}
					{\bf{0}} &&&&\\
					& \ddots &&&\\
					&&C_{uv} &&\\
					&&&\ddots &\\
					&&&&{\bf{0}}
				\end{psmallmatrix}
			\end{equation}
			such that 
			\begin{equation*}\label{E:subonedimengrp}
				B_{vv} := \left\{ \begin{psmallmatrix}
					1 &&&& \\
					&\ddots&&&\\
					&&*&&\\
					&&&\ddots & \\
					&&& &1
				\end{psmallmatrix} \right\}, \,
				B_{uv} := \left\{ \begin{psmallmatrix}
					1&0&\dots&0 \\
					0&\ddots&*&\vdots \\
					\vdots&\ddots&\ddots&0  \\
					0&\dots&0  &1
				\end{psmallmatrix} \right\},
			\, \text{and} \, 
				C_{uv} := \left\{ \begin{psmallmatrix}
					0 &&&& \\
					&&\ddots&*&\\
					&&&\ddots & \\
					&&& &0
				\end{psmallmatrix} \right\},
			\end{equation*}
			where $*$ in $B_{uv}$ and $C_{uv}$ are in the $(u,v)$-coordinates. Note that the Lie algebras of the $2 \cdot r^2/\bss^2$ algebraic groups $T_{uv}$ and $U_{uv}$ span $\Lie (\overline{\Gamma}_{\rP_1 M_\rho}/\overline{\bk})_i$. In what follows, we construct one dimensional algebraic subgroups $T_{uv}'$ and $U'_{uv}$ of $\overline{H}_{1,i}$ so that $T_{uv}' \cong T_{uv}$ and $U_{uv}' \cong U_{uv}$.
			Then, since $\Lie(\cdot)$ is a left exact functor, it follows that $\Lie T_{uv}' \cong \Lie T_{uv}$  and $\Lie U_{uv}' \cong \Lie U_{uv}$, and so $\rd\overline{{\pi}}_{1,i}$ is surjective. 
			Since $\overline{Q}_{1,i}$ is a zero dimensional vector group, $\overline{\pi}_{1,i}$ is injective on points and so it follows by checking directly that 
			\begin{itemize}
				\item for $w \neq v$, all $W_{0,w}$ and $W_{1,w}$ coordinates of $\overline{{\pi}}_{1,i}^{-1}(T_{uv})$ are zero;
				\item all $(W_0)$ coordinates of $\overline{{\pi}}_{1,i}^{-1}(U_{uv})$ are zero, and for $w \neq  v$, all $W_{1,w}$ coordinates of $\overline{{\pi}}_{1,i}^{-1}(U_{uv})$ are zero.
			\end{itemize}
			\noindent To construct $T_{vv}'$, we let $a_{v} \in \overline{\bk}^\times\setminus \overline{\FF_q}^\times$ and pick element  $\gamma_{1,v} \in \overline{H}_{1,i}(\overline{\bk})$ so that 
			\begin{equation}\label{E:elementspecial1}
				\overline{{\pi}}_{1,i}(\gamma_{1,v}) = \begin{pmatrix} 
					\mathfrak{a}_v &\\
					&\mathfrak{a}_v
				\end{pmatrix},\enskip \textup{where} \enskip \mathfrak{a}_v = \begin{psmallmatrix}
					1 &&&&\\
					&\ddots &&&\\
					&&a_v&&\\
					&&&\ddots&\\
					&&&&1
				\end{psmallmatrix} \in (\GL_{r/\bss}(\overline{\bk}))_i,\end{equation}
			\noindent such that $a_v$ is in the $(i\cdot r/\bss +v)$-th diagonal entry of $\mathfrak{a}_v$. For $1\leq v \leq r/\bss$, we let $c_{0,v}$ and $c_{1,v}$ respectively be the $(2r+1,(i-1)\cdot r/\bss+v)$-th and the $(2r+1,\left(r+(i-1)\cdot r/\bss\right)+v)$-th the entry of $\gamma_{1,v}$. Let $T_{vv}'$ be the Zariski closure of the subgroup of $\overline{H}_{1,i}$ generated by $\gamma_{1,v}$, for each $1 \leq v \leq r/\bss$. Then, one checks directly that the defining equations of the one dimensional subgroup $T_{vv}'$ of $\overline{H}_{1,i}$ can be written as follows:
			\[ \begin{cases} 
				(a_{v}-1)W_{0,v}-c_{0,v}((Z_0)_{v,v}-1)=0, \enskip  1\leq v \leq r^2/\bss,\\
				(Z_0)_{w,w}=1, \enskip  w \neq v, \enskip 1\leq v \leq r^2/\bss,\\
				(Z_1)_{u,v}=0,  \enskip  1\leq u, v \leq r/\bss, \\ 
				W_{h,w} =0 \enskip  w \neq v; \enskip h=0,1, \enskip 1\leq v \leq r^2/\bss,\\
				W_{0,v} \cdot c_{1,v}- W_{1,v}\cdot c_{0,v}=0, \enskip 1\leq v \leq r^2/\bss.
			\end{cases}
			\]
			Then, we see that $T'_{vv} \cong T_{vv}$ via $\overline{\pi}_{1,i}$. To construct $T_{uv}'$ when $u \neq v$, we let $b_{u,v} \in T_{uv}(\bk)$ be a $\bk$-rational basis for the one dimensional vector group $T_{uv}$ and pick $b'_{u,v} \in \overline{H}_{1,i}(\overline{\bk})$ so that $\overline{\pi}_{1,i}(b'_{u,v}) = b_{u,v}$. We define $T_{uv}'$ to be the one dimensional vector group in $\overline{H}_{1,i}$ via the conjugations
			\[
			\eta_{v}^{-1}b'_{uv}\eta_{v}, \quad \text{for} \, \, \eta_v \in T'_{vv}, \, \, \, v=1,\dots,r/\bss.
			\]
			Then, we have $T'_{uv} \cong T_{uv}$ via $\overline{\pi}_{1,i}$. Similarly, we use the methods used for $T_{vv}'$ and conjugations as above to construct suitable one dimensional $U_{uv}'$ such that $U_{uv}'\cong U_{uv}$ for $1 \leq u, v \leq r/\bss$. The arguments are essentially the same as the ones used to construct $T_{vv}'$ and $T_{uv}'$, and so we omit the details and leave it to the reader. This proves our claim.
		\end{proof}
  
   \begin{claim3}\label{CL:Claim00}
   For $\overline{Q}_{1,i}$ if all entries of the $\bv_{i}$ coordinate vector are zero and $\dim \overline{Q}_{1,i}=r/\bss$, then $\rd \overline{{\pi}}_{1,i}: \Lie \overline{H}_{1,i} \rightarrow \Lie (\overline{\Gamma}_{\rP_1M_\rho}/\overline{\bk})_i$ is surjective.
\end{claim3}
   \begin{proof}[Proof of Claim~\ref{CL:Claim00}]
We have $\dim \overline{H}_{1,i}= 2r^2/\bss^2+r/\bss$  and by \eqref{E:LQ1i}, 
\[
\overline{Q}_{1,i}=\left\{ \begin{pmatrix}
\Iden_r & {\bf 0} & {\bf{0}}\\
{\bf{0}} &\Iden_r & {\bf{0}}\\
{\bf{0}} & \bz &1 \\ 
\end{pmatrix}  \, :  \, \bz = ({\bf{0}}, \dots, {\bf{0}},\bv_{\bss+i}, {\bf{0}},\dots, {\bf 0}) \in \GG_a^{r} \text{ where } \bv_{\bss+i} \in \GG_a^{r/s}
				\right\}.
\]
Note that $\overline{\Gamma}_{N_0}$ is an algebraic subgroup of $\GL_{r+1}/\overline{\bk}$ such that the surjective map $\overline{\Xi}_0: \overline{\Gamma}_{N_1} \rightarrow \overline{\Gamma}_{N_0}$ induced by $\Xi_0$ in \eqref{E:mapNtoN_0} is given by 
\begin{equation}\label{E:mapNntoN}
\begin{pmatrix}
\gamma_0 & \gamma_1 & {\bf 0}\\
{\bf 0} &\gamma_0 & {\bf 0}\\
\bz_0 & \bz_1 &1 \\ 
\end{pmatrix} \mapsto \begin{pmatrix} \gamma_0 & {\bf 0} \\ \bz_0 & 1\end{pmatrix}.
\end{equation}
Then, the elements of $\Ker \Xi_0^{(\overline{\bk})} \subseteq \overline{\Gamma}_{N_1}(\overline{\bk})$ are of the form 
\[
\begin{pmatrix}
\Iden_r & \gamma_1 & {\bf 0}\\
{\bf 0} &\Iden_r & {\bf 0}\\
{\bf 0} & \bz_1 &1 \\ 
\end{pmatrix}.
\]
From this, we see that for any $\bbb_i \in \Mat_{r/\bss}(\overline{\bk})$, elements of the form $\overline{\bbb}_i'$ in \eqref{E:Matrixbprime} with $\bw _{i}={\bf 0}$ are in $\overline{H}_{1,i}(\overline{\bk})$. Multiplying such $\overline{\bbb}_i'$ by suitable elements of $\overline{Q}_{1,i}(\overline{\bk})$, we have $\overline{\bbb}_i'$ of the form \eqref{E:Matrixbprime} where $\bw_i=\bw_{\bss+i}={\bf 0}$ in $\overline{H}_{1,i}(\overline{\bk})$. Let $\overline{P}_{1,i}$ be the Zariski closure inside $\overline{H}_{1,i}$ of the subgroup generated by all such $\overline{\bbb}_i'$ with $\bbb_i$ running over all elements of $\Mat_{r/\bss}(\overline{\bk})$. Then, clearly $\overline{P}_{1,i} \cong \Mat_{r/\bss}/\overline{\bk}$.  

For any $\xi_i \in \GL_{r/\bss}(\overline{\bk})$, multiplying the elements $\overline{\xi}_i' \in \overline{H}_{1,i}(\overline{\bk})$ of the form \eqref{E:MatrixGammaprime} by suitable elements of $Q_{1,i}(\overline{\bk})$, we obtain $\overline{\xi}_i'$ where $\bu_{\bss+i}={\bf 0}$. 

For all $\xi_i \in \GL_{r/\bss}(\overline{\bk})$, if there is a $\overline{\xi}_i'\in \overline{H}_{1,i}(\overline{\bk})$ with
$\bu _{i}={\bf 0}$,
then by using $\overline{P}_{1,i}$ and all such element $\overline{\xi}_i'$ for all $\xi_i \in \GL_{r/\bss}(\overline{\bk})$, we could simply construct $\overline{S}_{1,i}$ as in \eqref{E:Xsplitsmoothlog} and restrict $\rd \overline{\pi}_1$ to $\Lie \overline{S}_{1,i}$ to obtain a surjection onto $\Lie (\overline{\Gamma}_{\rP_1 M_\rho}/\overline{\bk})_i$.

Next suppose $\bu_i\neq {\bf 0}$. Consider the short exact sequence of linear algebraic groups (see \eqref{E:overlineSES1})
		\begin{equation*}
			1\rightarrow \overline{X}_0\rightarrow \overline{\Gamma}_{N_0} \xrightarrow{\overline{\pi}_0} \overline{\Gamma}_{M_\rho} \rightarrow 1.
		\end{equation*}
 Consider the one dimensional subgroups of $\overline{\Gamma}_{M_{\rho}}$ of the form $\mathcal{B}_{uv} \in (\GL_{r/\bss}/\overline{\bk})_i$ given in \eqref{E:mathcalBC}  for $1\leq u, v\leq r/\bss$. 
 The same methods used in Claim~\ref{CL:Claim2} to construct $T_{uv}'$ can be applied in a straightforward manner to construct one dimensional subgroups 
$\mathcal{B}_{uv}'$ of $\overline{\Gamma}_{N_0}$ so that $\mathcal{B}_{uv}' \cong \mathcal{B}_{uv}$. We leave this 
to the reader. For $\xi_i \in \GL_{r/\bss}(\overline{\bk})$, consider $\overline{\xi}_i' \in \overline{H}_{1,i}(\overline{\bk})$ of the form \eqref{E:MatrixGammaprime} with $\bu_{\bss+i}={\bf 0}$. Let $\overline{V}_{1,i}$ be the Zariski closure inside $\overline{H}_{1,i}$ of the subgroup generated by all such $\overline{\xi}_i'$ with $\xi_i$ running over all elements of $\GL_{r/\bss}(\overline{\bk})$. 
 Then, we can identify $\nu \in \overline{V}_{1,i}(\overline{\bk})$ with the image $\smash{\overline{\Xi}_0^{(\overline{\bk})}}(\nu) \in \overline{\Gamma}_{N_0}(\overline{\bk})$ where $\overline{\Xi}_0$ is the surjective map \eqref{E:mapNntoN}. Via this identification, each $\mathcal{B}_{uv}'$ for $1\leq u,v\leq r/\bss$ is a one dimensional subgroup of $\overline{V}_{1,i}$. The Lie algebras of the $r^2/\bss^2$ subgroups $\mathcal{B}_{uv}$ span $\Lie \GL_{r/\bss}/\overline{\bk}$. Thus, since $\overline{P}_{1,i} \cong \Mat_{r/\bss}/\overline{\bk}$, the Lie algebras of each $\mathcal{B}_{uv}$ and $\Lie \overline{P}_{1,i}$ span $\Lie(\overline{\Gamma}_{\rP_1 M_\rho}/\overline{\bk})_i$ by \eqref{E:GammaP1i}.
			Then, since $\Lie(\cdot)$ is a left exact, $\rd\overline{{\pi}}_{1,i}$ is surjective.
  \end{proof}
 
As we vary all $1\leq i\leq \bss$, the surjection of $\rd \overline{{\pi}}_{1}$ follows. Thus, for $n=1$ the proof of the lemma is complete. \par
		
		Now suppose $n >1$. We follow the methods used for $n =1$ to prove that the induced tangent map $\rd\overline{\pi}_n$ at the identity is surjective onto $\Lie\overline{\Gamma}_{\rP_n M_\rho}$. Recall $\overline{\Gamma}_{\rP_n M_\rho}$ from \eqref{E:Gammasepn}. Let $w=1$ and consider the short exact sequence \eqref{E:overlineSES1} of linear algebraic groups. Fix $1\leq i \leq \bss$. We follow the methods used for the construction of $\overline{H}_{1,i}$ above to construct the Zariski closure $\overline{H}_{n,i}$ inside $\overline{\Gamma}_{N_n}$ of the subgroup generated by suitably chosen elements of $\overline{\Gamma}_{N_n}$ such that $\overline{H}_{n,i}$ is contained in the $(n+1) r^2/\bss^2+(n+1) r^2/\bss$ dimensional group
		\begin{equation*}\label{E:GammasepellH}
			G_{n,i} := \left\{ 
			\begin{pmatrix}
				\begin{array}{c|c|c|c|c} 
					\eta_0 & \eta_1 & \dots & \eta_n&\bf{0}\\
					\hline 
					& \eta_0 &\ddots &\vdots&\vdots\\
					\hline 
					&&\ddots&\eta_1&\vdots\\
					\hline 
					&&&\eta_0&\bf{0}\\
					\hline
					\bs_0 & \bs_1 &\dots & \bs_n&1
				\end{array} 
			\end{pmatrix}:
			\begin{aligned}&\eta_0 \in (\GL_{r/\bss}/\overline{\bk})_i, \, \eta_j \in (\Mat_{r/\bss}/\overline{\bk})_i,   1 \leq j \leq n \\
				& \bs_h = ({\bf{0}}, \dots,{\bf{0}}, \bs_{h,i},{\bf{0}},\dots, {\bf{0}}), \quad \bs_{h,i} \in \GG_a^{r/\bss} \\
				& \, \, \text{for each} \, \, 0 \leq h \leq n
			\end{aligned}
			\right\}. 
		\end{equation*}
		Let 
		\begin{equation*}
			(\overline{\Gamma}_{\rP_n M_\rho}/\overline{\bk})_i := \left\{ 
			\begin{pmatrix}
				\begin{array}{c|c|c|c} 
					\gamma_0 & \gamma_1 & \dots & \gamma_n\\
					\hline 
					& \gamma_0 &\ddots &\vdots\\
					\hline 
					&&\ddots&\gamma_1\\
					\hline 
					&&&\gamma_0
				\end{array} 
			\end{pmatrix}: \begin{aligned} &\gamma_0 \in (\GL_{r/\bss}/\overline{\bk})_i, \, \gamma_j \in (\Mat_{r/\bss}/\overline{\bk})_i,\\
				&\, \, \text{where} \, \, 1 \leq j \leq n
			\end{aligned}
			\right\}. 
		\end{equation*}
		If $\dim \overline{H}_{n,i}= (n+1)\cdot r^2/\bss^2 + (n+1) \cdot r/\bss$, similar to $\overline{S}_{1,i}$ in \eqref{E:Xsplitsmoothlog} we simply construct
		\begin{equation}\label{E:Sni}
		\overline{S}_{n,i} = \left\{
		\begin{pmatrix}
			\vartheta_i & {\bf{0}}\\
			{\bf{0}} & 1
		\end{pmatrix} : \vartheta_i \in (\Gamma_{\rP_n M_\rho}(\overline{\bk}))_i
		\right\},
		\end{equation}
		and restrict $\rd \overline{\pi}_{n}$ to $\Lie \overline{S}_{n,i}$ to obtain a surjection onto $\Lie (\overline{\Gamma}_{\rP_n M_\rho}/\overline{\bk})_i$. As we vary all $1\leq i\leq \bss$, the surjection of $\rd \overline{{\pi}}_{n}$ follows. \par 
  
		Next, suppose $\dim \overline{H}_{n,i}< (n+1)\cdot r^2/\bss^2 + (n+1)\cdot r/\bss$. Then, via $\overline{{\pi}}_n$ we have a short exact sequence,
		\[
		1 \rightarrow \overline{Q}_{n,i} \rightarrow \overline{H}_{n,i} \xrightarrow{{\overline{\pi}}_{n,i}} (\overline{\Gamma}_{\rP_n M_\rho}/\overline{\bk})_i \rightarrow 1. 
		\] 
For $\overline{Q}_{n.i}$ if some entry of each $\bs_{h,i}$ coordinate vector where $0\leq h\leq n-\ell$ is non-zero or $\dim \overline{Q}_{n,i}\neq\ell r/\bss$, then the methods used in Claim~\ref{CL:CLaim1} to prove $\dim \overline{Q}_{1,i}=0$ can be applied in a straightforward manner to prove $\dim \overline{Q}_{n,i}=0$, which we leave to the reader. 

		\begin{claim3}\label{CL:Claim3}
			If $\dim \overline{Q}_{n,i}=0$, then $\rd \overline{{\pi}}_{n,i}: \Lie \overline{H}_{n,i} \rightarrow \Lie (\overline{\Gamma}_{\rP_nM_\rho}/\overline{\bk})_i$ is surjective.
		\end{claim3}
		
		\begin{proof}[Proof of Claim~\ref{CL:Claim3}]
			The proof follows the same line of argument as in the proof of Claim~\ref{CL:Claim2} ($n=1$ case) and so we include only a sketch. Similar to the coordinates $\bZ_1$ of $\overline{H}_{1,i}$ in \eqref{E:coordinatesX1}, we let the coordinates of $\overline{H}_{n,i}$ be as follows:
			\begin{equation*}
				\bZ_n = 
				\begin{pmatrix}
					\mathcal{Z}_0 & \mathcal{Z}_1 & \dots & \mathcal{Z}_n&\bf{0}\\
					& \mathcal{Z}_0 &\ddots &\vdots&\vdots\\
					&&\ddots&\mathcal{Z}_1&\vdots\\
					&&&\mathcal{Z}_0&\bf{0}\\
					\mathcal{W}_0 & \mathcal{W}_1 &\dots & \mathcal{W}_n&1
				\end{pmatrix},
			\end{equation*}
			where 
			\[
			\mathcal{Z}_0 = \begin{psmallmatrix}
				\Iden_{r/\bss} &&&&\\
				& \ddots &&&\\
				&&(Z_0) &&\\
				&&&\ddots&\\
				&&&&\Iden_{r/\bss}
			\end{psmallmatrix},
			\quad
			\mathcal{Z}_j = \begin{psmallmatrix}
				{\bf{0}} &&&&\\
				& \ddots &&&\\
				&&(Z_0) &&\\
				&&&\ddots&\\
				&&&&{\bf{0}} 
			\end{psmallmatrix},
			\]
			for each $1 \leq j \leq n$ such that $(Z_0)$ is as in \eqref{E:coordinatesX1} and $(Z_j)$ is the $r/\bss \times r/\bss$ block $((Z_j)_{a,b})$ for $1~\leq~a,b \leq~r/\bss$. Moreover, set $\mathcal{W}_h := (0, \dots, 0,(W_h),0, \dots 0)$, where we set $(W_h):= (W_{h,1}, \dots, W_{h,r/\bss})$ for each $0 \leq h \leq n$. \par 
			
			Now, we prove that $\rd \overline{{\pi}}_{n,i}: \Lie \overline{H}_{n,i} \rightarrow \Lie (\overline{\Gamma}_{\rP_nM_\rho}/\overline{\bk})_i$ is surjective. Similar to \eqref{E:onedimgroup}, we construct one-dimensional subgroups of $\overline{\Gamma}_{\rP_n M_\rho}$:
			\[
			T_{0,u,v} := \left\{\begin{psmallmatrix}
				\mathcal{B}_{uv} & {\bf{0}} & \dots & {\bf{0}}\\
				& \ddots &\ddots &\vdots \\
				&&\ddots&\vdots \\
				&&& \mathcal{B}_{uv}
			\end{psmallmatrix}\right\},
			\quad 
			U_{\ell,u,v} := \left\{\begin{psmallmatrix}
				\Iden_{r} & {\bf{0}} & \dots &\mathcal{C}_{uv}& \dots& {\bf{0}}\\
				& \ddots &\ddots &&&\vdots \\
				&&\ddots&\ddots&\ddots&\mathcal{C}_{uv}\\
				&&&\ddots&\ddots&\vdots \\
				&&& &\ddots&{\bf{0}}\\
				&&& &&\Iden_{r}
			\end{psmallmatrix}
			\right\}, \]
			such that $\mathcal{B}_{uv}$ and $\mathcal{C}_{uv}$ are as in \eqref{E:mathcalBC}, and for $1\leq \ell \leq n$ , $\mathcal{C}_{uv}$ is in the $\ell$-th superdiagonal block of $U_{\ell,u,v}$. Similar to $n=1$ case, note that the Lie algebras of the $(n+1) \cdot r^2/\bss^2$ algebraic groups $T_{0,u,v}$ and $U_{\ell,u,v}$ span $\Lie (\overline{\Gamma}_{\rP_n M_\rho}/\overline{\bk})_i$. In what follows, we construct one dimensional algebraic subgroups $T_{0,u,v}'$ and $U'_{\ell,u,v}$ of $\overline{H}_{n,i}$ so that $T_{0,u,v}' \cong T_{0,u,v}$ and $U_{\ell,u,v}' \cong U_{\ell,u,v}$. Then, since $\Lie(\cdot)$ is a left exact functor, it follows that $\Lie T_{0,u,v}' \cong \Lie T_{0,u,v}$ and $\Lie U_{\ell,u,v}' \cong \Lie U_{\ell,u,v}$, and so $\rd\overline{{\pi}}_{n,i}$ is surjective. Since $\overline{Q}_{n,i}$ is a zero dimensional vector group, $\overline{\pi}_{n,i}$ is injective on points and so it follows by checking directly that 
			\begin{itemize}
				\item for $w \neq v$ and $0 \leq h \leq n$, all $W_{h,w}$ coordinates of $\overline{{\pi}}_{n,i}^{-1}(T_{0,u,v})$ are zero;
				\item all $(W_0)$ coordinates of $\overline{{\pi}}_{n,i}^{-1}(U_{\ell,u,v})$ are zero, and for $w \neq  v$ and $1 \leq j \leq n$, all $W_{j,w}$ coordinates of $\overline{{\pi}}_{n,i}^{-1}(U_{\ell,u,v})$ are zero.
			\end{itemize}
			To construct $T_{0,v,v}'$, we let $a_{v} \in \overline{\bk}^\times\setminus \overline{\FF_q}^\times$ and pick elements  $\gamma_{n,v} \in \overline{H}_{n,i}(\overline{\bk})$ so that 
			\[
			\overline{{\pi}}_{n,i}(\gamma_{n,v}) =
			\begin{psmallmatrix} 
				\mathfrak{a}_v & & \\
				&\ddots &\\
				&&\mathfrak{a}_v
			\end{psmallmatrix},\]
			where $\mathfrak{a}_v$ is as in \eqref{E:elementspecial1}. For $1\leq v \leq r/\bss$ and $0\leq h \leq n$, we let $c_{h,v}$ be the $(n r+1,hr+(i-1)\cdot r/\bss+v)$-th the entry of $\gamma_{n,v}$. Let $T_{0,v,v}'$ be the Zariski closure of the subgroup of $\overline{H}_{n,i}$ generated by $\gamma_{n,v}$. Then, one checks directly  that the defining equations of the one dimensional subgroup $T_{0,v,v}'$ of $\overline{H}_{n,i}$ can be written as follows:
			\[ \begin{cases} 
				(a_{v}-1)W_{0,v}-c_{0,v}((Z_0)_{v,v}-1)=0, \enskip 1\leq v \leq r^2/\bss,\\
				(Z_0)_{w,w}=1, \enskip w \neq v, \enskip 1\leq v \leq r^2/\bss,\\
				(Z_j)_{u,v}=0,  \enskip 1 \leq j \leq \ell, \enskip 1\leq u, v \leq r/\bss,\\ 
				W_{h,w} =0, \enskip w \neq v; \enskip 0 \leq h \leq \ell, \enskip 1\leq v \leq r^2/\bss,\\
				W_{h_1,v} \cdot c_{h_2,v}- W_{h_2,v}\cdot c_{h_1,v}=0, \enskip 0 \leq h_1, h_2 \leq \ell, \enskip  1\leq v \leq r^2/\bss.
			\end{cases}
			\]
			Then, we see that $T'_{0,v,v} \cong T_{0,v,v}$ via $\overline{\pi}_{n,i}$. Similarly, we use the methods used for $T_{0,v,v}'$ and conjugations as in the $n=1$ case to construct $U_{\ell,u,v}'$ such that $U_{\ell,u,v}'\cong U_{\ell,u,v}$ for all $1 \leq u, v \leq r/\bss$, $1\leq \ell \leq n$, and $T_{0,u,v}'$ such that $T_{0,u,v}'\cong T_{0,u,v}$ for all $1 \leq u, v \leq r/\bss$, $u\neq v$. The arguments are essentially the same as the arguments used to construct $T_{uv}'$ and $U_{uv}'$ in the $n=1$ case and $T_{0,v,v}'$ above, and so we omit the details and leave it to the reader. This proves our claim.
		\end{proof}
  
   For $Q_{n,i}$ if $\bs_{h,i}={\bf{0}}$ for all $0\leq h\leq n-\ell$ and $\dim Q_{n,i}=\ell r/\bss$, then the  methods used in Claim~\ref{CL:Claim00} to prove that $\rd \overline{{\pi}}_{1,i}: \Lie \overline{H}_{1,i} \rightarrow \Lie (\overline{\Gamma}_{\rP_1M_\rho}/\overline{\bk})_i$  is surjective can be applied in a straightforward manner to prove that $\rd \overline{{\pi}}_{n,i}: \Lie \overline{H}_{n,i} \rightarrow \Lie (\overline{\Gamma}_{\rP_n M_\rho}/\overline{\bk})_i$ is surjective, which we leave to the reader.

		As we vary all $1\leq i\leq \bss$ the surjection of $\rd \overline{{\pi}}_{n}$ follows. Thus, for $n>1$ the proof of the lemma is complete.
	\end{proof}

	\subsection{Algebraic independence of logarithms and quasi-logarithms}  
	In this subsection, we prove Theorem~\ref{T:Main02} (restated as Theorem~\ref{T:Main2}) and Corollary~\ref{C:Main02}. Recall the short exact sequence \eqref{E:SESlog}:
 \begin{equation*}
		1 \rightarrow X_n \rightarrow \Gamma_{N_n} \xrightarrow{\pi_n} \Gamma_{\rP_n M_\rho} \rightarrow 1.
	\end{equation*}
 We will first show that $X_n$ can be identified with a $\Gamma_{\rP_n M_\rho}$-submodule of $((\rP_n M_\rho)^B)^w$.  Let $\bn \in \Mat_{((n+1)rw +1)\times 1}(N_n)$ be such that its entries form a $\ok(t)$-basis of $N_n$ and $\sigma \bn = \Phi_{N_n} \bn$. The entries of $\Psi_{N_n}^{-1}\bn$ form a $\bk$-basis of $N_n^B$ \cite[Prop. 3.3.9]{P08}. If we write $\bn = [\bn_1, \dots, \bn_w, y]^\tr$ where each $\bn_i \in \Mat_{(n+1)r \times 1} (N_n)$, then the entries of $[\bn_1, \dots, \bn_w]^\tr$ form a $\ok(t)$-basis of  $(\rP_n M_\rho)^w$ and the entries of $\bu := [\Psi_{\rP_n M_\rho}^{-1} \bn_1 , \dots, \Psi_{\rP_n M_\rho}^{-1} \bn_w]^\tr$ form a $\bk$-basis of $((\rP_n M_\rho)^B)^w$. Given any $\bk$-algebra $\rR$, we recall the action of $\Gamma_{\rP_n M_\rho}(\rR)$ on $\rR~\otimes_\bk~((\rP_n M_\rho)^B)^w$ from \cite[\S4.5]{P08} (see also \eqref{E:action2}) as follows: for any $\mu \in \Gamma_{\rP_n M_\rho}(\rR)$ and any $\bv_h \in \Mat_{1\times (n+1)r}(\rR)$, $0 \leq h \leq n$, the action of $\mu$ on $(\bv_1, \dots, \bv_w) \cdot \bu \in R \otimes_{\bk} ((\rP_n M_\rho)^B)^w$~is 
	\begin{equation*}\label{E:action3}
		(\bv_1, \dots, \bv_w) \cdot \bu \mapsto (\bv_1 \mu^{-1}, \dots, \bv_w \mu^{-1}) \cdot \bu.
	\end{equation*}
	\noindent Thus, by \eqref{E:Actionlog} the action of $\Gamma_{\rP_n M_\rho}$ on $((\rP_n M_\rho)^B)^w$ is compatible with the action of $\Gamma_{\rP_n M_\rho}$ on $X_n$. Then, when we regard $((\rP_n M_\rho)^B)^w$ as a vector group over $\bk$, by Lemma~\ref{L:defined1} we get the desired result. \par
	
	Now, note that since $X_n$ is a $\Gamma_{\rP_n M_\rho}$-submodule of $((\rP_n M_\rho)^w)^B$, by the equivalence of categories $\cT_{\rP_n M_\rho} \approx \Rep(\Gamma_{\rP_n M_\rho}, \bk)$, there exists a sub-$t$-motive $V_n$ of $(\rP_n M_\rho)^w$ such that as $\Gamma_{\rP_nM_\rho}$-modules
	\begin{equation}\label{E:Tmotiveandbetti}
		X_n \cong V_n^B.
	\end{equation}
	By \eqref{E:motivespro1}, we see that for any $n \geq 1$ and $0 \leq j \leq n-1$ we obtain a short exact sequence of $t$-motives
	\begin{equation}\label{E:motivespro1111}
		0 \rightarrow (\rP_j M_\rho)^w \xrightarrow{\iota} (\rP_{n}M_\rho)^w \xrightarrow{\boldsymbol{\rpr}_{w,{n-j-1}}} (\rP_{n-j-1}M_\rho)^w \rightarrow 0.
	\end{equation}

	\begin{lemma}\label{L:Equivvv}
		For $n \geq 1$, let $V_n$ be as in \eqref{E:Tmotiveandbetti}. Then, for $0 \leq j \leq n-1$ there is a surjective map of $t$-motives $\overline{\boldsymbol{\rpr}}_{w,{n-j-1}}: V_n \rightarrow V_{n-j-1}$ via the map $\boldsymbol{\rpr}_{w,{n-j-1}}$ in \eqref{E:motivespro1111}.
	\end{lemma}
	
	\begin{proof}
		We prove the result for $w = 1$. The following argument for $w = 1$ can be applied in
		a straightforward manner to prove the arbitrary $w$ case, which we leave to the reader. Let $w=1$. Recall from \eqref{E:NntoN} that for any $\bk$-algebra $\rR$, if $\nu_n = \begin{psmallmatrix}\mu_n & {\bf 0}\\ \bsw_n & 1\end{psmallmatrix} \in \Gamma_{N_n}(\rR)$, then
		\[
		\nu_{n-j-1} = \begin{pmatrix}\mu_{n-j-1} & {\bf 0}\\ \bsw_{n-j-1} & 1\end{pmatrix} \in \Gamma_{N_{n-j-1}}(\rR),
		\]
		where $\mu_{n-j-1}$ is the matrix formed by the $r(n-j)\times r(n-j)$ upper-left square of $\mu_n$ and $\bsw_{n-j-1} =(w_0, \dots, w_{n-j-1})$. 
		Also recall from \eqref{E:mapNtoN_0} that the surjective map of affine group schemes $\Gamma_{N_n} \twoheadrightarrow \Gamma_{N_{n-j-1}}$ is given by 
  \begin{equation*}\label{E:mapNtoN_01}
			\nu_n \mapsto \nu_{n-j-1}.
		\end{equation*}
		Since $X_n$ and $X_{n-j-1}$ are $\bk$-smooth by Lemma~\ref{L:defined1}, this map gives a surjective map of group schemes $X_n \rightarrow X_{n-j-1}$. By \eqref{E:Tmotiveandbetti}, this corresponds to a map of representations of $\Gamma_{\rP_n M_\rho}$ over $\bk$, $\overline{\boldsymbol{\rpr}}_{w,{n-j-1}}^B: V_n^B \rightarrow V_{n-j-1}^B$ via the map $\boldsymbol{\rpr}_{w,{n-j-1}}^B:((\rP_{n}M_\rho)^w)^B \rightarrow ((\rP_{n-j-1}M_\rho)^w)^B$, where $\boldsymbol{\rpr}_{w,{n-j-1}}$ is as in \eqref{E:motivespro1111}. By the equivalence of categories $\cT_{\rP_n M_\rho} \approx \Rep(\Gamma_{\rP_n M_\rho}, \bk)$, we obtain the required conclusion.
	\end{proof}
	
	\begin{theorem}\label{T:Main2}
		Let $\rho$ be a Drinfeld $\bA$-module of rank $r$ defined over $k^\sep$. Suppose that $K_\rho$ is separable over $k$ and $[K_{\rho}:k]=\bss$. Let $u_1, \dots, u_w \in \KK$ with $\Exp_\rho(u_i) = \alpha_i \in k^\sep$ for each $1\leq i \leq w$ and suppose that $\dim_{K_\rho} \Span_{K_\rho}(\lambda_1, \dots, \lambda_r, u_1, \dots, u_w) = r/\bss + w$. For $n \geq1$, let $N_n$ and $\Psi_{N_n}$ be defined as in \eqref{E:PhiPsiN}, and for each $1 \leq i \leq w$, let $Y_{i,n}:= Y_{u_i,n}$ be defined as in $\S$\ref{S:Ext}. Then, $\dim \Gamma_{N_n} = (n+1)\cdot r(r/\bss+w)$. In particular, 
		\[
		\trdeg_{\ok} \ok \bigg(\bigcup\limits_{s=0}^{n} \bigcup\limits_{i=1}^{r-1}\bigcup\limits_{m=1}^w \bigcup\limits_{j=1}^r \{\pd_\theta^s(\lambda_j), \pd_\theta^s(F_{\tau^i}(\lambda_j)),\pd_\theta^s(u_m), \pd_\theta^s(F_{\tau^i}(u_m))\}\bigg) = (n+1)(r^2/\bss
		+ rw).\]
	\end{theorem}
	
	\begin{proof}
		From the construction of $\Psi_{N_n}$, by Theorem~\ref{P:rigidhyper} we have
		\[
		\ok(\Psi_{N_n}|_{t=\theta})=  \ok \bigg(\bigcup\limits_{s=0}^{n} \bigcup\limits_{i=1}^{r-1}\bigcup\limits_{m=1}^w \bigcup\limits_{j=1}^r \{\pd_\theta^s(\lambda_j), \pd_\theta^s(F_{\tau^i}(\lambda_j)),\pd_\theta^s(u_m), \pd_\theta^s(F_{\tau^i}(u_m)\}\bigg), 
		\]
		and by Theorem~\ref{T:Tannakian} and Theorem~\ref{T:Main1}, we have
		\begin{equation*}\label{E:TrdegLogGamma}
			\dim \Gamma_{{N_n}} = \trdeg_{\ok} \ok(\Psi_{N_n}|_{t=\theta}) \leq (n+1)\frac{r^2}{\bss} + (n+1)rw.
		\end{equation*}
		Thus, we need to prove that $\dim X_n = (n+1)rw$, where $X_n$ is as in \eqref{E:SESlog}. By \eqref{E:Tmotiveandbetti} it suffices to show that $V_n^B \cong ((\rP_n M_\rho)^w)^B$. To prove this, we adapt the arguments of the proof of \cite[Thm.~5.1.5]{CP12} (see also \cite[Lem.~1.2]{Hardouin}). \par
		
		Note from \eqref{E:motivespro1111} that for $n \geq 1$ we have a short exact sequence of $t$-motives
		\[
		0 \rightarrow (\rP_0 M_\rho)^w \xrightarrow{\iota} (\rP_n M_\rho)^w \xrightarrow{\boldsymbol{\rpr}_{w,n-1}} (\rP_{n-1} M_\rho )^w\rightarrow 0.
		\]
		\noindent By Lemma~\ref{L:Equivvv}, there is a surjective map $\overline{\boldsymbol{\rpr}}_{w,n-1}: V_n \rightarrow V_{n-1}$ via ${\boldsymbol{\rpr}}_{w,n-1}$. Then $\ker (\overline{\boldsymbol{\rpr}}_{w,{n-1}})$ is a sub-$t$-motive of $M_\rho^w$. \par
		
		We claim that if $V_{n-1} \cong (\rP_{n-1}M_\rho)^w$, then $N_n/V_n$ is trivial in $\Ext_\cT^1(\mathbf{1}, \rP_n M_\rho/ V_n)$. Since $X_n \cong V_n^B$, we see that $\Gamma_{N_n}$ acts on $N_n^B/V_n^B$ through $\Gamma_{N_n}/X_n \cong  \Gamma_{\rP_n M_\rho}$ via \eqref{E:SESlog}. Since $\overline{\boldsymbol{\rpr}}_{w,n-1}$ is surjective onto $V_{n-1} \cong (\rP_{n-1}M_\rho)^w$, by using \eqref{E:ProSESLog} one finds that $N_n^B/V_n^B\cong N_0^B/(\ker \overline{\boldsymbol{\rpr}}_{w,n-1})^B$. Recall that for any $\bk$-algebra $\rR$, an element of $\Gamma_{\rP_nM_\rho}(\rR)$ is of the form \eqref{E:elementn} such that $\gamma_0$ is an element of $\Gamma_{M_\rho}(\rR)$. Then, \eqref{E:action22} shows the action of $\Gamma_{\rP_n M_\rho}$ on $N_n^B/ V_n^B$ is the same as the action of $\Gamma_{M_\rho}$ on it. Thus, $N_n^B/V_n^B$ is an extension of $\bk$ by $((\rP_n M_\rho)^w)^B/ V_n^B$ in $\Rep(\Gamma_{M_\rho}, \bk)$. By \cite[Cor.~3.5.7]{CP12} and the equivalence of categories $\cT_{M_\rho} \approx \Rep(\Gamma_{M_\rho}, \bk)$, we get the required conclusion of the claim. 
		
		Now, we prove the main result by induction. For $n =1$ case, suppose on the contrary that $V_1^B \subsetneq ((\rP_1 M_\rho)^w)^B$. From \cite[Thm.~5.1.5]{CP12}, we have $M_\rho^w\cong V_0$ and so, since $M_\rho^w \cong (\rP_0 M_\rho)^w$ we have $\ker (\overline{\boldsymbol{\rpr}}_{w,n-1}) \subsetneq M_\rho^w$. Since  $M_\rho^w$ is completely reducible in $\cT_{M_\rho}$ by \cite[Cor.~3.3.3]{CP12} and $\ker (\overline{\boldsymbol{\rpr}}_{w,n-1})$ is a sub-$t$-motive of $M_\rho^w$, there exists a non-trivial morphism $\phi_1~\in~\Hom_\cT(M_\rho^w, M_\rho)$ so that $\ker (\overline{\boldsymbol{\rpr}}_{w,n-1}) \subseteq \ker \phi_1$. Moreover, the morphism $\phi_1$ factors through the map ${M}_\rho^w/\ker (\overline{\boldsymbol{\rpr}}_{w,n-1}) \rightarrow {M}_\rho^w/\left(\ker \phi_1\right)$. Since $\phi_1 \in \Hom_\cT(M_\rho^w, M_\rho)$, there exist $e_{i,1}\in\bK_\rho$ not all zero such that $\phi_1\left(n_1, \dots, n_w\right)=\sum_{i=1}^w e_{i,1}(n_i)$. For a $\ok(t)$-basis $\bsm \in \Mat_{r\times 1} (M_{\rho})$ of $M_{\rho}$, suppose that $\rE_{i,1} \in \Mat_r(\ok(t))$ satisfies $e_{i,1}(\bsm) = \rE_{i,1} \bsm$. Set 
		\[
		\bE_{i,1} := \begin{pmatrix} \mathbf{0} & \rE_{i,1} \\
			&\mathbf{0}
		\end{pmatrix} \in \Mat_{2r}(\ok(t)).
		\]
		Recall from $\S$\ref{S:Pro} that $\bsD_1 \bsm$ forms a $\ok(t)$-basis of $\rP_1M_\rho$. By \eqref{E:Maprho} there exists $\be_{i,1} \in \End_\cT((\rP_1 M)^w)$ such that $\be_{i,1}(\bsD_1 \bsm) = \bE_{i,1}  \bsD_1 \bsm$. Let $\psi_1 \in \Hom_\cT((\rP_1M_\rho)^w, \rP_1M_\rho)$ such that $\psi_1(D_j n_1, \dots, D_j n_w) = \sum_{i=1}^w \be_{i,1}(D_j n_i)$ for each $j=0, 1$. We see that $\ker \psi_1/M_\rho^w \cong \ker \phi_1$ and $(\rP_1 M_\rho)^w/\ker \psi_1 \cong M_\rho^w/\ker \phi_1\cong M_\rho$. Then the pushout $\psi_{1*} N_1 := \be_{1,1*} Y_{1,1} + \dots + \be_{w,1*} Y_{w,1}$ is a quotient of $N_1/V_1$. By using the claim above, it follows that $\psi_{1*}N_1$ is trivial in $\Ext_\cT^1(\mathbf{1}, \rP_1 M_\rho)$. However, by Theorem~\ref{T:Trivial}, this is a contradiction. \par

		Now suppose that we have shown the result for $n-1$, that is, $V_{n-1} \cong (\rP_{n-1}M_\rho)^w$. Suppose  that $V_n^B \subsetneq ((\rP_n M_\rho)^w)^B$. Then, $\ker (\overline{\boldsymbol{\rpr}}_{w,n-1}) \subsetneq M_\rho^w$. Since $M_\rho^w$ is completely reducible in $\cT_{M_\rho}$ by \cite[Cor.~3.3.3]{CP12} and $\ker (\overline{\boldsymbol{\rpr}}_{w,n-1})$ is a sub-$t$-motive of $M_\rho^w$, there exists a non-trivial morphism $\phi_n \in \Hom_\cT(M_\rho^w, M_\rho)$ so that $\ker (\overline{\boldsymbol{\rpr}}_{w,n-1}) \subseteq \ker \phi_n$. Moreover, the morphism $\phi_n$ factors through the map
		${M}_\rho^w/\ker (\overline{\boldsymbol{\rpr}}_{w,n-1}) \rightarrow {M}_\rho^w/\left(\ker \phi_n\right)$. Since $\phi_n \in \Hom_\cT(M_\rho^w, M_\rho)$, we can write $\phi_n(n_1, \dots, n_w) = \sum_{i=1}^w e_{i,n}(n_i)$ for some $e_{1,n}, \dots, e_{w,n} \in \bK_\rho$ not all zero. Suppose that $e_{i,n}(\bsm) = \rE_{i,n} \bsm$ where $\rE_{i,n} \in \Mat_r(\ok(t))$. Set 
		\[
		\bE_{i,n} := \begin{psmallmatrix} \mathbf{0} &\dots &\mathbf{0}& \rE_{i,n} \\
			& \ddots & \ddots&\mathbf{0} \\
			&&\ddots&\vdots \\
			&&&\mathbf{0}
		\end{psmallmatrix} \in \Mat_{(n+1)r}(\ok(t)).
		\]
		Recall also from $\S$\ref{S:Pro} that $\bsD_n \bsm$ forms a $\ok(t)$-basis of $\rP_n M_\rho$. By \eqref{E:Maprho} there exists $\be_{i,n} \in \End_\cT((\rP_n M)^w)$ such that $\be_{i,n}(\bsD_1 \bsm) = \bE_{i,n}  \bsD_1 \bsm$. Let $\psi_n \in \Hom_\cT((\rP_n M_\rho)^w, \rP_n M_\rho)$ such that $\psi_1(D_j n_1, \dots, D_j n_w) = \sum_{i=1}^w \be_{i,1}(D_j n_i)$ for each $0 \leq j \leq n$. Similar to the base case, we see that $\ker \psi_n/(\rP_{n-1} M_\rho)^w \cong \ker \phi_n$ and $(\rP_n M_\rho)^w/\ker \psi_n \cong M_\rho^w/\ker \phi_n\cong M_\rho$. Then the pushout $\psi_{n*} N_n := \be_{1,n*} Y_{1,n} + \dots + \be_{w,n*} Y_{w,n}$ is a quotient of $N_n/V_n$. By using the claim above, it follows that $\psi_{n*}N_n$ is trivial in $\Ext_\cT^1(\mathbf{1}, \rP_n M_\rho)$. However, by Theorem~\ref{T:Trivial}, this is a contradiction. 
	\end{proof}
	
	\begin{proof}[Proof of Corollary~\ref{C:Main02}]
		Let $\{\eta_1, \dots, \eta_\alpha\} \subseteq \{\lambda_1, \dots, \lambda_r, u_1, \dots, u_w\}$ be a maximal $K_\rho$-linearly independent set containing $\{u_1, \dots, u_w\}$. Clearly, $r/\bss \leq \alpha \leq r/\bss+w$.  
		Since the quasi-periodic functions $F_{\delta}$ are linear in $\delta$ and satisfy the difference equation \eqref{E:Fdeltafneq}, we have 
		\begin{equation*}
			\ok \bigg( \bigcup\limits_{i=1}^{r-1}\bigcup\limits_{m=1}^w \bigcup\limits_{j=1}^r \{\lambda_j, F_{\tau^i}(\lambda_j), u_m, F_{\tau^i}(u_m)\}\bigg)\\
			= \ok \bigg(\bigcup_{j=1}^{r} \bigcup\limits_{m=1}^\alpha \left\{F_{\delta_{j}}(\eta_m)\right\}\bigg).
		\end{equation*}
		Moreover, for any $1 \leq i_1, i_2\leq  r$, $1 \leq j_1, j_2 \leq \alpha$, $0\leq s\leq n$ and $v_1, v_2 \in K_\rho$, by the product rule of hyperderivatives we obtain 
		\begin{equation*}
			\pd_\theta^s\left(v_1 F_{\delta_{i_1}}(\eta_{j_1}) + v_2F_{\delta_{i_2}}(\eta_{j_2})\right) = \sum\limits_{h=0}^s\bigg(\pd_\theta^{s-h}(v_1)\pd_\theta^h\left(F_{\delta_{i_1}}(\eta_{j_1})\right) + \pd_\theta^{s-h}(v_2)\pd_\theta^h\left(F_{\delta_{i_1}}(\eta_{j_2})
			\right)\bigg).
		\end{equation*}
		Thus, 
		\begin{equation*}
			\begin{split}
				\ok \bigg(\bigcup\limits_{s=0}^{n} \bigcup\limits_{i=1}^{r-1}\bigcup\limits_{m=1}^w \bigcup\limits_{j=1}^r \{\pd_\theta^s(\lambda_j), \pd_\theta^s(F_{\tau^i}(\lambda_j)),\pd_\theta^s(u_m), &\pd_\theta^s(F_{\tau^i}(u_m))\}\bigg)\\
				&= \ok \bigg(\bigcup_{s=0}^{n}\bigcup_{j=1}^{r} \bigcup\limits_{m=1}^\alpha \left\{\pd_\theta^s(F_{\delta_{j}}(\eta_m))\right\}\bigg).
			\end{split}
		\end{equation*}
Then, the result follows by Theorem~\ref{T:Main2}.
	\end{proof}
	
	\appendix
	\section{Differential Algebraic Geometry}\label{DAG}
	We present a few topics from differential algebraic geometry in positive characteristic \cite{Okugawa} (cf.~ \cite{HardouinSS} for characteristic zero). For the most part, we follow the terminology of \cite{HardouinSS}. Even though the proofs of most of the results presented here are covered in \cite{Okugawa}, we present them here nevertheless for completeness. 
	
	\subsection{Differential algebraic geometry in positive characteristic}\label{S: Kolchin+}
	Let $R$ be a commutative ring with unity of characteristic $p>0$. A \textit{differential ring} or \textit{$\pd$-ring} is a pair $(R, \pd)$, where $\pd$ represents a sequence of additive maps $\pd^j: R\rightarrow R$ that satisfy
	\begin{enumerate}
		\item $\pd^0(a)=a$,
		\item $\pd^j(a+b)=\pd^j(a)+\pd^j(b)$,
		\item $\pd^j(ab) = \sum_{i=0}^j \pd^i(a)\pd^{j-i}(b)$,
		\item $\pd^k\pd^j(a)=\binom{k+j}{j}\pd^{k+j}(a)$,
	\end{enumerate}
	for all $a, b \in R$ and $j, k\geq 0$. If $R$ is a field, then we say that $(R, \pd)$ is a \textit{differential field} or a \textit{$\pd$-field}. When the context is clear, we shall write $R$ instead of $(R,\pd)$. Moreover, a {\emph{$\pd$-morphism}} between two $\pd$-rings $R$ and $S$ is a morphism of rings that commute with $\pd$. For a $\pd$-ring $R$, if we let $\mathfrak{I} \subseteq R$ be an ideal, then $\mathfrak{I}$ is called a \textit{$\pd$-ideal} if $\pd^j(\mathfrak{I}) \subseteq \mathfrak{I}$ for all $j \geq 1$. If, in addition, $\mathfrak{I}$ is a radical (respectively prime) ideal of the $\pd$-ring $R$ regarded as a ring, then we say that $\mathfrak{I}$ is a {\emph{radical}} (respectively {\emph{prime) $\pd$-ideal}} of the $\pd$-ring $R$. For a set $\Sigma \subseteq R$, the intersection of all $\pd$-ideals containing $\Sigma$ is a $\pd$-ideal of $R$, which we denote by $\mathfrak{D}(\Sigma)$ and it is the smallest $\pd$-ideal of $R$ containing $\Sigma$. We see that $\mathfrak{D}(\Sigma)$ is the ideal, generated $\{\pd^j(a) \mid j\geq0, a \in \Sigma\}$, of the $\pd$-ring $R$ regarded as a ring. We denote by $\mathfrak{R}(\mathfrak{D}(\Sigma))$ or $\mathfrak{R}(\Sigma)$ the radical of $\mathfrak{D}(\Sigma)$ in the $\pd$-ring $R$. 
	
	\begin{proposition}[{Okugawa~\cite[p.45,~Thm.~5]{Okugawa}}]\label{charp}
		Let $R$ be a $\pd$-ring of characteristic $p >0$ and let $\mathfrak{I} \subseteq R$ be a $\pd$-ideal of $R$. Then, the radical $\mathfrak{R}(\mathfrak{I})$ is a $\pd$-ideal of $R$. 
	\end{proposition}
	\begin{proof}
		It suffices to prove that $\pd^j(\mathfrak{R}(\mathfrak{I})) \subseteq \mathfrak{R}(\mathfrak{I})$ for all $j \geq 1$. Let $a \in \mathfrak{R}(\mathfrak{I})$. Then $a^n \in \mathfrak{I}$ for some $n \geq 1$. For a sufficiently large $e \geq 1$, we see that $a^m\cdot a^n = a^{p^e}\in \mathfrak{I}$ for some $m \in \NN$. Note that Proposition~\ref{P:hyperderprod} applies here and so, for all $j\in \NN$ we see that $\pd^{jp^e}(a^{p^e}) = (\pd^j(a))^{p^e}$. Since $\mathfrak{I}$ is a $\pd$-ideal of $R$, we have $\pd^{jp^e}(a^{p^e}) \in \mathfrak{I}$ for all $j \geq 1$. Thus,  $(\pd^j(a))^{p^e} \in \mathfrak{I}$ and so $\pd^j(a) \in \mathfrak{R}(\mathfrak{I})$.  
	\end{proof}
	
	\begin{remark}
		The proof of Proposition~\ref{charp} does not work in characteristic $0$. See \cite[Prop.~2.19]{HardouinSS} for the proof of the characteristic $0$ case.
	\end{remark}
	
	The \textit{$\pd
		$-polynomial ring} denoted by $R\{y_1, \dots, y_m\}$ in the $\pd$-variables $(y_1, \dots, y_m)$ is the polynomial ring over a $\pd$-ring $R$ in the variables $\pd^j(y_i), \, j \geq 0, \,  1 \leq i \leq m$ made into a $\pd$-ring by setting 
	\begin{itemize}
		\item[(a)] $\pd^j(a) := \pd^j(a)$ for $a \in R$,
		\item[(b)] $ \pd^k(\pd^j(y_i)):= \binom{k+j}{j}\pd^{k+j}(y_i)$, $k \geq 0$. 
	\end{itemize}
	Here $y_1, \dots, y_m$ are called \emph{$\pd$-indeterminates}. \par

 Let $K$ be a $\pd$-field. A {\emph{$\pd$-extension field of $K$}} is a $\pd$-field $L$ which is an extension field of the $\pd$-field $K$. Note that $K$ and $L$ are fields. Let $\oK$ be an algebraic closure of the field $K$ and $K^\sep$ be the separable closure of $K$ in $\oK$. 
	
	\begin{proposition}
		There is a unique extension of $\pd^j:K\rightarrow K$ to $\pd^j:K^\sep\rightarrow K^\sep$, which satisfy all the rules of $\pd$.
	\end{proposition}
	\begin{proof}
		The proof follows the same argument as that for hyperderivatives \cite[Thm.~5]{Conrad00}. 
	\end{proof}
	Let $a \in \oK \setminus K^\sep$. We say that $\pd$ can be {\emph{extended to $a$}} if $\pd$ can be extended to some extension field of $K^\sep$ that contains $a$. The largest extension field $\oK^\pd$ of $K^\sep$ in $\oK$ that has an extension of $\pd$ is called the {\emph{$\pd$-closure}} of $K$ in $\oK$. \par

 For a set $X \subseteq (\oK^\pd)^m$, if we set 
	\[\mathfrak{I}(X):=\{P \in K\{y_1, \dots, y_m\} \mid P(a_1, \dots, a_m) = 0, \quad (a_1, \dots, a_m) \in X\},\] then $\mathfrak{I}(X)$ is a radical $\pd$-ideal in $R$, and we call it the \textit{defining $K$-$\pd$-ideal} of $X$.\par

	\begin{proposition}[{cf.~\cite[Prop.~3.8]{HardouinSS}}]
		Let $X_1, X_2 \subseteq (\oK^\pd)^m$. Then,
		\begin{enumerate}
			\item If $X_1 \subseteq X_2$, then $\mathfrak{I}(X_2) \subseteq \mathfrak{I}(X_1)$,
			\item $\mathfrak{I}(X_1\cup X_2)= \mathfrak{I}(X_1)\cap\mathfrak{I}(X_2)$.
		\end{enumerate}
	\end{proposition}
	\begin{proof}
		The proofs follow the same line of argument as that for the Zariski topology.
	\end{proof}

	Given a set $X \subseteq ((\oK^\pd)^m, \pd)$, we consider the Zariski closure $\overline{X}^Z\subseteq \oK^m$ of $X$, the closure of $X$ as a subset of $(\oK^\pd)^m$ equipped with the Zariski topology. Let $S \subseteq K[y_1, \dots, y_m]$ be a set of polynomials. The {\emph{zero set}} of $S$ is defined as
	\[ \mathcal{Z}(S):= \{(a_1, \dots, a_m)\in \oK^m \mid f(a_1, \dots, a_m) =0, \quad  f \in S\}.
	\]
	
	\begin{lemma} [{cf.~\cite[Lem.~3.42]{HardouinSS}}] \label{L:ZariskiKolchin}
		Let $X\subseteq (\oK^\pd)^m$ and let $\mathfrak{I}(X) \subseteq K\{y_1, \dots, y_m\}$ be its defining $K$-$\pd$-ideal. Also, let $K[y_1, \dots, y_m]$ be the usual polynomial ring in the variables $y_1, \dots, y_m$ over the field $K$. Then its Zariski closure is the set 
		\[\overline{X}^Z= \mathcal{Z}(\mathfrak{I}(X)\cap K[y_1, \dots, y_m]), \]
		where $\mathfrak{I}(X)\cap K[y_1, \dots, y_m] \subseteq K[y_1, \dots, y_m]$. 
	\end{lemma}
	
	\begin{proof}
		We follow the outline of the proof of \cite[Lem.~3.42]{HardouinSS}. Since $\mathcal{Z}(\mathfrak{I}(X) \cap K[y_1, \dots, y_m])$ is Zariski closed, it is straightforward to see that 
		\[X \subseteq \overline{X}^Z \subseteq \mathcal{Z}(\mathfrak{I}(X) \cap K[y_1, \dots, y_m]).\]
		Conversely, if $S \subseteq K[y_1, \dots, y_m] \subseteq K\{y_1, \dots, y_m\}$ is such that $S \subseteq \mathfrak{I}(X)$, then clearly we have $\mathfrak{R}(S) \subseteq \mathfrak{I}(X)$. This implies that $S\subseteq \mathfrak{R}(S) \cap K[y_1, \dots, y_m] \subseteq \mathfrak{I}(X) \cap K[y_1, \dots, y_m]$. Thus, $\mathcal{Z}(\mathfrak{I}(X) \cap K[y_1, \dots, y_m])  \subseteq \mathcal{Z}(S)$. Since $S$ was chosen arbitrarily, we see that $\mathcal{Z}(\mathfrak{I}(X) \cap K[y_1, \dots, y_m])  \subseteq \overline{X}^Z$.
	\end{proof}
	
	If $f \in K\{y_1, \dots, y_m\}$ is a $\pd$-polynomial given by a linear combination over the $\pd$-field $K$ of $1$ and elements of the set $\{\pd^j(y_i) \mid j \geq 0, 1 \leq i \leq m\}$, then we say that $f$ is a {\emph{degree one $\pd$-polynomial}} in $K\{y_1, \dots, y_m\}$. Moreover if the coefficient of $1$ is $0$, then we say that such $f$ is a {\emph{homogeneous degree one $\pd$-polynomial}}. 
	
	\begin{proposition}[{Okugawa~\cite[p.74~Thm.~5]{Okugawa}}]\label{P:raddiff}
		Let $S\subseteq K\{y_1, \dots, y_m\}$ be a set of degree one $\pd$-polynomials. Then, $\mathfrak{R}(S) = \mathfrak{D}(S)$.
	\end{proposition}
	\begin{proof}
		It suffices to show that $\mathfrak{D}(S)$ is a prime ideal of the $\pd$-ring $K\{y_1, \dots, y_m\}$ regarded as a usual ring. By definition $\mathfrak{D}(S)$ is generated, as an ideal of the ring $K\{y_1, \dots, y_m\}$, by $\{\pd^j(L_i) \mid i, j\geq 0, L_i \in S\}$. Suppose that $f, g \notin \mathfrak{D}(S)$  such that $fg \in \mathfrak{D}(S)$. Then, 
		\[
		fg= \sum\limits_{\substack{L_i\in S, \, j\geq1}}h_{i,\ell_j} \pd^{\ell_j}(L_i),
		\]
		\noindent where $\ell_j \geq 0$, and $h_{i,\ell_j}\in K\{y_1, \dots, y_m\}$, and all but finitely many $h_{i,\ell_j}$ are zero. We see that $fg$ is a polynomial in a finite subset of the variables $\{\pd^j(y_i) \mid j \geq 0, \, 1 \leq i \leq m\}$ over the $\pd$-field $K$ regarded as a usual field. Let us denote this subset of variables by $\{x_1, \dots, x_n\}$ for some $n \geq 1$. Then, $L = (\{\pd^{\ell_j}(L_i)\})$ is an ideal in $K[x_1, \dots, x_n]$ such that $f,g \notin L$ and $fg \in L$ and so, $L$ is not a prime ideal. However, for a polynomial ring in finitely many indeterminates, ideals generated by degree one polynomials are prime ideals and thus, we obtain a contradiction.
	\end{proof}

\end{document}